\documentclass[11pt,reqno,a4paper]{amsart}
 \usepackage[british]{babel}
 \usepackage{amsmath,amsfonts,mathrsfs,amsthm,amssymb}
 \usepackage{amssymb}
 \usepackage{graphicx}
 \usepackage{framed}
 \usepackage{color}
 \usepackage{hyperref}
 \usepackage{mathabx}
 \usepackage{mathtools}
 \usepackage{tikz}
 \usepackage{Dynkin}

 \usepackage{DT-macros}
 \def\graycolor{gray!60}

 \newcommand\bH{\mathbf{H}}
 \renewcommand\ss{{\rm ss}}
 \newcommand\hol{\mathfrak{hol}}

 \newcommand\bi{{\mathbf i}}
 \newcommand\bj{{\mathbf j}}
 \newcommand\bk{{\mathbf k}}
 
 \newenvironment{psm}
  {\left(\begin{smallmatrix}}
  {\end{smallmatrix}\right)}
 \newcommand\sym{\operatorname{sym}}
 \newcommand\ev{\operatorname{ev}}
 \newcommand\ffr{\mathfrak{f}_\bbR}
 \newcommand\CD{\widetilde{D}}
 \newcommand\CE{\widetilde{E}}
 \newcommand\CF{\widetilde{F}}
 \newcommand\inp[2]{(#1,#2)}
 
 \newcommand\Jac{\operatorname{Jac}}
 \newcommand\Gtwo[2]{\scalebox{0.75}{\Gdd{#1}{#2}}}
 \renewcommand\Aut{\operatorname{Aut}}
 \newcommand\bR{{\mathbf R}}
 \newcommand\GS{G_2^{\mbox{\tiny\rm split}}}

\newtheorem{theorem}{Theorem}[section]
\newtheorem{lemma}[theorem]{Lemma}
\newtheorem{cor}[theorem]{Corollary}
\newtheorem{prop}[theorem]{Proposition}
\theoremstyle{defn}
\newtheorem{defn}[theorem]{Definition}
\newtheorem{example}[theorem]{Example}

\theoremstyle{remark}

\numberwithin{equation}{section}

\title[A Cartan-theoretic classification of multiply-transitive $(2,3,5)$-distributions]
      {A Cartan-theoretic classification of\\ multiply-transitive $(2,3,5)$-distributions}

\author[D.~The]{Dennis~The}

\date{\today}

\address{Department of Mathematics \& Statistics, UiT The Arctic University of Norway, N-9037, Troms\o, Norway}

\email{dennis.the@uit.no}



\begin{document}
\maketitle

%%%%%%%%%%%%%%%%%%%%%%%%%%%%%%%%%%%%%%%%%%%%%%%%%%%%

 \begin{abstract}
 In his 1910 paper, \'Elie Cartan gave a tour-de-force solution to the (local) equivalence problem for generic rank 2 distributions on 5-manifolds, i.e.\ $(2,3,5)$-distributions.  From a modern perspective, these structures admit equivalent descriptions as (regular, normal) parabolic geometries modelled on a quotient of $G_2$, but this is not transparent from his article: indeed, the Cartan ``connection'' of 1910 is {\em not} a ``Cartan connection'' in the modern sense.  We revisit the classification of multiply-transitive $(2,3,5)$-distributions from a modern Cartan-geometric viewpoint, incorporating $G_2$ structure theory throughout, obtaining: (i) the complete (local) classifications in the complex and real settings, phrased ``Cartan-theoretically'', and (ii) the full curvature and infinitesimal holonomy of these models.  Moreover, we Cartan-theoretically prove exceptionality of the $3:1$ ratio for two 2-spheres rolling on each other without twisting or slipping, yielding a $(2,3,5)$-distribution with symmetry  the Lie algebra of the split real form of $G_2$.
 \end{abstract}
 
%%%%%%%%%%%%%%%%%%%%%%%%%%%%%%%%%%%%%%%%%%%%%%%%%%%%

 \section{Introduction}
 \label{S:intro}
 
 Our aim in this article is to present the symmetry classification of a classical geometric structure from a modern {\em Cartan-geometric} perspective.  The geometry of {\sl $(2,3,5)$-distributions} is an ideal setting for illustrating the utility of such techniques, owing to its complexity yet relative low-dimensionality, and its relevance because of connections to control theory, sub-Riemannian geometry, conformal geometry, special holonomy, and the geometry of PDE \cite{Yam1999, Nur2005, Zel2006, HS2009, Willse2014, SW2017a}.  It is a prototype geometry among the much broader class of {\sl parabolic geometries}, and the reader should bear in mind that the techniques illustrated here can be similarly applied to other geometries in this class, e.g.\ projective and conformal geometries.

 On a 5-manifold $M$, a {\sl $(2,3,5)$-distribution} is a subbundle $\cD \subset T M$  satisfying the following genericity conditions with respect to the Lie bracket: the derived distribution $[\cD,\cD]$ has rank 3, and then $[\cD,[\cD,\cD]]$ has rank 5.  Given two {\sl $(2,3,5)$-geometries} $(M,\cD)$ and $(\tilde{M},\tilde\cD)$, they are (locally) equivalent if there exists a (local) diffeomorphism $\phi : M \to \tilde{M}$ with $\phi_* \cD = \tilde\cD$, and an (infinitesimal) symmetry of $(M,\cD)$ is a vector field $X \in \fX(M)$ such that $\cL_X \cD \subset \cD$.  Letting $\sym(\cD)$ be the symmetry algebra, $(M,\cD)$ is {\sl transitive} (or {\sl homogeneous}) if the evaluation map $\ev_o : \sym(\cD) \to T_o M$ is surjective at all points $o \in M$.  In this case, it is {\sl multiply-transitive} if $\dim \sym(\cD) > 5$, and {\sl simply-transitive} if $\dim \sym(\cD) = 5$.  Our study is entirely local in nature, and we will focus on the multiply-transitive case in this article.
 
 The study of $(2,3,5)$-distributions has a long history, originating in 1910 with \'Elie Cartan's celebrated ``5-variables'' paper \cite{Car1910}.  Indeed, Cartan was studying certain classes of nonlinear second order PDE in the plane and arrived at $(2,3,5)$-distributions through remarkable correspondence theorems.  (We will not discuss these PDE aspects further here -- see \cite{Str2009, Yam1999, The2018} for recent reviews and extensions of this story.)  He then studied the local equivalence problem for $(2,3,5)$-geometries $(M,\cD)$ through his masterful use of what is now known as the {\sl Cartan equivalence method}.  Main results include:
 \begin{enumerate}
 \item He constructed a canonical coframing $\omega$ (``absolute parallelism'' / ``connection'') on a 14-dimensional bundle $\cG \to M$.  Thus, any $(2,3,5)$-geometry has at most 14-dimensional symmetry.
 \item There exists a fundamental curvature tensor for any $(2,3,5)$-geometry, which takes the form of a binary quartic, called the {\sl Cartan quartic}, and its root type is a pointwise algebraic invariant.  Label these by $\sfO$ (trivial), $\sfN$ (quadruple root), $\sfD$ (pair of double roots), $\mathsf{III}$ (triple root), $\mathsf{II}$ (double root and two simple roots), $\mathsf{I}$ (generic).  Type $\sfO$ is precisely the case of maximal symmetry.
 \item Augmenting the above types by their symmetry dimensions, Cartan classified (over $\bbC$) the $\sfO.14$, $\sfN.7$ and $\sfD.6$ branches, which was claimed to cover all multiply-transitive cases.  However, models in the $\sfN.6$ branch were discovered by Strazzullo \cite[Model 6.7.2]{Str2009} and Doubrov--Govorov \cite{DG2013}, and these are equivalent \cite[Remark 1]{DG2013}.  Although \cite[Remark 3]{DG2013} claimed completeness of the above list, to our knowledge no argument has previously appeared in print.
 \end{enumerate} 
 
 From a modern viewpoint, $(2,3,5)$-distributions are prototypical examples of {\sl parabolic geometries} \cite{CS2009}.  More precisely, there is a ``fundamental theorem of $(2,3,5)$-geometry'' asserting a categorical equivalence between $(2,3,5)$-distributions on $M$ and {\sl regular, normal} parabolic geometries $(\cG \to M, \omega)$ of type $(G_2,P)$, where $G_2$ is the 14-dimensional exceptional simple Lie group (take the split real form $\GS$ over $\bbR$), and $P \leq G_2$ is a certain (9-dimensional) parabolic subgroup.  Here, $\cG \to M$ is a $P$-principal bundle and $\omega$ is a {\sl Cartan connection}, i.e.\ a coframing that is in particular {\em equivariant} with respect to the structure group $P$.  However, a striking fact that we learned in late-2013 (and paraphrase here) is that:
 \begin{align} \label{E:Cartan}
 \framebox{{\bf The Cartan ``connection'' (of \cite{Car1910}) is {\em not} a ``Cartan connection''.}}
 \end{align}
 Indeed, Cartan's coframing is {\em not} fully $P$-equivariant, but this is easily missed by a modern reader of \cite{Car1910} since the full coframe structure equations were not stated.  (See further discussion in \cite[\S 14]{McK2022}.)  To our knowledge, \eqref{E:Cartan} was first observed by Pawel Nurowski, but even this is slightly obscured in the literature, so we now attempt to clarify.  Well-aware of Tanaka's work and modern developments on Cartan geometries, Nurowski found a (normalized) Cartan connection for $(2,3,5)$-geometries in \cite[Thm.8]{Nur2005}.  Nurowski (generously) attributed the structure equations to Cartan, but a comparison shows that although \cite[(50)]{Nur2005} and \cite[(5) in \S31]{Car1910} agree, \cite[(51)]{Nur2005} and \cite[(8) in \S32]{Car1910} are different.
 
 In 2013, we were present at a lunch discussion at the Australian National University, where Nurowski posed the question to Robert Bryant about modifying Cartan's coframing to obtain a Cartan connection.  Bryant subsequently found this modification and communicated it in an illuminating email \cite{Bry2013}, from which the following is an extract:\\

\setlength{\leftskip}{0.5cm}
\setlength{\rightskip}{0.5cm}

``{\em My conclusion is that Cartan, at the time that he wrote his 5-variables paper, was
more interested in practical calculation and exposition than he was in a theoretical
development of what later became Cartan connections, and he chose his coframing
solving his equivalence problem so that the formulae that he needed to write down to
explain his results would be as simple as possible.''}\\
 
 \setlength{\leftskip}{0pt}
 \setlength{\rightskip}{0pt}

 The article \cite{Car1910} is full of interesting ideas and has strongly influenced methods of studying equivalence and symmetry in the 20th century.  Nevertheless, when a Cartan connection exists, our contention is that there are strong technical benefits for using it.  Historically, the general theory of Cartan connections came much later\footnote{To our knowledge, the earliest appearance of the modern Cartan connection definition appears to be \cite[p.218]{KN1964}.}, and Cartan geometry has developed significantly since its inception, so we feel it is important to draw upon current tools to complete and communicate aspects of the $(2,3,5)$ story for a modern audience.
 
 Beyond this technical prelude, let us further clarify motivations for our study.  It is important to emphasize three {\em equivalent} perspectives on homogeneous $(2,3,5)$-structures:
 \begin{align}
 \framebox{Coordinate} \,\,\longleftrightarrow\,\,
 \framebox{Lie-theoretic} \,\,\longleftrightarrow\,\,
 \framebox{Cartan-theoretic} 
 \end{align}
 {\bf Coordinate models} are the most concrete and familiar manner of presenting a given structure.  It is well-known that any $(2,3,5)$-geometry $(M,\cD)$ can always be put in {\sl Monge form}, i.e.\ there exist coordinates $(x,y,p,q,z)$ on $M$ and some {\sl Monge function} $F = F(x,y,p,q,z)$ with $F_{qq} \neq 0$ everywhere such that
 \begin{align}
 \cD = \cD_F := \langle \partial_x + p\partial_y + q\partial_p + F \partial_z, \, \partial_q \rangle.
 \end{align}
 The local classification of complex multiply-transitive models \cite{Car1910, DG2013, DK2014, Willse2019} is then given in Table \ref{F:MongeModels}.  We also refer the reader to the recent classifications in \cite{Zhi2021} (not in Monge form).

 \begin{table}[h]
 \[
 \begin{array}{|c|c|c|c|} \hline
 \mbox{Type} & \mbox{Monge function $F$} & \mbox{Parameter range} & \mbox{Classifying invariant}\\ \hline\hline
 \sfO & q^2 & \cdot & \cdot\\ \hline
 \sfN.7 & q^{(k+1)/2}, \quad \log(q) & k \in \bbC \backslash \{ \pm \frac{1}{3}, \pm 1, \pm 3\} &{\bf I}^2 = \frac{(k^2+1)^2}{(k^2-9)(\frac{1}{9} - k^2)} \\ \hline
 \sfN.6 & q^{1/3} + y & \cdot & \cdot\\ \hline
 \sfD.6 & y^{(3\ell-1)/2} q^{(\ell+1)/2},\quad \frac{q^2}{p^2},\quad p^2 + q\log(q) & \ell \in \bbC \backslash \{ 0, \pm \frac{1}{3}, \pm 1,  \pm 3 \} & \ell^2\\ \hline
 \end{array}
 \]
 \caption{Local classification of complex multiply-transitive $(2,3,5)$-distributions}
 \label{F:MongeModels}
 \end{table}

 More abstract, but still familiar, is the {\bf Lie-theoretic} description, namely as a filtered Lie algebra $\ff = \ff^{-3} \supset ... \supset \ff^0 \supset ... \supset \ff^3$ with $\ff / \ff^0$ endowed with an $\ff^0$-invariant filtration of $(2,3,5)$-type\footnote{For non-flat homogeneous $(2,3,5)$-models, the filtration is shorter, since $\ff^1 = \ff^2 = \ff^3 = 0$ is forced (Corollary \ref{C:PR}).}:
 \begin{align} \label{E:LieTh}
 \ff / \ff^0 = \ff^{-3} / \ff^0 \supset \ff^{-2} / \ff^0 \supset \ff^{-1} / \ff^0.
 \end{align}
The passage between these descriptions is familiar: (i) given a coordinate model, find the symmetry algebra (which is filtered via a choice of generic basepoint); (ii) given a Lie-theoretic model, locally integrate structure equations to find a coordinate model.

 However, beyond model classifications, one is often also interested in further structural questions, e.g.\
 \begin{enumerate}
 \item[(a)] What is the {\em full} Cartan curvature of the model?  (The Cartan quartic is only the fundamental (harmonic) part.)
 \item[(b)] What is the (infinitesimal) holonomy of the model?
 \item[(c)] From \cite{Nur2005}, any $(2,3,5)$-geometry $(M,\cD)$ admits a canonically associated {\sl Nurowski conformal structure} $\sfc_\cD$ on $M$. Does $\sfc_\cD$ contain an Einstein metric?
 \end{enumerate}
 Answering such questions for a model presented in a coordinate or Lie-theoretic fashion is generally highly non-trivial.  (Indeed, we are not aware of full Cartan curvature previously stated for any non-flat $(2,3,5)$-models.)  Substantial work has been done on questions (b) and (c) in \cite{Willse2014, SW2017a, SW2017b, Willse2018}.  In particular, complete results were established there for the type $\sfN$ case, and partial results were given in the $\sfD.6$ case.
 
 We will advocate for the yet more abstract, but much less known {\bf Cartan-theoretic} descriptions $(\ff;\fg,\fp)$ (Definition \ref{D:alg-model}), from which a Lie-theoretic description is immediate, for which the full Cartan curvature is inherent to the description, and the infinitesimal holonomy is quickly computed.  Moreover, we show how to systematically classify such Cartan-theoretic descriptions.  (We also exhibit the passage from Lie-theoretic to Cartan-theoretic descriptions in numerous examples.)
 
 Let us contrast our $(2,3,5)$ study with that of previous classification studies (in particular, \cite{Car1910,Willse2019,Zhi2021}):
 \begin{itemize}
 \item We take Theorem \ref{T:fundthm} as our starting point, and {\em use $G_2$ structure theory throughout our study}.  In \cite{Car1910}, $G_2$ only gets brief mentions, and is linked only with maximally symmetric models.
 \item The curvature module (Table \ref{F:CurvMod}) for arbitrary $(2,3,5)$-geometries is derived (as a $P$-module).  Corresponding Cartan connection structure equations are given in Appendix \ref{S:curvature}.
 \item For all {\em non-flat} (complex or real) multiply-transitive $(2,3,5)$-structures:
 \begin{itemize}
 \item We give Cartan-theoretic descriptions and establish completeness of the list.
 \item Full Cartan curvature and infinitesimal holonomy is obtained.
 \item Those $(M,\cD)$ for which $\sfc_\cD$ contains an Einstein metric are identified.
 \end{itemize}
 \item We Cartan-theoretically prove exceptionality of the $3:1$ ratio for 2-spheres rolling on each other without twisting/slipping, yielding a $(2,3,5)$-distribution with $\Lie(\GS)$ symmetry (Theorem \ref{T:rolling}).
 \end{itemize}
 We hope that this article serves as a thorough illustration of Cartan-theoretic approaches to homogeneous classification, and that these methods also find applicability well-beyond the $(2,3,5)$-setting.
 
%%%%%%%%%%%%%%%%%%%%%%%%%%%%%%%%%%%%%%%%%%%%%%%%%%%%
 \section{Preliminaries on $(G_2,P_1)$}
 \label{S:G2data}
%%%%%%%%%%%%%%%%%%%%%%%%%%%%%%%%%%%%%%%%%%%%%%%%%%%%
 \subsection{The Lie algebra $\fg = \Lie(G_2)$}
%%%%%%%%%%%%%%%%%%%%%%%%%%%%%%%%%%%%%%%%%%%%%%%%%%%%
 Consider the 14-dimensional {\em complex} simple Lie algebra $\fg = \Lie(G_2)$.  This has the following matrix realization\footnote{We conjugated \cite[eq.\ (25)]{HS2009} by $\diag\left(-1,1,1,1,\begin{psmallmatrix} 0 & -1\\ -1 & 0 \end{psmallmatrix}, 1\right)$ and relabelled the diagonal to obtain \eqref{E:g2rep}.}, corresponding to the standard representation of $\fg$ on $\bbV = \bbC^7$:
 \begin{align} \label{E:g2rep}
 \begin{footnotesize}
 \begin{array}{c}
 \begin{pmatrix}
 2z_1 + z_2 & b_{10} & b_{11} & \sqrt{2} \,b_{21} & b_{31} & b_{32} & 0\\
 a_{10} & z_1 + z_2 & b_{01} & \sqrt{2}\, b_{11} & -b_{21} & 0 & -b_{32}\\
 a_{11} & a_{01} & z_1 & -\sqrt{2}\, b_{10} & 0 & b_{21} & -b_{31}\\
 \sqrt{2} \, a_{21} & \sqrt{2}\, a_{11} & -\sqrt{2}\, a_{10} & 0 & \sqrt{2}\, b_{10} & -\sqrt{2}\, b_{11} & -\sqrt{2}\,b_{21}\\
 a_{31} & -a_{21} & 0 & \sqrt{2}\, a_{10} & -z_1 & -b_{01} & -b_{11}\\
 a_{32} & 0 & a_{21} & -\sqrt{2}\, a_{11} & -a_{01} & -z_1 - z_2 & -b_{10}\\
 0 & -a_{32} & -a_{31} & -\sqrt{2}\, a_{21} & -a_{11} & -a_{10} & -2z_1 - z_2
 \end{pmatrix}
 \end{array}.
 \end{footnotesize}
 \end{align}
Let $[\cdot,\cdot]$ be the Lie bracket on $\fg$, given by restriction of the matrix commutator.
Denoting the standard basis of $\bbV$ by $\{ v_i \}$, and $\{ v^i \}$ its dual basis, one has the following $\fg$-invariant tensors\footnote{Our convention for symmetric tensor products is $vw := \half(v\otimes w + w \otimes v)$, and similarly for wedge products.}\,:
 \begin{align}
 g &= 2 (v^1 v^7 + v^2 v^6 + v^3 v^5) + v^4 v^4, \\
 \Psi &= v^1 \wedge v^4 \wedge v^7 - v^2 \wedge v^4 \wedge v^6 - v^3 \wedge v^4 \wedge v^5 + \sqrt{2} (v^1 \wedge v^5 \wedge v^6 + v^2 \wedge v^3 \wedge v^7).
 \end{align}

Differentiate \eqref{E:g2rep} with respect to $z_i, a_{st}, b_{st}$ to respectively obtain $\sfZ_i, f_{st}, e_{st}$, with brackets in Table \ref{F:G2br}.  Taking Cartan subalgebra $\fh = \langle \sfZ_1, \sfZ_2 \rangle$, and any $\alpha \in \fh^*$, let $\fg_\alpha := \{ x \in \fg : [h,x] = \alpha(h) x, \, \forall h \in \fh \}$.  The root system is $\Delta := \{ \alpha \in \fh^* \backslash \{ 0 \} : \fg_\alpha \neq 0 \}$ and the root space decomposition is $\fg = \fh \op \bigoplus_{\alpha \in \Delta} \fg_\alpha$.  Let $\{ \sfZ_1, \sfZ_2 \}$ have dual basis $\{ \alpha_1, \alpha_2 \}$, which is a simple root system, with Borel subalgebra and positive system 
 \begin{align}
 \fb = \fh \op \bigoplus_{\alpha \in \Delta^+} \fg_\alpha, \qquad 
 \Delta^+ := \{  \alpha_1, \,\, \alpha_2, \,\, \alpha_1 + \alpha_2, \,\, 2\alpha_1 + \alpha_2, \,\, 3\alpha_1 + \alpha_2, \,\, 3\alpha_1 + 2\alpha_2 \}.
 \end{align}
 We have $\Delta = \Delta^+ \cup (-\Delta^+)$.  For each $\alpha = s\alpha_1 + t\alpha_2 \in \Delta^+$, we have standard $\fsl_2$-triples $\{ f_{st}, h_{st}, e_{st} \}$ given by root vectors $e_\alpha := e_{st}$ and $e_{-\alpha} := f_{st}$, and $h_{st} := [e_{st},f_{st}]$.
 \begin{footnotesize}
 \begin{table}[h]
 \[
 \begin{array}{c|cccccccccccccc}
 [\cdot,\cdot] & \sfZ_1 & \sfZ_2 & e_{01} & e_{10} & e_{11} & e_{21} & e_{31} & e_{32} & f_{01} & f_{10} & f_{11} & f_{21} & f_{31} & f_{32} \\ \hline
 \sfZ_1 & \cdot & \cdot & \cdot & e_{10} & e_{11} & 2 e_{21} & 3 e_{31} & 3 e_{32} & \cdot & -f_{10} & -f_{11} & -2 f_{21} & -3 f_{31} & -3 f_{32} \\
 \sfZ_2 & & \cdot & e_{01} & \cdot & e_{11} & e_{21} & e_{31} & 2 e_{32} & -f_{01} & \cdot & -f_{11} & -f_{21} & -f_{31} & -2 f_{32}\\
 e_{01} &&& \cdot & -e_{11} & \cdot & \cdot & e_{32} & \cdot & -\sfZ_1 + 2\sfZ_2 & \cdot & f_{10} & \cdot & \cdot & -f_{31}\\
 e_{10} &&&&\cdot & 2 e_{21} & -3 e_{31} & \cdot & \cdot & \cdot & 2\sfZ_1 - 3\sfZ_2 & -3 f_{01} & -2 f_{11} & f_{21} & \cdot\\
 e_{11} &&&&&\cdot & 3 e_{32} & \cdot & \cdot & e_{10} & -3 e_{01} & -\sfZ_1 + 3\sfZ_2 & 2 f_{10} & \cdot & -f_{21}\\
 e_{21} &&&&&&\cdot & \cdot & \cdot & \cdot & -2 e_{11} & 2 e_{10} & \sfZ_1 & -f_{10} & f_{11}\\
 e_{31} &&&&&&&\cdot & \cdot & \cdot & e_{21} & \cdot & -e_{10} & \sfZ_1 - \sfZ_2 & f_{01}\\
 e_{32} &&&&&&&&\cdot & -e_{31} & \cdot & -e_{21} & e_{11} & e_{01} & \sfZ_2\\
 f_{01} &&&&&&&&&\cdot & f_{11} & \cdot & \cdot & -f_{32} & \cdot\\
 f_{10} &&&&&&&&&&\cdot & -2 f_{21} & 3 f_{31} & \cdot & \cdot\\
 f_{11} &&&&&&&&&&&\cdot & -3f_{32} & \cdot & \cdot\\
 f_{21} &&&&&&&&&&&&\cdot & \cdot & \cdot\\
 f_{31} &&&&&&&&&&&&&\cdot & \cdot\\
 f_{32} &&&&&&&&&&&&&&\cdot \\
 \end{array}
 \]
 \caption{Bracket relations for $\fg = \Lie(G_2)$}
 \label{F:G2br}
 \end{table}
 \end{footnotesize}
 
 The Killing form $B(x,y) = \tr(\ad_x \circ \ad_y)$ is the symmetric bilinear pairing on $\fg$ given by
  \begin{align} \label{E:KillingForm}
 B = 16\left(3 \sfZ_1^*\sfZ_1^* + 3 \sfZ_1^* \sfZ_2^* + \sfZ_2^* \sfZ_2^* + e_{01}^* f_{01}^* + 3 e_{10}^* f_{10}^* + 3 e_{11}^* f_{11}^* + 3 e_{21}^* f_{21}^* + e_{31}^* f_{31}^* + e_{32}^* f_{32}^*\right),
 \end{align}
 with $B|_\fh$ is non-degenerate, and let $\inp{\cdot}{\cdot}$ be the induced pairing on $\fh^*$.  Explicitly, $B|_\fh$ is represented by $8\begin{psmallmatrix} 6 & 3 \\ 3 & 2 \end{psmallmatrix}$ in the basis $\{ \sfZ_1, \sfZ_2 \}$, and $\inp{\cdot}{\cdot}$ is represented by $\tfrac{1}{24}\begin{psmallmatrix} 2 & -3 \\ -3 & 6 \end{psmallmatrix}$ in the basis $\{ \alpha_1, \alpha_2 \}$.  The latter restricts to an inner product on $\fh^*_\bbR := \tspan_\bbR \Delta$, inducing the root diagram, cf.\ Figure \ref{F:G2gr}.  Letting $\alpha_j^\vee = \frac{2\alpha_j}{\inp{\alpha_j}{\alpha_j}}$, the Cartan matrix is
 $(c_{ij}) = (\inp{\alpha_i}{\alpha_j^\vee\!}) = \begin{psm} 2 & -1\\ -3 & 2 \end{psm}$, with corresponding Dynkin diagram $\Gdd{ww}{}$.  The fundamental weights $\{ \lambda_1, \lambda_2 \}$ are defined via $\inp{\lambda_i}{\alpha_j^\vee\!} = \delta_{ij}$.  The Cartan matrix provides a change of basis: $\alpha_i = \sum_j c_{ij} \lambda_j$ and $\lambda_i = \sum_j c^{ij} \alpha_j$, where $(c^{ij})$ is the inverse of $(c_{ij})$, so we have $\lambda_1 = 2\alpha_1 + \alpha_2$ and $\lambda_2 = 3\alpha_1 + 2\alpha_2$.

%%%%%%%%%%%%%%%%%%%%%%%%%%%%%%%%%%%%%%%%%%%%%%%%%%%%
 \subsection{A parabolic subalgebra}
 \label{S:par}
%%%%%%%%%%%%%%%%%%%%%%%%%%%%%%%%%%%%%%%%%%%%%%%%%%%%
 Given a complex semisimple Lie algebra $\fg$, a subalgebra $\fp$ is {\sl parabolic} if it contains a Borel subalgebra $\fb$.  Having fixed one such $\fb$, a parabolic $\fp \supset \fb$ is completely determined by which negative simple root spaces are used to augment $\fb$.  These are conveniently encoded by crosses on Dynkin diagrams: the $\alpha_i$ node is crossed iff $\fg_{-\alpha_i} \not\subset \fp$.  Alternatively, $\fp$ can be introduced via gradings: if $I_\fp$ is the crossed index set, define the {\sl grading element} $\sfZ := \sum_{i \in I_\fp} \sfZ_i$, induce the grading $\fg = \bigoplus_{k \in \bbZ} \fg_k$ via $\fg_k = \{ x \in \fg : [\sfZ,x] = kx \}$ (with $[\fg_i,\fg_j] \subset \fg_{i+j}$, $\forall i,j \in \bbZ$), and then let $\fp := \fg_{\geq 0}$ and $\fp_+ := \fg_+$.

Although the grading is convenient, it is only $\fg_0$-invariant.  Instead, it is better to focus on the $\fp$-invariant (decreasing) filtration on $\fg$ given by $\fg^i := \bigoplus_{j \geq i} \fg_j$, for which $\fg$ becomes a filtered Lie algebra, i.e.\ $[\fg^i,\fg^j] \subset \fg^{i+j}$, $\forall i,j \in \bbZ$.  This has associated-graded $\tgr(\fg) := \bigoplus_{k \in \bbZ} \tgr_k(\fg)$, where $\tgr_k(\fg) := \fg^k / \fg^{k+1}$.  As graded Lie algebras, $\tgr(\fg) \cong \fg$ and in this identification, we let $\tgr_k : \fg^k \to \fg_k$ denote the {\sl leading part}, i.e.\ if $x \in \fg^k$ and $x = x_k + x_{k+1} + ... $, where $x_j \in \fg_j$, then we define $\tgr_k(x) = x_k$. 
 
 Let us define some relevant Lie groups.  Let $\Aut(\fg)$ denote the automorphisms of $\fg$, which has Lie algebra  the space of derivations $\fder(\fg)$.  Let $P,P_+,G_0 \subset \Aut(\fg)$ be the connected Lie subgroups corresponding to the inclusions $\fp,\fp_+,\fg_0 \inj \fder(\fg)$.  We have $P_+ = \exp(\ad\, \fp_+)$, and $P \cong G_0 \ltimes P_+$.  The filtration defined above is $P$-invariant, while the grading is $G_0$-invariant.

 \framebox{$(2,3,5)$}: Consider only the $\Gdd{xw}{}$ case, so that $\sfZ := \sfZ_1$ and $\fp := \fb \op \fg_{-\alpha_2}$.  (The corresponding group is denoted $P = P_1$, with the subscript corresponding to the position of the cross.)  This yields the grading $\fg = \fg_{-3} \op ... \op \fg_3$ (Figure \ref{F:G2gr}) with $\fg_0 = \fz(\fg_0) \times \fg_0^{\ss} \cong \fgl(2,\bbC)$, where $\fz(\fg_0) = \langle \sfZ \rangle$ and $\fg_0^{\ss} = \langle f_{01}, h_{01}, e_{01} \rangle \cong \fsl(2,\bbC)$.  The filtration on $\fg$ descends to the quotient $\fg / \fp$ as
  \begin{align} \label{E:gp-filtration}
 \fg / \fp = \fg^{-3} / \fp \supset \fg^{-2} / \fp \supset \fg^{-1} / \fp \supset 0,
 \end{align}
 which is a {\sl $(2,3,5)$-filtration}, i.e.\ its associated-graded is $\tgr(\fg/\fp) \cong \fg_-$ as nilpotent graded Lie algebras.
 
 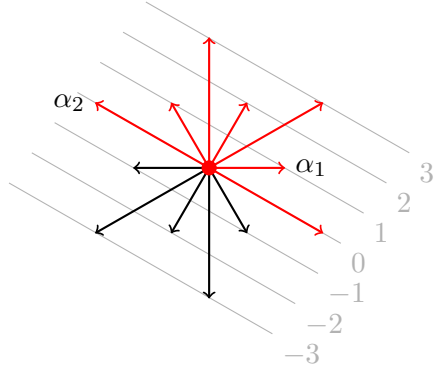
\begin{figure}[h]
 \begin{tabular}{c}
 \begin{tikzpicture}[scale=1]
 \foreach \i in {-3,...,3}
	\draw[\graycolor, xshift=3*\i mm, yshift=4*\i mm]  (150:2) to (150:-2) node[below right]{${\i}$};
 \foreach \i in {180,240,300}
	\draw[thick,->] (0,0) to (\i:1);
 \foreach \i in {210, 270}
	\draw[thick,->] (0,0) to (\i:{sqrt(3)});
 \foreach \i in {0,60,120}
	\draw[thick,->, red] (0,0) to (\i:1);
 \foreach \i in {-30,30,...,150}
	\draw[thick,->, red] (0,0) to (\i:{sqrt(3)});
 \fill[red] (0,0) circle (0.1);
 \node[right] at (1,0) {$\alpha_1$};
 \node[left] at (150:{sqrt(3)}) {$\alpha_2$};
 \end{tikzpicture}
 \end{tabular}
 \caption{Grading of $\fg = \Lie(G_2)$ relevant for our $(2,3,5)$-study}
 \label{F:G2gr}
 \end{figure}
%%%%%%%%%%%%%%%%%%%%%%%%%%%%%%%%%%%%%%%%%%%%%%%%%%%%
 \subsection{Kostant's theorem}
 \label{S:Kostant}
%%%%%%%%%%%%%%%%%%%%%%%%%%%%%%%%%%%%%%%%%%%%%%%%%%%%
 Consider $C_k(\fp_+,\fg) := \bigwedge^k \fp_+ \otimes \fg$, with induced $P$-module structure, and the chain complex $(C_\bullet(\fp_+,\fg),\partial^*)$ with $\partial^*$ the standard Lie algebra homology differential.  By $P$-equivariancy of $\partial^*$, the homology groups $H_k(\fp_+,\fg) := \frac{\ker(\partial^*)}{\im(\partial^*)}$ are also $P$-modules, which are in fact {\sl completely reducible}, i.e.\ $P_+$ acts trivially, so only the $G_0$-action is relevant.  The Killing form induces $\fp_+ \cong (\fg/\fp)^*$ as $P$-modules and $\fp_+ \cong (\fg_-)^*$ as $G_0$-modules, so  $C_k(\fp_+,\fg) \cong C^k(\fg_-,\fg) := \bigwedge^k (\fg_-)^* \otimes \fg$ as $G_0$-modules.  The latter form a cochain complex $(C^\bullet(\fg_-,\fg),\partial)$ with respect to the standard Lie algebra cohomology differential $\partial$, which is $G_0$-equivariant.  Moreover, we have a $G_0$-invariant algebraic Hodge decomposition:
 \begin{align} \label{E:Hodge}
 C^k(\fg_-,\fg) \cong \im(\partial) \op \ker(\Box) \op \im(\partial^*),
 \end{align}
 where $\Box = \partial \partial^* + \partial^* \partial$ is an algebraic Laplacian, with $\ker(\Box) = \ker(\partial) \cap \ker(\partial^*)$.   Then 
 \begin{align}
 H_k(\fp_+,\fg) = \frac{\ker(\partial^*)}{\im(\partial^*)} \cong \ker(\Box) \cong \frac{\ker(\partial)}{\im(\partial)} \cong H^k(\fg_-,\fg).
 \end{align}
 Its $G_0$-module structure is efficiently computed via Kostant's theorem \cite{Kos1961}: if $\lambda$ is the highest (= minus lowest) weight of $\fg$, then as a $\fg_0$-module,
 \begin{align}
 H_k(\fp_+,\fg) \cong H^k(\fg_-,\fg) \cong \bigoplus_{w \in W^\fp(k)} \bbU_{-w \bullet \lambda},
 \end{align}
 where:
 \begin{enumerate}
 \item[(a)] $W^\fp(k)$ consists of the length $k$ elements of the Hasse diagram $W^\fp$, which is a distinguished subset of the Weyl group of $\fg$; see below for $k=2$.
 \item[(b)] the affine Weyl group action is used, i.e.\ $w\bullet \lambda = w(\lambda + \rho) - \rho$, where $\rho$ is the lowest form, i.e.\ the sum of the fundamental weights;
 \item[(c)] $\bbU_\mu$ is the (unique up to isomorphism) $\fg_0$-irrep with {\em lowest} weight $\mu$.
 \end{enumerate}
 Moreover, Kostant gave an explicit lowest weight vector (lwv) $\phi_0$, realized as a harmonic (co-)chain.

 The $k=1$ case underpins parabolic geometries, e.g.\ for $(2,3,5)$, see the discussion above Theorem \ref{T:fundthm}.
 
 The $k=2$ case is most relevant to our current study, and $\ker(\partial^*)$ (``{\sl normal} elements'') will henceforth refer to the kernel of $\partial^* : C_2(\fp_+,\fg) \to C_1(\fp_+,\fg)$ given on decomposable elements by
 \begin{align}
 \partial^*(X \wedge Y \otimes v) = -Y \otimes [X,v] + X \otimes [Y,v] - [X,Y] \otimes v.
 \end{align}
 This subspace inherits a $P$-invariant filtration, and one of the filtrands is $\ker(\partial^*)^1$ (``{\sl regular, normal} elements''), which is the subspace on which the grading element $\sfZ$ acts with positive eigenvalues.  We have the corresponding subspace $H_2(\fp_+,\fg)^1 \cong_{\fg_0} H^2_+(\fg_-,\fg)$.  Written as a produced of simple root reflections, the relevant words are $w = (jk) = \sigma_j \circ \sigma_k \in W^\fp(2)$, where $j$ corresponds to a crossed node, and $k\neq j$ is either crossed or adjacent to $j$ (in the Dynkin diagram).  In terms of root vectors $e_\alpha$ for $\alpha \in \Delta$, we have
 \begin{align}
 \phi_0 = e_{\alpha_j} \wedge e_{\sigma_j(\alpha_k)} \otimes e_{-w(\lambda)}.
 \end{align}
 
 \framebox{$(2,3,5)$}: The Killing form \eqref{E:KillingForm} induces $\fp_+ \cong (\fg/\fp)^*$ as $P$-modules and $\fp_+ \cong (\fg_-)^*$ as $G_0$-modules via:
 \begin{align} \label{E:KFconvert}
  (e_{10},e_{11}, e_{21}, e_{31}, e_{32}) \mapsto (24 f_{10}^*,24 f_{11}^*,24 f_{21}^*,8 f_{31}^*,8 f_{32}^*).
 \end{align}
 We have $W^\fp(2) = \{ (12) \}$, $\lambda = \lambda_2$, so a simple calculation \cite[Example B.1]{KT2017} yields:
 \begin{align} \label{E:lwv}
 H_2(\fp_+,\fg) \cong \bbU_\mu, \qquad
 \mu = 8\lambda_1 - 4\lambda_2 = 4\alpha_1, \qquad
 \phi_0 = e_{10} \wedge e_{31} \otimes f_{01} \equiv 192 f_{10}^* \wedge f_{31}^* \otimes f_{01}.
 \end{align}
 From this, $H_2(\fp_+,\fg) \cong S^4 \fg_1 \cong S^4 (\fg_{-1})^*$ as modules for $\fg_0 \cong \fgl_2$, since the standard $\fg_0$-rep $\fg_1$ has lowest weight $\alpha_1 = 2\lambda_1 - \lambda_2$.  Thus, $H_2(\fp_+,\fg) = H_2(\fp_+,\fg)^1$ is the space of {\em binary quartics}. Repeatedly applying the raising operator $e_{01} \in \fg_0$ (weight $\alpha_2$) to $\phi_0$ (weight $4\alpha_1$) gives all weights of $H_2(\fp_+,\fg)$:
 \begin{align} \label{E:sl2basis}
 H_2(\fp_+,\fg): \quad
 \begin{array}{|c||c|c|c|c|c|} \hline
 \mbox{Weight} & 4\alpha_1 & 4\alpha_1 + \alpha_2 & 4\alpha_1 + 2\alpha_2 & 4\alpha_1 + 3\alpha_2 & 4\alpha_1 + 4\alpha_2\\ \hline
 \mbox{Abstract weight basis} & \sfy^4 & \sfx\sfy^3 & \sfx^2 \sfy^2 & \sfx^3 \sfy & \sfx^4\\ \hline
 \end{array}
 \end{align}
 Here, take the basis $(\sfx,\sfy) = (-e_{11}, e_{10})$ for $\fg_1$, with respect to which we identify
 \begin{align}
 (\sfZ,h_{01},e_{01}, f_{01}) \quad\longleftrightarrow\quad (\sfx \partial_{\sfx} + \sfy \partial_{\sfy}, \sfx \partial_{\sfx} - \sfy \partial_{\sfy}, \sfx \partial_{\sfy}, \sfy \partial_{\sfx}),
 \end{align}
 and the $\fg_0$-action on $H_2(\fp_+,\fg) \cong S^4 \fg_1$ is induced.

%%%%%%%%%%%%%%%%%%%%%%%%%%%%%%%%%%%%%%%%%%%%%%%%%%%%
 \subsection{Tanaka prolongation}
%%%%%%%%%%%%%%%%%%%%%%%%%%%%%%%%%%%%%%%%%%%%%%%%%%%%
 \begin{defn} Fix $(\fg,\fp)$ as in \S\ref{S:par} with induced grading on $\fg$.  Let $\bbU$ be a $\fg_0$-rep, and $\phi \in \bbU$.  The {\sl extrinsic Tanaka prolongation algebra} $\fa^\phi \subset \fg$ is the graded Lie subalgebra defined by:
 \begin{align}
 \begin{split}
 \fa^\phi_- &:= \fg_-, \quad \fa^\phi_0 := \fann(\phi) = \{ x \in \fg_0 : x \cdot \phi = 0 \},\\
 \fa_k^\phi  &:= \{ x \in \fg_k : [x,\fg_{-1}] \subset \fa_{k-1}^\phi \}, \quad \forall k > 0.
 \end{split}
 \end{align}
 \end{defn}

 We will be concerned with the case $0 \neq \phi \in H_2(\fp_+,\fg)^1$.  Sometimes $(G,P)$ is {\sl prolongation-rigid}, i.e.\ $\fa^\phi_+ = 0$ whenever $0 \neq \phi \in H_2(\fp_+,\fg)^1$, and efficient criteria for assessing this were given in \cite{KT2017}.  Prolongation-rigidity of $(G_2,P_1)$ follows from this, but we also confirm it below in a direct manner.\\
 
 \framebox{$(2,3,5)$}: Any (complex) binary quartic can be written as a product of linear factors, and the multiplicities of these factors indicates the {\sl root type}.  Aside from the trivial root type $\sfO$, all non-trivial root types are shown in Table \ref{F:RootType} with representative element $\phi$ normalized using the $G_0$-action.
  
  \begin{table}[h]
 \[
 \begin{array}{|c|c|c|c|c|} \hline
 \mbox{Root type} & \mbox{Normal form $\phi$} & \mbox{Basis for } \fann(\phi) & \dim(\fa^\phi) \\ \hline\hline
 \sfN & \sfy^4 & \sfZ_2, \,\, f_{01} & 7 \\ \hline
 \mathsf{III} & \sfx\sfy^3 & \sfZ_1 - 4\sfZ_2 & 6 \\ \hline
 \sfD & \sfx^2 \sfy^2 & \sfZ_1 - 2\sfZ_2 & 6 \\ \hline
 \mathsf{II} & \sfx^2\sfy(\sfx-\sfy) & 0 & 5 \\ \hline
 \mathsf{I} & \sfx\sfy(\sfx-\sfy)(\sfx-k\sfy) & 0 & 5 \\ \hline
 \end{array}
 \]
 \caption{Non-trivial root types for $\phi \in H_2(\fp_+,\fg)^1$}
 \label{F:RootType}
 \end{table}
 
 The final column of Table \ref{F:RootType} immediately follows from:

 \begin{lemma} \label{L:PR} $(G_2,P_1)$ is prolongation-rigid.
 \end{lemma}
 
 \begin{proof} Using the $G_0$-action, we may assume that $0 \neq \phi \in H_2(\fp_+,\fg)^1$ is one of the normal forms in Table \ref{F:RootType}.  Let $X = a e_{10} +  b e_{11} \in \fa_1^\phi \subset \fg_1$, so $[X,\fg_{-1}] \subset \fann(\phi)$, i.e. the following lie in $\fann(\phi)$:
 \begin{align}
 [X,f_{10}] = a(2\sfZ_1 - 3\sfZ_2) - 3b e_{01}, \quad
 [X,f_{11}] = b(- \sfZ_1 + 3 \sfZ_2) -3 a f_{01}.
 \end{align}
 From Table \ref{F:RootType}, this forces $X=0$. Thus, $\fa_1^\phi = 0$ and hence $\fa_2^\phi = \fa_3^\phi = 0$, so $\fa_+^\phi = 0$. 
 \end{proof}
 
%%%%%%%%%%%%%%%%%%%%%%%%%%%%%%%%%%%%%%%%%%%%%%%%%%%%
 \section{Cartan geometries and Cartan-theoretic descriptions of homogeneous models}
%%%%%%%%%%%%%%%%%%%%%%%%%%%%%%%%%%%%%%%%%%%%%%%%%%%%
 \subsection{Cartan geometries}
%%%%%%%%%%%%%%%%%%%%%%%%%%%%%%%%%%%%%%%%%%%%%%%%%%%%
  Let $G$ be a Lie group and $P$ a closed subgroup.  A {\sl Cartan geometry} of type $(G,P)$ is a (right) $P$-principal bundle $\cG \to M$ equipped with a {\sl Cartan connection} $\omega \in \Omega^1(\cG;\fg)$:
 \begin{enumerate}
 \item $\omega_u : T_u \cG \to \fg$ is a linear isomorphism $\forall u \in \cG$;
 \item $R_p^* \omega = \Ad_{p^{-1}} \omega$, $\forall p \in P$;
 \item $\omega(X^\dagger) = X$, $\forall X \in \fp$, in terms of  fundamental vertical vector fields: $(X^\dagger)_u = \left.\frac{d}{dt}\right|_{t=0} R_{\exp(tX)} u$.
 \end{enumerate}
 The {\sl curvature} is $K = d\omega + \frac{1}{2} [\omega,\omega] \in \Omega^2(\cG;\fg)$, which is horizontal, i.e.\ $K(X^\dagger,\cdot) = 0$, $\forall X \in \fp$.  Letting $\kappa(x,y) = K(\omega^{-1}(x), \omega^{-1}(y))$, we induce the {\sl curvature function} $\kappa : \cG \to \bigwedge^2(\fg/\fp)^* \otimes \fg$.  Flatness ($\kappa=0$) characterizes local equivalence to $(G \to G/P, \omega_G)$, where $\omega_G$ is the (left) Maurer--Cartan form on $G$.
 
 Parabolic geometries are Cartan geometries of type $(G,P)$, where $G$ is semisimple and $P$ is parabolic.  The geometry is {\sl regular and normal} if $\kappa$ is valued in $\ker(\partial^*)^1$.  In this case, the {\sl harmonic curvature} is $\kappa_H := \kappa \,\mod \im(\partial^*) \in H_2(\fp_+,\fg)^1$, which completely obstructs flatness, i.e.\ $\kappa_H = 0$ iff $\kappa = 0$.
 
 There is a well-established categorical equivalence between regular, normal parabolic geometries of type $(G,P)$ and underlying geometric structures.  Specializing to $(2,3,5)$, we use Kostant's theorem with $W^\fp(1) = \{ (1) \}$ to calculate that $H^1(\fg_-,\fg)$ is a 3-dimensional $\fg_0$-irrep with lowest $\fg_0$-weight $2\lambda_1 - 2\lambda_2 = -2\alpha_1 - 2\alpha_2$, so $\sfZ = \sfZ_1$ acts on it with eigenvalue $-2$.  From \cite[Thm.3.1.14 \& Rmk.3.3.7]{CS2009}, we obtain:
 
 \begin{theorem}[Fundamental theorem of $(2,3,5)$-geometry] \label{T:fundthm}
 There is an equivalence of categories between $(2,3,5)$-geometries $(M,\cD)$ and regular, normal Cartan geometries $(\cG \to M, \omega)$ of type $(G,P) = (G_2,P_1)$.
 \end{theorem}

 From \S\ref{S:Kostant}, $H_2(\fp_+,\fg)$ is the space of binary quartics, and we refer to $\kappa_H$ as the {\sl Cartan quartic} \cite{Car1910}.

%%%%%%%%%%%%%%%%%%%%%%%%%%%%%%%%%%%%%%%%%%
\subsection{The curvature module for $(2,3,5)$}
%%%%%%%%%%%%%%%%%%%%%%%%%%%%%%%%%%%%%%%%%%%%%%%%%%%%
 Although curvature $\kappa$ is algebraically constrained to lie in $\ker(\partial^*)^1$, this $P$-module is a priori quite large.  The {\sl structure equations}
 \begin{align} \label{E:streq}
 d\omega = -\tfrac{1}{2} [\omega,\omega] + K,
 \end{align}
 are further geometrically constrained by integrability conditions, i.e.\ $d^2 = 0$ (yielding Bianchi identities).  To deduce finer constraints on $\kappa$ beyond regularity / normality, we use the following key result \cite[Cor.3.2]{Cap2006}:
 
 \begin{theorem} \label{T:Cap} Let $\bbE \subset \ker(\partial^*) \subset \bigwedge^2 \fp_+ \otimes \fg \cong \bigwedge^2 (\fg/\fp)^* \otimes \fg$ be a $P$-submodule and put $\bbE_0 := \bbE \cap \ker(\Box)$.  Let $(\cG \to M, \omega)$ be a regular, normal parabolic geometry such that $\kappa_H$ has values in $\bbE_0$.  If either:
 \begin{enumerate}
 \item[(a)] $(\cG \to M, \omega)$ is torsion-free, i.e.\ $\kappa$ is valued in $\bigwedge^2 (\fg/\fp)^* \otimes \fp$, or
 \item[(b)] $\Box(\bbE) \subset \bbE$ and $\bbE$ is {\sl stable under $\bbE$-insertions}, i.e.\ $\forall \varphi, \psi \in \bbE$, we have $\partial^* (\iota_\psi \varphi) \in \bbE$, where $\iota_\psi \varphi$ is the alternation of $(X_0,X_1,X_2) \mapsto \varphi(\psi(X_0,X_1) + \fp,X_2)$ for $X_i \in \fg / \fp$,
 \end{enumerate}
 then $\kappa$ has values in $\bbE$.
 \end{theorem}
 
 Restricting now to $(2,3,5)$, we define $\bbE$ in a simple manner from the seed data $\phi_0 = e_{10} \wedge e_{31} \otimes f_{01}$:
 
 \begin{defn}
 Let $\bbE \subset \ker(\partial^*)^1$ be the $P$-submodule generated by the lowest weight vector $\phi_0 = e_{10} \wedge e_{31} \otimes f_{01}$ from \eqref{E:lwv}.  We refer to it as the {\sl curvature module} for $(2,3,5)$-geometries.
 \end{defn}
 
 \begin{table}[h]
 \begin{tiny}
 \begin{framed}
 \[ \hspace{-0.1in}
 \begin{array}{@{}c@{\,}c@{\,\,}c@{\,\,}l@{\,\,}l@{}}
 \mbox{Hom.} & \mbox{Label} & \mbox{$\fg_0$-module} & \mbox{Weight} & \mbox{Weight vector as a 2-chain in $\bigwedge^2 \fp_+ \otimes \fp$} \\ \hline
 4 & A & \Gtwo{xw}{-8,4} 
 &\begin{array}{l} 
 4\alpha_1\\
 4\alpha_1 +\,\,\, \alpha_2\\
 4\alpha_1 + 2\alpha_2\\
 4\alpha_1 + 3\alpha_2\\
 4\alpha_1 + 4\alpha_2
 \end{array} & 
 \begin{array}{l}
 e_{10} \wedge e_{31} \otimes f_{01}\\
 (e_{10} \wedge e_{32} - e_{11} \wedge e_{31}) \otimes f_{01} + e_{10} \wedge e_{31} \otimes h_{01}\\
 (e_{10} \wedge e_{32} - e_{11} \wedge e_{31}) \otimes h_{01} - e_{10} \wedge e_{31} \otimes e_{01} - e_{11} \wedge e_{32} \otimes f_{01}\\
 (e_{11} \wedge e_{31} - e_{10} \wedge e_{32} ) \otimes e_{01} - e_{11} \wedge e_{32} \otimes h_{01}\\
 e_{11} \wedge e_{32} \otimes e_{01}
 \end{array}\\ \hline
 5 & B & \Gtwo{xw}{-7,3} &
 \begin{array}{ll} 
 5\alpha_1 + \,\,\,\alpha_2\\
 5\alpha_1 + 2\alpha_2\\
 5\alpha_1 + 3\alpha_2\\
 5\alpha_1 + 4\alpha_2
 \end{array} & 
 \begin{array}{l}
 e_{10} \wedge e_{31} \otimes e_{10} - 2 e_{21} \wedge e_{31} \otimes f_{01}\\
 (e_{10} \wedge e_{32} - e_{11} \wedge e_{31}) \otimes e_{10} - e_{10} \wedge e_{31} \otimes e_{11} - 2 e_{21} \wedge e_{32} \otimes f_{01} - 2 e_{21} \wedge e_{31} \otimes h_{01}\\
 (e_{11} \wedge e_{31} - e_{10} \wedge e_{32}) \otimes e_{11} - e_{11} \wedge e_{32} \otimes e_{10}+ 2 e_{21} \wedge e_{31} \otimes e_{01} - 2 e_{21} \wedge e_{32} \otimes h_{01}\\
 e_{11} \wedge e_{32} \otimes e_{11} + 2 e_{21} \wedge e_{32} \otimes e_{01}
 \end{array}\\ \hline
 6 & C & \Gtwo{xw}{-6,2} &
 \begin{array}{l}
 6\alpha_1 + 2\alpha_2\\
 6\alpha_1 + 3\alpha_2\\
 6\alpha_1 + 4\alpha_2\\
 \end{array} & 
 \begin{array}{l}
 -e_{10} \wedge e_{31} \otimes e_{21} 
 - 2 e_{21} \wedge e_{31} \otimes e_{10} 
 + 3 e_{31} \wedge e_{32} \otimes f_{01}\\
 (e_{11} \wedge e_{31} - e_{10} \wedge e_{32}) \otimes e_{21} 
 - 2 e_{21} \wedge e_{32} \otimes e_{10}  
  + 2 e_{21} \wedge e_{31} \otimes e_{11} 
 + 3 e_{31} \wedge e_{32} \otimes h_{01}\\
 e_{11} \wedge e_{32} \otimes e_{21} 
 + 2 e_{21} \wedge e_{32} \otimes e_{11} 
 - 3 e_{31} \wedge e_{32} \otimes e_{01}
 \end{array}\\ \hline
 7 & D & \Gtwo{xw}{-5,1} & 
  \begin{array}{l}
 7\alpha_1 + 3\alpha_2\\
 7\alpha_1 + 4\alpha_2\\
 \end{array} &
 \begin{array}{l}
 (e_{10} \wedge e_{32} - e_{11} \wedge e_{31}) \otimes e_{31} - 2 e_{10} \wedge e_{31} \otimes e_{32} + 6 e_{21} \wedge e_{31} \otimes e_{21} + 9 e_{31} \wedge e_{32} \otimes e_{10} \\
 (e_{11} \wedge e_{31} - e_{10} \wedge e_{32}) \otimes e_{32} - 2 e_{11} \wedge e_{32} \otimes e_{31} + 6 e_{21} \wedge e_{32} \otimes e_{21} - 9 e_{31} \wedge e_{32} \otimes e_{11}
 \end{array}
 \\ \hline
 7 & \CD & \Gtwo{xw}{-8,3} & 
 \begin{array}{l}
 7\alpha_1 + 2\alpha_2\\
 7\alpha_1 + 3\alpha_2\\
 7\alpha_1 + 4\alpha_2\\
 7\alpha_1 + 5\alpha_2
 \end{array} & 
 \begin{array}{l}
 e_{10} \wedge e_{31} \otimes e_{31}\\
 (e_{10} \wedge e_{32} - e_{11} \wedge e_{31}) \otimes e_{31} + e_{10} \wedge e_{31} \otimes e_{32}\\
 (e_{10} \wedge e_{32} - e_{11} \wedge e_{31}) \otimes e_{32} 
 - e_{11} \wedge e_{32} \otimes e_{31} \\
 - e_{11} \wedge e_{32} \otimes e_{32} 
 \end{array}\\ \hline
 8 & E & \Gtwo{xw}{-4,0} & 
 \begin{array}{l}
 8\alpha_1 + 4\alpha_2
 \end{array} & \begin{array}{l} 
 e_{21} \wedge e_{31} \otimes e_{32} - e_{21} \wedge e_{32} \otimes e_{31} - 3 e_{31} \wedge e_{32} \otimes e_{21}
 \end{array}\\ \hline
 8 & \CE & \Gtwo{xw}{-7,2} & 
 \begin{array}{l}
 8\alpha_1 + 3\alpha_2\\
 8\alpha_1 + 4\alpha_2\\
 8\alpha_1 + 5\alpha_2
 \end{array} & 
 \begin{array}{l} 
 e_{21} \wedge e_{31} \otimes e_{31} \\
 e_{21} \wedge e_{32} \otimes e_{31} + e_{21} \wedge e_{31} \otimes e_{32} \\
 e_{21} \wedge e_{32} \otimes e_{32}
 \end{array}\\ \hline
 9 & \CF &  \Gtwo{xw}{-6,1} & 
 \begin{array}{l}
 9\alpha_1 + 4\alpha_2\\
 9\alpha_1 + 5\alpha_2
 \end{array} &
 \begin{array}{l}
 e_{31} \wedge e_{32} \otimes e_{31}\\
 e_{31} \wedge e_{32} \otimes e_{32}
 \end{array}
 \end{array}
 \]

 \begin{center}
 \begin{tabular}{c}
 \begin{tikzpicture}[scale=0.75]
  \coordinate (Origin)   at (0,0);
  \draw[brown] ( 1.15  ,  4) circle (0.2);
  \draw[brown] ( 2.3  ,  4) circle (0.2);
  \draw[brown] ( 2.88  ,  3) circle (0.2);
  \clip (-3.1,-2.1) rectangle (5.8,5.5);
  \pgftransformrotate{60}
  \pgftransformcm{0.8660254}{-0.5}{0.5}{0.8660254}{\pgfpoint{0cm}{0cm}}
  \pgftransformcm{1}{-0.57734}{0}{1.15468}{\pgfpoint{0cm}{0cm}}
  \draw[style=dotted] (-15,-15) grid[step=1cm] (15,15);    
  \pgftransformcm{1}{1}{0}{1}{\pgfpoint{0cm}{0cm}}
  \draw[style=dotted] (-15,-15) grid[step=1cm] (15,15);
  \pgftransformcm{1}{0}{-2}{1}{\pgfpoint{0cm}{0cm}}
  \pgftransformcm{-1}{-1}{3}{2}{\pgfpoint{0cm}{0cm}}

 % Coordinates (a,b) correspond to a\alpha_1 + b\alpha_2

\foreach \i in {0,...,5}
    \pgfmathtruncatemacro{\ifour}{\i+4}
	\draw[gray, xshift=10*\i mm, yshift=6.7*\i mm]  (4,4) to (4,-0.5) node {\ifour};

 \foreach \pt in {(-1,0),(-1,-1),(-2,-1),(-3,-1),(-3,-2)}{
   \draw[->,thick,blue] (0,0) -- \pt;
 }
 \foreach \pt in {(1,0),(0,-1),(0,1),(1,1),(2,1),(3,1),(3,2)}{
   \draw[->,very thick,red] (0,0) -- \pt;
 }
 \foreach \pt in {(4,0),(4,1),(4,2),(4,3),(4,4),(5,1),(5,2),(5,3),(5,4),(6,2),(6,3),(6,4),(7,3),(7,4),(8,4)}{
   \node[draw,circle,inner sep=1.5pt,fill] at \pt {};
 }
 \foreach \pt in {(7,2),(7,5),(8,3),(8,5),(9,4),(9,5)}{
   \node[draw,brown,circle,inner sep=1.5pt,fill] at \pt {};
 }
%  \fill[black] (7,3) circle (1.0ex); % Fill circle with base colour (arg#2)
%  \fill[brown] (7,3) -- (180:1ex) arc (180:0:1ex) -- cycle;
 \node[above] at (0,1) {$\alpha_2$};  
 \node[right] at (1,0) {$\alpha_1$};
 \node[below] at (4,0) {$4\alpha_1$};
% \node[below] at (0,0) {$0$};  
 \end{tikzpicture} 
 \end{tabular}\\
 \end{center}
 \end{framed}
 \end{tiny}
 \caption{The curvature module for $(2,3,5)$-geometries}
 \label{F:CurvMod}
 \end{table}
 
 \begin{prop} \label{P:kappa} Any regular, normal parabolic geometry $(\cG \to M, \omega)$ of type $(G,P) = (G_2,P_1)$ has $\kappa$ valued in $\bbE$ defined above.  In particular, (a) $\kappa$ is torsion-free, and (b) $\kappa(x,y) = 0$, $\forall x,y \in \fg^{-2} / \fp$.
 \end{prop}

 \begin{proof} We proceed as in \cite{Sag2008} for (a). Let $\bbE_0 := \bbE \cap \ker(\Box)$ and $\kappa_H$ is clearly valued in $\bbE_0$ by Kostant's theorem. Since $\bbE$ is a $P$-module, decompose it into $G_0$-irreps.  On each $G_0$-irrep, the $G_0$-equivariant map $\Box$ acts by a scaling, so $\Box(\bbE) \subset \bbE$.  Since $\bbE \subset \bigwedge^2 \fp_+ \otimes \fp$, then $\forall \varphi, \psi \in \bbE$, we clearly have $\iota_\psi \varphi = 0$, so $\bbE$ is stable under $\bbE$-insertions.  Theorem \ref{T:Cap}(2) now yields (a).  Claim (b) follows since $\bbE_0$ is a $\fg_0$-submodule of $\fg_1 \wedge \fg_3 \otimes \fg_0$, and applying $\fp$ to it does not produce terms along $(\bigwedge^2 \fg_1 \otimes \fp) \op (\fg_1 \wedge \fg_2 \otimes \fp)$.
 \end{proof}
 
  In Table \ref{F:CurvMod}, we derive $\bbE$ explicitly, find that $\dim(\bbE) = 24$, and stratify $\bbE$ into $G_0$-irreps, with an $\fsl_2$-adapted basis (as in \eqref{E:sl2basis}) specified in each.  Note that we use the ``minus lowest weight'' convention in the 3rd column, e.g.\ $\Gtwo{xw}{-8,4}$ refers to $8\lambda_1 - 4\lambda_2 = 4\alpha_1$.  The weights in $\bbE$ are indicated on the weight diagram of $G_2$.  Each dot indicates a multiplicity one weight space, while that for a ringed dot has multiplicity two.  Tilded components are colored brown.

 Consider the induced filtration $\bbE = \bbE^4 \supset \bbE^5 \supset ... \supset \bbE^9 \supset \bbE^{10} = 0$ induced from $\bigwedge^2(\fg/\fp)^* \otimes \fg$.  Let $\widetilde\bbE \subset \bbE^7$ be the $P$-module corresponding to the $\CD,\CE,\CF$ components (brown dots).   While the $\fp$-action on $\bbE$ is quite complicated, some parts have a simple abstract description.  Namely,
 \begin{enumerate}
 \item As $\fp$-modules, $\bbE / \bbE^5 \cong H_2(\fp_+,\fg)$, which is identified with binary quartics;
 \item $\fg^2 = \langle e_{21}, e_{31}, e_{32} \rangle$ acts trivially on $\bbE / \widetilde\bbE$ (black dots), which we identify with {\em ternary} quartics. 
 \end{enumerate}
 Corresponding covariant tensors were identified in \cite{Car1910}; see also Appendix \ref{S:curvature} for more details.

%%%%%%%%%%%%%%%%%%%%%%%%%%%%%%%%%%%%%%%%%%%%%%%%%%%%
 \subsection{Cartan-theoretic descriptions}
 \label{S:AM}
%%%%%%%%%%%%%%%%%%%%%%%%%%%%%%%%%%%%%%%%%%%%%%%%%%%%
 Suppose that a parabolic geometry $(\cG \to M, \omega)$ of type $(G,P)$ is {\sl homogeneous}, i.e.\ a (local) Lie group acts (on the left) by principal bundle automorphisms preserving $\omega$ with transitive projected action on $M$, so $M \cong F / F^0$ (upon fixing a basepoint $o \in M$).  In this case, we can identify $\cG = F \times_{F^0} P$, and $\omega$ is determined by its value at a point $u \in \cG$ (in the fibre over $o$) \cite{Hammerl2007}. This leads to what we refer to as the {\sl Cartan-theoretic} description of a homogeneous parabolic geometry.
 
 \begin{defn} \label{D:alg-model}
 An {\sl algebraic model $(\ff;\fg,\fp)$} is a Lie algebra $(\ff,[\cdot,\cdot]_\ff)$ such that:
 \begin{enumerate}
 \item[\rm (M1)] $\ff \subset \fg$ is a filtered subspace, with filtrands $\ff^i := \ff \cap \fg^i$, and $\ff / \ff^0 \cong \fg / \fp$.
 \item[\rm (M2)] $\ff^0$ inserts trivially into the {\sl curvature} $\kappa(x,y) := [x,y] - [x,y]_\ff$.
 \item[\rm (M3)] $\kappa \in \bigwedge^2(\ff/\ff^0)^* \otimes \fg \cong \bigwedge^2 (\fg/\fp)^* \otimes \fg$ is regular and normal, i.e.\ $\kappa \in \ker(\partial^*)^1$.
 \end{enumerate}
Define $\kappa_H := \kappa \,\mod \im(\partial^*) \in H_2(\fp_+,\fg)^1$ as the {\sl harmonic curvature}.  When $\kappa = 0$, we say that $(\ff;\fg,\fp)$ is {\sl flat}.  Also, we say that $(\ff;\fg,\fp)$ is {\sl multiply-transitive} if $\dim(\ff^0) > 0$ and {\sl simply-transitive} if $\dim(\ff^0) = 0$.
 \end{defn}
 
 From the Cartan-theoretic description, one can immediately deduce the {\sl Lie-theoretic description}, e.g.\ for $(2,3,5)$, this was described around \eqref{E:LieTh}. The converse passage from Lie-theoretic to Cartan-theoretic descriptions is assured by Theorem \ref{T:fundthm}, and we will give numerous examples of this.
 
 Having fixed $(\fg,\fp,P)$, let $\cM$ be the set of all algebraic models $(\ff;\fg,\fp)$.  Then:
 \begin{enumerate}
 \item $\cM$ is {\em partially ordered}: we declare $\ff \leq \ff'$ when $\ff \inj \ff'$ as Lie algebras.  We are most interested in the {\sl maximal} elements of $(\cM,\leq)$, i.e.\ those $\ff \in \cM$ for which $\ff \leq \ff'$ implies $\ff' = \ff$.
 \item $\cM$ admits a $P$-action via $\ff \mapsto \Ad_p \cdot \ff$ for any $p \in P$.  (Those in the same $P$-orbit are regarded as equivalent.  This arises by varying the point $u \in \cG$ in the same fibre over the basepoint $o \in M$.)
 \end{enumerate}

 The following important facts will be used to carry out classifications -- see \cite[\S 2.4]{The2021} for proofs.
 
  \begin{prop} \label{P:algCC} 
  Let $(\ff;\fg,\fp) \in \cM$.  Then:
 \begin{enumerate}
 \item[(a)] $(\ff,[\cdot,\cdot]_\ff)$ is a filtered Lie algebra.
 \item[(b)] $\ff^0 \cdot \kappa = 0$, i.e.\ $[z,\kappa(x,y)] = \kappa([z,x],y) + \kappa(x,[z,y])$, $\forall x,y \in \ff$, $\forall z \in \ff^0$.
 \item[(c)] Let $\fs := \tgr(\ff)$.  Then $\fs \subset \fa^{\kappa_H}$ is a graded Lie subalgebra.
 \item[(d)] Fix a graded complement $\fs^\perp \subset \fp$ to $\fs$ in $\fg$, and write $\ff = \bigoplus_{i=-\nu}^\nu \{ x + \fd(x) : x \in \fs_i \}$ for some unique linear map $\fd : \fs \to \fs^\perp$ of positive degree.  Fix $T \in \ff^0$ and suppose that $\fs$ and $\fs^\perp$ are $\ad_T$-invariant.  Then $T \cdot \fd = 0$, i.e.\ $\ad_T \circ \fd = \fd \circ \ad_T$.
 \end{enumerate}
 \end{prop}
 
 Prolongation-rigidity (Lemma \ref{L:PR}) implies $\fa_+^{\kappa_H} = 0$, so property (c) immediately yields:
 
 \begin{cor} \label{C:PR} For $(2,3,5)$, consider $(\ff;\fg,\fp) \in \cM$ with $\kappa_H \neq 0$.  Then $\ff^1 = \ff^2 = \ff^3 = 0$.
 
 \end{cor}
 
 The leading parts of $\ff$ are constrained by property (c).  Writing $\ff$ as a graph over $\fs$, and say $x + \fd(x) \in \ff$ for some $x \in \fs_i$, the {\sl deformation map} $\fd$ from (d) is automatically constrained in some cases, e.g.\ when some $T \in \ff^0$ acts in a diagonalizable manner, we can often use this to efficiently constrain $\fd$.
 
 Given $(\ff;\fg,\fp)$, the infinitesimal holonomy algebra $\hol$ of the associated homogeneous Cartan geometry can be efficiently computed (\`a la Ambrose--Singer) \cite{Hammerl2007}.  Namely, define the following subspaces of $\fg$:
 \begin{align} \label{E:hol}
 \hol^0 := \langle \kappa(x,y) : x,y \in \ff \rangle, \qquad
 \hol^i := \hol^{i-1} + [\ff,\hol^{i-1}], \,\,\forall i \geq 1.
\end{align}
Since $\dim(\fg)$ is finite, this increasing sequence necessarily stabilizes to some $\hol^\infty$, and we have $\hol = \hol^\infty$.
 
 Finally, let us discuss automorphisms and anti-involutions.  The latter determine real forms.
 
 \begin{defn}
 An {\sl automorphism} of an algebraic model $(\ff;\fg,\fp)$ is $A \in \Aut(\fg)$ such that $A(\fp) = \fp$ and $A|_\ff \in \Aut(\ff,[\cdot,\cdot]_\ff)$ (or equivalently, $A^* \kappa = \kappa$), and let $\Aut(\ff;\fg,\fp)$ denote all such automorphisms. 
 \end{defn}

 Recall that an anti-involution of a Lie algebra $\fg$ is an anti-linear map $\psi : \fg \to \fg$ with $\psi^2 = \id$ and $\psi([x,y]) = [\psi(x),\psi(y)]$, $\forall x,y \in \fg$.  The fixed point set $\fg^\psi = \{ x \in \fg : \psi(x) = x \}$ is a {\sl real form} of $\fg$.
 
 \begin{defn} \label{D:AI}
  An {\sl anti-involution} of $(\ff;\fg,\fp)$ is an anti-involution $\psi$ of $\fg$ with $\psi(\fp) = \fp$ and $\psi|_\ff$ an anti-involution of $(\ff,[\cdot,\cdot]_\ff)$.  Two such $\psi_1,\psi_2$ are {\sl conjugate} if $\psi_2 = A \circ \psi_1 \circ A^{-1}$ for some $A \in \Aut(\ff;\fg,\fp)$.
 \end{defn}
 
 Since $\fp$ and $\ff$ are stable under (anti-)automorphisms, then so is the canonical $\fp$-invariant filtration on $\fg$ and the induced filtration on $\ff$.

%%%%%%%%%%%%%%%%%%%%%%%%%%%%%%%%%%%%%%%%%%%%%%%%%%%%
 \section{Multiply-transitive complex $(2,3,5)$-distributions}
 \label{S:MT}
%%%%%%%%%%%%%%%%%%%%%%%%%%%%%%%%%%%%%%%%%%%%%%%%%%%%
 In this section, our aim is to establish the following result:
 
 \begin{theorem}
 For {\bf complex} $(2,3,5)$-structures, the complete Cartan-theoretic classification of all non-flat, maximal, multiply-transitive algebraic models $(\ff;\fg,\fp)$, up to the $P$-action, is given in Table \ref{F:alg-models}.
 \end{theorem}
 
 Note that from Table \ref{F:alg-models}, the filtration \eqref{E:LieTh} is inherited via the inclusion $\ff \subset \fg$, and brackets $[\cdot,\cdot]_\ff$ are immediately obtained, so this provides the Lie-theoretic description of these structures (Table \ref{F:Fstr}).

 \begin{footnotesize}
   \begin{table}[h]
 \[
 \begin{array}{|c|l|l@{}|} \hline
 \mbox{Label} & \mbox{Filtered basis of $\ff$} & \mbox{Curvature $\kappa$ (with $\kappa_i$ terms of homogeneity $i$)}\\ \hline\hline
 \sfN.7_c & 
 \begin{array}{@{}l@{}}
  \,\,\,T = \sfZ_2, \, N = f_{01},\\
  X_1 = f_{10} + c e_{10}, \, X_2 = f_{11},\\
  X_3 = f_{21},\,
  X_4 = f_{31}, \, X_5 = f_{32}
  \end{array} &
 %%%%
 \begin{array}{@{}l}
 \kappa = \kappa_4 = f_{10}^* \wedge f_{31}^* \otimes f_{01}\\
 \mbox{(Classifying invariant: $c^2 \in \bbC$.)}
 \end{array} \\ \hline
 %%%%%
 \sfN.6 & 
  \begin{array}{@{}l@{}}
  \,\,N = f_{01},\\
  X_1 = f_{10} + e_{01} + 6 e_{10} + 2 e_{32}, \\
  X_2 = f_{11} + \sfZ_1 - 2\sfZ_2 + 2 e_{31}, \\
  X_3 = f_{21} + 9 e_{10} + 2 e_{21},\\
  X_4 = f_{31} - 2 \sfZ_1 + \sfZ_2 - e_{11} - 4 e_{31}, \\
  X_5 = f_{32} - e_{10}
  \end{array} &
 %%%%
 \begin{array}{@{}l}
 \kappa = 42\kappa_4 - 30 \kappa_5 + 20 \kappa_6 - 4 \kappa_7 + 6 \kappa_8,\qbox{where}\\
 \begin{cases}
 \kappa_4 = f_{10}^* \wedge f_{31}^* \otimes f_{01}\\
 \kappa_5 = f_{10}^* \wedge f_{31}^* \otimes e_{10} - 2 f_{21}^* \wedge f_{31}^* \otimes f_{01}\\
 \kappa_6 = -f_{10}^* \wedge f_{31}^* \otimes e_{21} - 2 f_{21}^* \wedge f_{31}^* \otimes e_{10} + f_{31}^* \wedge f_{32}^* \otimes f_{01}\\
 \kappa_7 =  
 (f_{10}^* \wedge f_{32}^*  - f_{11}^* \wedge f_{31}^*) \otimes e_{31} 
 - 2 f_{10}^* \wedge f_{31}^* \otimes e_{32}\\
 \hspace{0.5in} 
 + 6 f_{21}^* \wedge f_{31}^* \otimes e_{21}
 + 3 f_{31}^* \wedge f_{32}^* \otimes e_{10} \\
 \kappa_8 = f_{21}^* \wedge f_{31}^* \otimes e_{32} - f_{21}^* \wedge f_{32}^* \otimes e_{31} - f_{31}^* \wedge f_{32}^* \otimes e_{21}
 \end{cases}
 \end{array} \\ \hline
 %%%%
 \sfD.6_a & 
 \begin{array}{@{}l@{}}
 \,\,\,T = h_{01} = -\sfZ_1 + 2\sfZ_2,\\
 X_1 = f_{10} + a e_{11} + e_{32},\\
 X_2 = f_{11} + a e_{10} + e_{31}, \\
 X_3 = f_{21} + (a^2 + 1) e_{21},\\
 X_4 = f_{31} + e_{11} + a(a^2 + \tfrac{1}{3}) e_{32},\\
 X_5 = f_{32} + e_{10} + a(a^2 + \tfrac{1}{3}) e_{31}
 \end{array} & 
 %%%%
 \begin{array}{@{}l@{}}
 \kappa = -4 \kappa_4 + \frac{4a}{3} \kappa_6 - 2a^2 \kappa_8,\\
 \begin{cases}
 \kappa_4 = (f_{10}^* \wedge f_{32}^* - f_{11}^* \wedge f_{31}^*) \otimes h_{01} - f_{10}^* \wedge f_{31}^* \otimes e_{01} - f_{11}^* \wedge f_{32}^* \otimes f_{01} \\
 \kappa_6 = (f_{11}^* \wedge f_{31}^* - f_{10}^* \wedge f_{32}^*) \otimes e_{21} - 2 f_{21}^* \wedge f_{32}^* \otimes e_{10}  \\
 \hspace{0.4in} + 2 f_{21}^* \wedge f_{31}^* \otimes e_{11} + f_{31}^* \wedge f_{32}^* \otimes h_{01}\\
 \kappa_8 = f_{21}^* \wedge f_{31}^* \otimes e_{32} - f_{21}^* \wedge f_{32}^* \otimes e_{31} - f_{31}^* \wedge f_{32}^* \otimes e_{21}
 \end{cases} \\
 \mbox{(Classifying invariant: $a^2 \in \bbC$.)}
 \end{array} \\ \hline
 \end{array}
 \]
 \caption{Classification of non-flat multiply-transitive algebraic models $(\ff;\fg,\fp)$}
 \label{F:alg-models}
 \end{table}

 \begin{table}[h]
 \[
 \begin{array}{|c|l|}\hline
 \mbox{Label} & \mbox{Lie bracket on $\ff$, calculated via $[\cdot,\cdot]_\ff = [\cdot,\cdot] - \kappa(\cdot, \cdot)$}\\ \hline\hline
 \sfN.7_c & 
 \begin{array}{c|ccccccc}
 [\cdot,\cdot]_\ff & T & N  & X_1 & X_2 & X_3 & X_4 & X_5\\ \hline
 T & \cdot & -N & \cdot & -X_2 & -X_3 & -X_4 & -2 X_5\\
 N && \cdot & X_2 & \cdot & \cdot & -X_5 & \cdot\\
 X_1 &&&  \cdot & -3c N - 2 X_3 & -2c X_2 + 3 X_4 & -N + cX_3 & \cdot\\
 X_2 &&& & \cdot & -3 X_5 & \cdot & \cdot \\
 X_3 &&&&&\cdot & \cdot & \cdot\\
 X_4 &&&&&&\cdot & \cdot\\
 X_5 &&&&&&&\cdot 
 \end{array}\\ \hline
 \sfN.6 & \begin{array}{c|ccccccc}
 [\cdot,\cdot]_\ff & N  & X_1 & X_2 & X_3 & X_4 & X_5\\ \hline
 N & \cdot & X_2 & -2N & \cdot & -X_5 + N & \cdot\\
 X_1 &&  \cdot & -18N + 2X_1- 2X_3 & -12X_2 + 3 X_4 & -2X_1 + 6X_3 - 42 N & -X_4\\
 X_2 && & \cdot & 27 N - 3 X_5 & - X_2 - X_4 & -N + X_5 \\
 X_3 &&&&\cdot & -60 N + 6 X_3 & \cdot\\
 X_4 &&&&&\cdot & -24 N + 2 X_3 + 4 X_5\\
 X_5 &&&&&&\cdot 
 \end{array}\\ \hline
 \sfD.6_a & \begin{array}{c|cccccc}
 [\cdot,\cdot]_\ff & T & X_1 & X_2 & X_3 & X_4 & X_5\\ \hline
 T & \cdot & X_1 & -X_2 & \cdot & X_4 & -X_5\\
 X_1 &&  \cdot & 3aT - 2X_3 & \,\,\,\,2aX_1 + 3X_4  & \cdot & 6T - aX_3 \\
 X_2 && & \cdot & - 2aX_2 - 3X_5 & - 6 T + aX_3 & \cdot\\
 X_3 &&&&\cdot & -(a^2+3) X_1 & (a^2+3) X_2\\
 X_4 &&&&&\cdot & a(a^2-1) T - 2 X_3 \\
 X_5 &&&&&&\cdot
 \end{array}\\ \hline
 \end{array}
 \]
 \caption{Lie-theoretic structure: $\ff^{-1} / \ff^0 = \langle X_1, X_2 \rangle$, $\ff^{-2} / \ff^{-1} = \langle X_3 \rangle$, $\ff^{-3} / \ff^{-2} = \langle X_4, X_5 \rangle$.}
 \label{F:Fstr}
 \end{table}
 \end{footnotesize}

 From Table \ref{F:RootType} and Proposition \ref{P:algCC}(c), the cases that need to be considered are:
  \begin{align}
 \sfN.7, \quad \sfN.6, \quad \sfD.6, \quad \mathsf{III}.6,
 \end{align}
 where $\dim(\ff)$ has been adjoined to the root type.
 The following tools will be used:
 \begin{enumerate}
 \item[(i)] The structure group $P \leq \Aut(\fg)$ will be used to normalize the inclusion $\ff \subset \fg$.  Indeed, we may assume that the $G_0$-action has already normalized $\kappa_H \in H_2(\fp_+,\fg)$ to correspond to (a multiple of) the representative element specified in Table \ref{F:RootType}.  Then $\fs = \tgr(\ff) \subset \fa^{\kappa_H}$.
 \item[(ii)] When $\ff^0$ satisfies the hypothesis of Proposition \ref{P:algCC}(d), we have $\ff^0 \cdot \fd = 0$, which constrains the inclusion $\ff \subset \fg$.  (In particular, this is valid when $\ff^0$ acts reductively, which is not the case for $\sfN.6$.)
 \item[(iii)] Suppose that $X_1,X_2 \in \ff^{-1}$, with leading parts spanning $\fg_{-1}$, i.e.\ $\fg_{-1} = \langle \tgr_{-1}(X_1), \tgr_{-1}(X_2) \rangle$.  Using Proposition \ref{P:kappa}(b), we can generate elements of $\ff^{-2} \backslash \ff^{-1}$ and $\ff^{-3} \backslash \ff^{-2}$ via:
 \begin{align} \label{E:235br}
 [X_1,X_2]_\ff = [X_1,X_2], \quad
 [X_1,[X_1,X_2]_\ff]_\ff = [X_1,[X_1,X_2]], \quad
 [X_2,[X_1,X_2]_\ff]_\ff = [X_2,[X_1,X_2]].
 \end{align}
 \item[(iv)] $\kappa \in \bbE$, where the curvature module $\bbE$ is given in Table \ref{F:CurvMod}.
 \item[(v)] $\kappa(\ff^0,\cdot) = 0$ and $\ff^0 \cdot \kappa = 0$.
 \item[(vi)] Closure of $\ff$ under the bracket $[\cdot,\cdot]_\ff = [\cdot,\cdot] - \kappa(\cdot,\cdot)$.
 \item[(vii)] Jacobi identities for $[\cdot,\cdot]_\ff$.  Let $\Jac^\ff$ denote the Jacobiator.
 \end{enumerate}
 For (ii) and (v), the following notation will be convenient:
 \begin{framed}
 Given an $\fh$-module $\bbW$ and a weight $\lambda \in \fh^*$, we let $\bbW_{[\lambda]}:= \bigoplus_r \bbW_{r\lambda}$ be the sum of all weight spaces for weights that are (arbitrary) multiples of $\lambda$.
 \end{framed}

 Finally, the $\sfN.7$ and $\sfD.6$ cases will contain residual parameters.  Invariant functions of these parameters under the residual (discrete) structure group are referred to as {\sl classifying invariants}.
 
%%%%%%%%%%%%%%%%%%%%%%%%%%%%%%%%%%%%%%%%%%%%%%%%%%%% 
 \subsection{Non-existence of $\mathsf{III}.6$ structures}
%%%%%%%%%%%%%%%%%%%%%%%%%%%%%%%%%%%%%%%%%%%%%%%%%%%%
 Assume $(\ff;\fg,\fp) \in \cM$ is type $\mathsf{III}.6$.  We may assume (Table \ref{F:RootType}) that $\kappa_H \leftrightarrow \mathsf{x y^3}$ with weight $4\alpha_1 + \alpha_2$.  Since $\dim(\ff) = 6$, then
 \begin{align}
 \fs &= \tgr(\ff) = \fa^{\kappa_H} = \fg_- \op \fa_0, \qbox{where} \fa_0 = \langle \sfZ_1 - 4\sfZ_2 \rangle.
 \end{align}
 
 \begin{prop} \label{P:III.6}
 Multiply-transitive type $\mathsf{III}.6$ algebraic models do not exist.
 \end{prop}
 
 \begin{proof}
 Let $T \in \ff^0$ with $\tgr_0(T) = \sfZ_1 - 4\sfZ_2$, i.e.\ $T = \sfZ_1 - 4\sfZ_2 + c_{10} e_{10} + ... + c_{32} e_{32}$.  Note $\ad_T|_{\fp_+} = \diag(1,-3,-2,-1,-5)$ in the basis $\{ e_{10}, e_{11}, e_{21}, e_{31}, e_{32} \}$ of $\fp_+$, so using $P_+ = \exp(\ad\, \fp_+)$ (and redefining $\ff$), we may normalize \framebox{$T = \sfZ_1 - 4\sfZ_2 \in \ff^0$}.
(Key here is that zero is not an eigenvalue of $\ad_T|_{\fp_+}$.  Equivalently, no multiple of $4\alpha_1 + \alpha_2$ occurs among $\Delta(\fp_+) := \{ \alpha \in \Delta : \sfZ(\alpha) > 0 \}$.)
 
 Letting $\fs^\perp = \langle \sfZ_1, e_{01}, f_{01} \rangle \op \fp_+$, the decomposition $\fg = \fs \op \fs^\perp$ is $\ad_T$-invariant.  By Proposition \ref{P:algCC}(d), the deformation map $\fd : \fs \to \fs^\perp$ satisfies $T \cdot \fd = 0$, i.e. $\fd \in \bbW_{[4\alpha_1 + \alpha_2]}$, where $\bbW = \fs^* \otimes \fs^\perp$.  This forces
 \begin{align}
 \fd(f_{10}) = a e_{31}, \quad \fd(f_{11}) = 0, \quad \fd(f_{21}) = 0, \quad \fd(f_{31}) = b e_{10}, \quad \fd(f_{32}) = 0,
 \end{align}
 for some $a,b \in \bbC$, so $X_1 := f_{10} + \fd(f_{10}) = f_{10} + a e_{31} \in \ff$, $X_2 := f_{11} \in \ff$, etc.  Also, $\fd(T) = 0$ since $T \in \ff^0$.
 
We have $T \cdot \kappa = 0$ from Proposition \ref{P:algCC}(b), so $\kappa \in \bbE_{[4\alpha_1+\alpha_2]}$.  Referring to $\bbE$ in Table \ref{F:CurvMod}, $\kappa$ must have weight $4\alpha_1 + \alpha_2$, so converting the corresponding weight vector to a 2-cochain via \eqref{E:KFconvert}, we have:
 \begin{align}
 \kappa = c\left( ( f_{10}^* \wedge f_{32}^* - f_{11}^* \wedge f_{31}^*) \otimes f_{01} + f_{10}^* \wedge f_{31}^* \otimes h_{01} \right),
 \end{align}
 where $c \in \bbC^\times$ and $h_{01} = [e_{01},f_{01}] = -\sfZ_1 + 2\sfZ_2$ (Table \ref{F:G2br}).  Now evaluate brackets $[\cdot,\cdot]_\ff = [\cdot,\cdot] - \kappa(\cdot,\cdot)$:
 \begin{align}
 [X_1,X_3]_\ff &= [f_{10} + a e_{31}, f_{21}] = 3f_{31} - a e_{10} = 3 X_4 - (a+3b) e_{10}, \label{E:III-1}\\
 [X_1,X_5]_\ff &= [f_{10} + a e_{31}, f_{32}] - c f_{01} = (a - c) f_{01}, \label{E:III-2}\\
 [X_1,X_4]_\ff &= [f_{10} + a e_{31},f_{31} + be_{10}] - c h_{01} = (a+c-2b) \sfZ_1 - (a+2c-3b) \sfZ_2. \label{E:III-3}
 \end{align}
 Since $\ff$ is closed under $[\cdot,\cdot]_\ff$, then \eqref{E:III-1} and \eqref{E:III-2} imply $c = a = -3b$, while \eqref{E:III-3} becomes $[X_1,X_4]_\ff = \tfrac{4c}{3}(2\sfZ_1 - 3\sfZ_2) \in \ff^0 = \langle T \rangle$, so $c=0$ and $\kappa = 0$, contradicting the type $\mathsf{III}$ hypothesis.
 \end{proof}
 
 In \cite{KT2017}, we gave an efficient proof of Proposition \ref{P:III.6} (alternative to \cite{Car1910}), but the approach used there is not effective for classifying all type $\sfN$ and $\sfD$ cases.  Our proof above is a warm-up for these cases.

%%%%%%%%%%%%%%%%%%%%%%%%%%%%%%%%%%%%%%%%%%%%%%%%%%%%
 \subsection{Type $\sfN$ structures}
 \label{S:typeN}
%%%%%%%%%%%%%%%%%%%%%%%%%%%%%%%%%%%%%%%%%%%%%%%%%%%%
 Let $(\ff;\fg,\fp) \in \cM$ be type $\sfN$.  We may assume (Table \ref{F:RootType}) that $\kappa_H \leftrightarrow \sfy^4$ with weight $4\alpha_1$.  Here, we have either $\dim(\ff) = 6$ or $7$, with:
 \begin{align} \label{E:TypeN}
 \fs &= \tgr(\ff) \subset \fa^{\kappa_H} = \fg_- \op \fa_0, \qbox{where} \fa_0 = \langle \sfZ_2, f_{01} \rangle.
 \end{align} 
%%%%%%%%%%%%%%%%%%%%%%%%%%%%%%%%%%%%%%%%%%%%%%%%%%%%
 \subsubsection{$\sfN.7$ structures}
 \label{S:N7}
%%%%%%%%%%%%%%%%%%%%%%%%%%%%%%%%%%%%%%%%%%%%%%%%%%%% 
 \begin{theorem} \label{T:N7} Any $\sfN.7$ algebraic model is $P$-equivalent to $(\ff;\fg,\fp)$ given in Table \ref{F:alg-models} $(\sfN.7_c)$ for some parameter $c \in \bbC$.  These have classifying invariant $\cI = c^2 \in \bbC$.
 \end{theorem}

 \begin{proof}
 Necessarily, $\fs = \tgr(\ff) = \fa^{\kappa_H}$.  Let $T,N \in \ff^0$ with $\tgr_0(T) = \sfZ_2$ and $\tgr_0(N) = f_{01}$.  Since $\sfZ_2(\alpha) \neq 0$, $\forall\alpha \in \Delta(\fp_+) \backslash \{ \alpha_1 \}$, we normalize to $T = \sfZ_2 + a e_{10}$ using the $P_+$-action.  This reduces $P_+$ to $\exp(\ad\,\fg_{\alpha_1})$.
 
 Let $X_1 \in \ff^{-1}$ with $\tgr_{-1}(X_1) = f_{10}$.  We use $\exp(\ad\,\fg_{\alpha_1})$ to remove the $h_{10} = 2\sfZ_1 - 3\sfZ_2$ component from $X_1$.  Taking linear combinations of $X_1$ with $T$ and $N$, we may assume that $X_1 \equiv f_{10} + b e_{01}\,\, \mod \fp_+$.  Since $\kappa(T,\cdot) = 0$, then $[T,X_1]_\ff = [T,X_1] \equiv ah_{10} + be_{01} \,\, \mod \fp_+$, so $[T,X_1]_\ff \in \ff^0$.  But $\tgr_0([T,X_1]_\ff) \cap \fa_0 = 0$, so $a=b=0$.  Thus, $[T,X_1]_\ff \in \ff^1 = 0$, so \framebox{$T = \sfZ_2 \in \ff^0$} and \framebox{$X_1 = f_{10} + c e_{10}$}.

 Letting $\fs^\perp = \langle \sfZ_1, e_{01} \rangle \op \fp_+$, the decomposition $\fg = \fs \op \fs^\perp$ is $\ad_T$-invariant.  By Proposition \ref{P:algCC}(d), we have $T \cdot \fd = 0$, so $\fd \in (\fs^* \otimes \fs^\perp)_{[\alpha_1]} = \langle \sfZ_2^* \otimes e_{10}, f_{10}^* \otimes e_{10} \rangle$.  (Note: all weights of $\fs^\perp$ are non-negative.)  Since $\fd(\sfZ_2) = 0$, we obtain the claimed basis for $\ff$ in Table \ref{F:alg-models} $(\sfN.7_c)$.

 By Proposition \ref{P:algCC}(2), we have $T \cdot \kappa = 0$, so $\kappa \in \bbE_{[\alpha_1]}$.  From Table \ref{F:CurvMod}, $\kappa$ is a nonzero multiple of $\kappa_4 = f_{10}^* \wedge f_{31}^* \otimes f_{01}$ (with weight $4\alpha_1$).  Using $\exp(\ad_{t\sfZ})$, we can arrange $\kappa = \kappa_4$.
  
  In Table \ref{F:Aut-algmodel}, $\Aut(\ff;\fg,\fp)$ is described.  Consider $A_\zeta \in \Aut(\fg)$, extended from $A_\zeta \begin{psm} f_{10} \\ f_{11} \end{psm} = \begin{psm} \zeta f_{10} \\ f_{11} \end{psm}$, which induces $(\kappa_H,c) \mapsto (\frac{\kappa_H}{\zeta^4}, \frac{c}{\zeta^2})$.  Setting $\zeta^4=1$ preserves $\kappa_H$, and $c^2$ is the classifying invariant.
  \end{proof}
   
 The $\sfN.7$ family is described by Cartan in \cite[\S 44]{Car1910} in terms of structure equations for an abstract coframing $(\omega_1,...,\omega_5,\varpi_1,\varpi_2)$ with parameter ${\bf I}$.  Here is a dictionary from the dual framing to our basis:
 \begin{align}
 (\partial_{\omega_1},...,\partial_{\omega_5}, \partial_{\varpi_1}, \partial_{\varpi_2}) = 
 \left(9 \left(a \zeta\right)^3 X_5, 3 (a \zeta) X_4, X_3, -3 \left(a \zeta\right)^3 X_2, a \zeta X_1, -T,  -3(a \zeta)^2 N \right),
 \end{align}
 where $a = 6^{-1/4}$, $\zeta^4 = 1$, ${\bf I}=\frac{\sqrt{6}}{2} \zeta^2 c$, and classifying invariant ${\bf I}^2 = \tfrac{3}{2} c^2$.
%%%%%%%%%%%%%%%%%%%%%%%%%%%%%%%%%%%%%%%%%%%%%%%%%%%%
 \subsubsection{$\sfN.6$ structures}
 \label{S:N6}
%%%%%%%%%%%%%%%%%%%%%%%%%%%%%%%%%%%%%%%%%%%%%%%%%%%%

Among multiply-transitive cases, $\sfN.6$ is the most challenging, and indeed was overlooked in \cite{Car1910}.  Models in this branch were discovered by Strazzullo \cite[Model 6.7.2]{Str2009} and Doubrov--Govorov \cite{DG2013}, and they are equivalent.  We now prove uniqueness of an $\sfN.6$ structure and match the Monge coordinate description \eqref{E:N6-Monge} (due to Doubrov--Govorov) to its Lie-theoretic and Cartan-theoretic descriptions.
 
 \begin{lemma} \label{L:N6}  Let $(\ff;\fg,\fp)$ be an $\sfN.6$ algebraic model.  Writing $\ff^0 = \langle N \rangle$, we may assume that $N:= f_{01}$.
 \end{lemma}

 \begin{proof} Letting $\fs = \tgr(\ff)$, we have $\dim(\fs_0) = 1$.  There are two inequivalent 1-dimensional subalgebras $\fs_0 \subset \fa_0^{\kappa_H} := \langle \sfZ_2, f_{01} \rangle$, namely $\fs_0 = \langle \sfZ_2 \rangle$ or $\fs_0 = \langle f_{01} \rangle$.  Assuming the former, we proceed as in the proof of Theorem \ref{T:N7}, obtaining $T, X_1 \in \ff$ with $T = \sfZ_2 + a e_{10}$ and $X_1 \equiv f_{10} + b e_{01} + c f_{01} \,\,\mod \fp_+$.  Then
 \begin{align}
 \ff \ni [T,X_1]_\ff = [T,X_1] \equiv a(2\sfZ_1 - 3\sfZ_2) + b e_{01} - c f_{01}\quad \mod \fp_+, 
 \end{align}
 so closure implies $a=b=c=0$, hence $T = \sfZ_2 \in \ff^0$.  We can now continue as in the proof of Theorem \ref{T:N7}, obtaining $\ff = \langle T, X_1, X_2, X_3, X_4, X_5 \rangle$ with $X_i$ defined there.  This is a non-maximal description of an $\sfN.7$ model (namely, we may simply extend it by $N = f_{01}$), so we may exclude this possibility.

 So we start with $\fs = \fg_- \op \langle f_{01} \rangle \subset \fa^{\kappa_H}$.  Let $N \in \ff^0$ with $\tgr_0(N) = f_{01}$.  Applying $\exp(\ad_{\langle e_{11}, e_{32} \rangle})$ to $N$, we may assume that $N = f_{01} + s e_{11} + t e_{21} + u e_{32}$. Now take $X_1 \in \ff^{-1}$ with $\tgr_{-1}(X_1) = f_{10}$.  We compute $[N,[N,X_1]_\ff]_\ff = [N,[N,X_1]] \in \ff$ (and simplify it)  several times:
 \begin{align}
 \begin{split}
 \ff \ni [N,[N,X_1]] &\equiv -4 s\sfZ_1 + 9 s \sfZ_2 \quad \mod \langle f_{01} \rangle \op \fg^1 \qRa s = 0;\\
 &\equiv 4t e_{10} \hspace{0.74in} \mod \langle f_{01} \rangle\op\fg^2  \qRa t = 0;\\
 &\equiv -u e_{21} \hspace{0.67in} \mod \langle f_{01} \rangle\op\fg^3  \qRa u = 0. \quad \therefore \framebox{$N = f_{01} \in \ff^0$}.
 \end{split}
 \end{align}
 \end{proof}
 
   From Table \ref{F:CurvMod}, the curvature condition $N \cdot \kappa = 0$ forces $\kappa$ to be a linear combination of the lowest $\fg_0$-weight vectors in $\bbE$, i.e.\ $\kappa = u_4 \kappa_4 + u_5 \kappa_5 + ... + \tilde{u}_8 \tilde\kappa_8 + \tilde{u}_9 \tilde\kappa_9$, where:
 \begin{align}
 \begin{split} \label{E:N6kappa}
 \kappa_4 &= f_{10}^* \wedge f_{31}^* \otimes f_{01}, \quad
 \kappa_5 = f_{10}^* \wedge f_{31}^* \otimes e_{10} - 2 f_{21}^* \wedge f_{31}^* \otimes f_{01},\\
\kappa_6 &= -f_{10}^* \wedge f_{31}^* \otimes e_{21} - 2 f_{21}^* \wedge f_{31}^* \otimes e_{10} + f_{31}^* \wedge f_{32}^* \otimes f_{01},\\
\kappa_7 &= (f_{10}^* \wedge f_{32}^* - f_{11}^* \wedge f_{31}^*) \otimes e_{31} - 2 f_{10}^* \wedge f_{31}^* \otimes e_{32} + 6 f_{21}^* \wedge f_{31}^* \otimes e_{21} + 3 f_{31}^* \wedge f_{32}^* \otimes e_{10},\\
\tilde\kappa_7 &= f_{10}^* \wedge f_{31}^* \otimes e_{31}, \quad
\kappa_8 = f_{21}^* \wedge f_{31}^* \otimes e_{32} - f_{21}^* \wedge f_{32}^* \otimes e_{31} - f_{31}^* \wedge f_{32}^* \otimes e_{21},\\
\tilde\kappa_8 &= f_{21}^* \wedge f_{31}^* \otimes e_{31}, \quad
\tilde\kappa_9 = f_{31}^* \wedge f_{32}^* \otimes e_{31}.
 \end{split}
 \end{align}
 
 We now continue from Lemma \ref{L:N6} and complete the classification in the $\sfN.6$ branch.  Computationally, the proof becomes quite tedious to do by-hand (particularly invoking the Jacobi identity \eqref{E:N6Jac}).
 
 \begin{theorem} \label{T:N6}
 Up to $P$-equivalence, there is a {\bf unique} $\sfN.6$ algebraic model, given in Table \ref{F:alg-models} $(\sfN.6)$.
 \end{theorem}

 \begin{proof} Let $X_1 \in \ff^{-1}$ with $\tgr_{-1}(X_1) = f_{10}$.  Applying $\exp(\ad(\langle e_{10}, e_{21}, e_{31} \rangle))$ to $X_1$ and adding $N = f_{01}$ if necessary, we may assume $X_1 \in \ff^{-1}$ as below, which we extend to $X_2, X_3, X_4, X_5 \in \ff$ using \eqref{E:235br}.  (Additional terms are added in order to ``row reduce'' the corresponding matrix of coefficients.)
 \begin{align}
 \begin{split}
 X_1 &:= f_{10} + t_1 \sfZ_2 + t_2 e_{01} + t_3 e_{10} + t_4 e_{31} + t_5 e_{32} \\
 X_2 &:= [N,X_1]_\ff - t_1 N = f_{11} + t_2(\sfZ_1 - 2\sfZ_2) + t_5 e_{31} \\
 X_3 &:= -\tfrac{1}{2} ([X_1,X_2]_\ff + t_1 X_2 - 2 t_2 X_1 + 3 t_3 N) \\
 &= f_{21} - \tfrac{1}{2} t_1 t_2 (\sfZ_1 - 4 \sfZ_2) + \tfrac{3}{2} t_2 t_3 e_{10} + t_5 e_{21} + (\tfrac{3}{2} t_2 t_4 - t_1 t_5) e_{31} \\
 X_4 &:= \tfrac{1}{3} ([X_1,X_3]_\ff + t_1 X_3 + 2 t_3 X_2 + \tfrac{t_1 t_2}{2}  X_1)\\
 &= f_{31} - \tfrac{t_2}{6} (t_1^2 + 2 t_3) \sfZ_1 + \tfrac{t_2}{6} ( 5 t_1^2 + t_3 ) \sfZ_2 - \tfrac{t_1 t_2^2}{2}  e_{01} + (\tfrac{5}{6} t_1 t_2 t_3 - \tfrac{t_4}{3}) e_{10} + (t_5 - \tfrac{t_2^2 t_3}{2} ) e_{11} \\
 &\qquad  + (t_1 t_5 - \tfrac{t_2 t_4}{2} ) e_{21} + (t_1 t_2 t_4 - \tfrac{2}{3} t_1^2 t_5 - \tfrac{t_3 t_5}{3}) e_{31} + t_2(\tfrac{t_2 t_4}{2} - t_1 t_5) e_{32}\\
 X_5 &:= -\tfrac{1}{3} ([X_2,X_3]_\ff - \tfrac{3 t_1 t_2}{2} X_2 - \tfrac{9t_2 t_3}{2} N) \\
 &= f_{32} + \tfrac{t_1 t_2^2}{2} (\sfZ_1 - 2\sfZ_2) + (t_5 - \tfrac{t_2^2 t_3}{2} ) e_{10} + t_2(t_1 t_5 - \tfrac{t_2 t_4}{2} ) e_{31}
 \end{split}
 \end{align}
 Define $[\cdot,\cdot]_\ff$ via $\kappa$ satisfying $N \cdot \kappa = 0$, i.e. $\kappa = u_4 \kappa_4 + u_5 \kappa_5 + ... + \tilde{u}_8 \tilde\kappa_8 + \tilde{u}_9 \tilde\kappa_9$, cf.\ \eqref{E:N6kappa}.  Closure of $[\cdot,\cdot]_\ff$ forces complicated relations, so we instead evaluate Jacobi identities (see {\tt Maple} file in {\tt arXiv} submission):
 \begin{align} \label{E:N6Jac}
 \begin{split}
 0 = \Jac^\ff(X_1,X_2,X_4) &\equiv (5t_2 u_4 + 7 u_5) f_{01} \quad\mod \fp_+ \qRa \framebox{$u_5 = -\tfrac{5}{7} t_2 u_4$};\\
 0 = \Jac^\ff(X_1,X_3,X_4) &\equiv -2t_1 (t_2 u_4 - u_5) f_{01} = -\tfrac{24}{7} t_1 t_2 u_4 f_{01} \quad\mod \fh \op \fp_+.
 \end{split}
 \end{align}
 Since $u_4 \neq 0$ (coefficient of the harmonic part $\kappa_4$), then \framebox{$t_1 t_2 = 0$} follows.  Using this, we evaluate:
 \begin{align}
 \begin{split}
 \ff \ni \,\,&[X_1,X_4]_\ff + t_1 X_4 - t_3 X_3 + \tfrac{t_2}{6} (2t_3 + t_1^2) X_1 + u_4 N  \\
 &\equiv t_4 (\tfrac{5}{3} \sfZ_1 - 2 \sfZ_2) + 4( t_5 - \tfrac{t_2^2 t_3}{3} ) e_{01} - (u_5 + \tfrac{5}{6} t_2 t_3^2 + \tfrac{t_1 t_4}{3}) e_{10} \quad \mod \langle e_{11}, e_{21}, e_{31} , e_{32} \rangle.
 \end{split}
 \end{align}
 By closure, we have \framebox{$t_4 = 0$, $t_5 = \tfrac{t_2^2 t_3}{3}$, $u_5 = -\tfrac{5}{6} t_2 t_3^2$}.  From $t_1 t_2 = 0$, we have two cases:
  \begin{enumerate}
 \item \framebox{$t_2 = 0$} ($t_4=t_5=0$): 
 $(X_1,X_2,X_3,X_4,X_5) = (f_{10} + t_1 \sfZ_2 + t_3 e_{10}, f_{11}, f_{21}, f_{31}, f_{32})$. We calculate:
 \begin{align}
 \begin{split}
 [X_1,X_4]_\ff &\equiv -u_5 e_{10} + u_6 e_{21} - \tilde{u}_7 e_{31} + 2 u_7 e_{32} \quad \mod \ff,\\
 [X_3,X_4]_\ff &\equiv 2 u_6 e_{10} - 6 u_7 e_{21} - \tilde{u}_8 e_{31} - u_8 e_{32} \quad \mod \ff,\\
 [X_4,X_5]_\ff &\equiv -3 u_7 e_{10} + u_8 e_{21} - \tilde{u}_9 e_{31} \quad \mod \ff.
 \end{split}
 \end{align}
 These lie in $\ff$, so $u_5 = ... = \tilde{u}_9 = 0$, i.e.\ $\kappa = u_4 \kappa_4$.  We may normalize to $u_4 = 1$, but then $\ff \inj \ff \op \langle \sfZ_2 \rangle$, so $\ff$ is non-maximal since it injects into an $\sfN.7$ model.
 
 \item \framebox{$t_1 = 0,\, t_2 \neq 0$}: Applying $\exp(\ad_{\langle \sfZ_2 \rangle})$ to $X_1$, we may normalize $t_2 = 1$.  From $u_5 = -\tfrac{5}{6} t_2 t_3^2 = -\tfrac{5}{7} t_2 u_4$ earlier, we have \framebox{$0 \neq u_4 = \tfrac{7}{6} t_3^2$}, so applying $\exp(\ad_{\langle \sfZ_1 - \sfZ_2 \rangle})$ to $X_1$ (which preserves the normalization $t_2 = 1$), we may normalize to $t_3 = 6$ (so $u_4 = 42$).
We obtain:
 \begin{align}
 \begin{split}
  X_1 &= f_{10} + e_{01} + 6 e_{10} + 2 e_{32}, \quad
  X_2 = f_{11} + \sfZ_1 - 2\sfZ_2 + 2 e_{31}, \\
  X_3 &= f_{21} + 9 e_{10} + 2 e_{21}, \quad
  X_4 = f_{31} - 2 \sfZ_1 + \sfZ_2 - e_{11} - 4 e_{31}, \quad
  X_5 = f_{32} - e_{10}
  \end{split}
  \end{align}
 We find that
 \begin{align}
 \begin{split}
 [X_1,X_4]_\ff &\equiv -(u_5 + 30) e_{10} + (u_6 - 20) e_{21} - \tilde{u}_7 e_{31} + 2(u_7 + 4) e_{32} \quad \mod \ff,\\
 [X_3,X_4]_\ff &\equiv 2(u_6 - 20) e_{10} - 6(u_7 + 4) e_{21} - \tilde{u}_8 e_{31} - (u_8 - 6) e_{32} \quad \mod \ff,\\
 [X_4,X_5]_\ff &\equiv -3(u_7 + 4) e_{10} + (u_8 - 6) e_{21} - \tilde{u}_9 e_{31} \quad \mod \ff.
 \end{split}
 \end{align}
 These all lie in $\ff$, so $(u_4,u_5,u_6,u_7,u_8,\tilde{u}_7,\tilde{u}_8,\tilde{u}_9) = (42,-30, 20, -4, 6,0,0,0)$ follows.
 \end{enumerate}
 This yields the unique $\sfN.6$ structure in Table \ref{F:alg-models}, for which we can verify $\Jac^\ff = 0$, cf.\ Table \ref{F:Fstr}.
 \end{proof}
 
 The brackets $[\cdot,\cdot]_\ff$ are given in Table \ref{F:Fstr} $(\sfN.6)$, and equipping $\ff \subset \fg$ with the inherited filtration completes the Lie-theoretic description of the structure.  Let us now match this to the Doubrov--Govorov model \eqref{E:N6-Monge}.
 
 \begin{example} \label{X:N6}
 Consider a space $M$ with coordinates $(x,y,p,q,z)$ and distribution $\cD$ spanned by
 \begin{align} \label{E:N6-Monge}
 \partial_x + p\partial_y + q\partial_p + (q^{1/3} + y) \partial_z, \quad \partial_q.
 \end{align}
 The symmetry algebra $\ff = \sym(\cD)$ is spanned by the vector fields
 \begin{align}
 \begin{array}{l}
 \bH = -x\partial_x + y\partial_y + 2p\partial_p + 3q\partial_q,\quad
 \bX = -y\partial_x + p^2 \partial_p + 3pq\partial_q - \frac{y^2}{2} \partial_z,\\
 \bY = -x\partial_y - \partial_p - \frac{x^2}{2}\partial_z,\quad
 \bS = \partial_x, \quad
 \bT = \partial_y + x\partial_z, \quad
 \bU = \partial_z.
 \end{array}
 \end{align}
 We have $\ff \cong \fsl(2,\bbC) \ltimes \mathfrak{heis}_3$ with brackets as in Table \ref{F:LieTh} $(\sfN.6)$.
 Fix the basepoint $o \in M$ given by $(x,y,p,q,z) = (0,0,0,1,0)$.  This point is {\em generic}, i.e.\ the evaluation map $\ev_o : \ff \to T_o M$ is surjective, and we can induce the filtration $\ff = \ff^{-3} \supset \ff^{-2} \supset \ff^{-1} \supset \ff^0 \supset 0$ as the pre-image of the one on $T_o M$, i.e.\ $\ff^0 = \ker(\ev_o)$ and $\ff^i = (\ev_o)^{-i}(\cD^i|_o)$ for $i=1,2,3$.  In Table \ref{F:LieTh} $(\sfN.6)$, a basis $v_0,...,v_5$ is specified with 
 \begin{align}
 \ff^{-3} / \ff^{-2} = \langle v_4, v_5 \rangle, \quad
 \ff^{-2} / \ff^{-1} = \langle v_3 \rangle, \quad 
 \ff^{-1} / \ff^0 = \langle v_1,v_2 \rangle, \quad
 \ff^0 = \langle v_0 \rangle,
 \end{align}
 and a new basis $N,X_1,...,X_5$ is specified satisfying the brackets in Table \ref{F:Fstr} $(\sfN.6)$.
 (To find this, we used undetermined parameters and then matched the Lie algebra structures -- see \S\ref{S:gencase} for an example of this.)  Consequently, Table \ref{F:alg-models} $(\sfN.6)$ is the Cartan-theoretic description of \eqref{E:N6-Monge}.
 \end{example}
 
%%%%%%%%%%%%%%%%%%%%%%%%%%%%%%%%%%%%%%%%%%%%%%%%%%%%
 \subsection{Type $\sfD$ structures}
 \label{S:typeD}
%%%%%%%%%%%%%%%%%%%%%%%%%%%%%%%%%%%%%%%%%%%%%%%%%%%%
 Let $\ff \in \cM$ be a $\sfD.6$ structure.  We may assume (Table \ref{F:RootType}) that $\kappa_H \leftrightarrow \sfx^2 \sfy^2$ with weight $4\alpha_1 + 2\alpha_2$, so for $h_{01} = -\sfZ_1 + 2\sfZ_2$, we have:
 \begin{align}
 \fs = \tgr(\ff) = \fa^{\kappa_H} = \fg_- \op \fa_0, \qbox{where} \fa_0 = \langle h_{01} \rangle.
 \end{align}
 
 \begin{theorem} \label{T:typeD}
 Any $\sfD.6$ algebraic model is $P$-equivalent to $(\ff;\fg,\fp)$ given in Table \ref{F:alg-models} $(\sfD.6_a)$, for some parameter $a \in \bbC$.  These have classifying invariant $\cI = a^2 \in \bbC$.
 \end{theorem}
 
 \begin{proof} 
 Let $T \in \ff^0$ with $\tgr_0(T) = h_{01} = -\sfZ_1 + 2\sfZ_2$.  Since $\alpha(h_{01}) \neq 0$, $\forall \alpha \in \Delta(\fp_+) \backslash \{ 2\alpha_1 + \alpha_2 \}$, then using the $P_+$-action, we normalize $T = -\sfZ_1 + 2\sfZ_2 + t e_{21}$, which reduces $P_+$ to $\exp(\ad\,\langle e_{21} \rangle)$.  Consider $X_1 \in \ff^{-1}$ with $\tgr_{-1}(X_1) = f_{10}$.   Adding $T$ if necessary, we may assume that $X_1 = f_{10} + c_1 f_{01} + c_2 \sfZ_2 + c_3 e_{01} \, \mod \fp_+$.
 We have $\kappa(T,\cdot) = 0$, so $[T,X_1]_\ff = [T,X_1]$.  The following must lie in $\ff$ (in fact, $\ff^0 = \langle T \rangle$):
 \begin{align} \label{E:TX1}
 \ff \ni [T,X_1]_\ff - X_1 &\equiv -3 c_1 f_{01} - c_2 \sfZ_2 + c_3 e_{01} - 2t e_{11} \quad \mod \langle e_{10}, e_{21}, e_{31}, e_{32} \rangle,
 \end{align}
 hence $t = c_1 = c_2 = c_3 = 0$.  Thus, \framebox{$T = h_{01} = -\sfZ_1 + 2\sfZ_2$} and $X_1 = f_{10}\, \mod \fp_+$.
 
  Letting $\fs^\perp = \langle \sfZ_1, e_{01}, f_{01} \rangle \op \fp_+$, the decomposition $\fg = \fs \op \fs^\perp$ is $\ad_T$-invariant.  By Proposition \ref{P:algCC}(d), we have $T \cdot \fd = 0$, so $\fd \in (\fs^* \otimes \fs^\perp)_{[2\alpha_1+\alpha_2]}$.  This implies that
 \begin{align}
 \begin{split}
 X_1 = f_{10} + a \,e_{11} + b \,e_{32}, \quad
 X_2 = f_{11} + a' e_{10} + b' e_{31}, \quad
 X_3 = f_{21} + c\, e_{21},
 \end{split}
 \end{align}
 corresponding to terms $f_{10}^* \otimes e_{11}$ and $f_{10}^* \otimes e_{32}$ of weights $2\alpha_1 + \alpha_2$ and $4\alpha_1 + 2\alpha_2$ respectively, etc.  Using $\exp(\ad\,\langle e_{21} \rangle)$, $X_3$ has been normalized by removing $h_{21} = \sfZ_1$ terms from it.

 By Proposition \ref{P:kappa}(b) and \eqref{E:235br}: for $1 \leq i < j \leq 3$, $\kappa(X_i,X_j) = 0$, so $[X_i,X_j]_\ff = [X_i,X_j]$.  Then:
 \begin{itemize}
 \item $\ff \ni [X_1,X_2]_\ff - 3 a T + 2X_3 = (a-a') (2\sfZ_1 - 3\sfZ_2) - (2 a a' + b+b' - 2 c) e_{21}$.
 Hence,
 \begin{align}
 a' = a, \quad c = a a' + \tfrac{1}{2} (b+b'), \quad [X_1,X_2]_\ff = 3 a T - 2X_3.
 \end{align} 

  \item Define $X_4,X_5 \in \ff^{-3}$ via 
 $[X_1,X_3]_\ff = 2 a X_1 + 3 X_4$ and
 $[X_2,X_3]_\ff = - 2 a X_2 - 3 X_5$.  Explicitly, 
 \begin{align}
 \begin{split}
  X_4 = f_{31} + \tfrac{2b+b'}{3} e_{11} + a\left( a^2 - \tfrac{b}{6} + \tfrac{b'}{2} \right) e_{32}, \\
  X_5 = f_{32} + \tfrac{b+2 b'}{3}e_{10}+a\left(a^2+\tfrac{b}{2} - \tfrac{b'}{6} \right) e_{31}.
 \end{split}
 \end{align}
 \end{itemize}
  
 By Proposition \ref{P:algCC}, $T \cdot \kappa = 0$, so $\kappa \in \bbE_{[2\alpha_1+\alpha_2]}$, which is spanned (see Table \ref{F:CurvMod}) by:
 \begin{align}
 \begin{split}
 \kappa_4 &= (f_{10}^* \wedge f_{32}^* - f_{11}^* \wedge f_{31}^*) \otimes h_{01} - f_{10}^* \wedge f_{31}^* \otimes e_{01} - f_{11}^* \wedge f_{32}^* \otimes f_{01}, \\
 \kappa_6 &= (f_{11}^* \wedge f_{31}^* - f_{10}^* \wedge f_{32}^* ) \otimes e_{21} - 2 f_{21}^* \wedge f_{32}^* \otimes e_{10} + 2 f_{21}^* \wedge f_{31}^* \otimes e_{11} + f_{31}^* \wedge f_{32}^* \otimes h_{01}, \\
 \kappa_8 &=  f_{21}^* \wedge f_{31}^* \otimes e_{32} - f_{21}^* \wedge f_{32}^* \otimes e_{31} - f_{31}^* \wedge f_{32}^* \otimes e_{21},\\
 \widetilde\kappa_8 &= f_{21}^* \wedge f_{32}^* \otimes e_{31} + f_{21}^* \wedge f_{31}^* \otimes e_{32},
 \end{split}
 \end{align}
 having weights $r(2\alpha_1 + \alpha_2)$ for $r=2,3,4,4$.  Thus, $\kappa = s \kappa_4 + t \kappa_6 + u \kappa_8 + \tilde{u} \widetilde\kappa_8$. Next, evaluate closure conditions.  Since $\kappa(X_1,X_4) = -s e_{01}$ and 
 $\kappa(X_2,X_5) = -s f_{01}$, we have
 \begin{align}
 [X_1,X_4]_\ff &= [X_1,X_4] - \kappa(X_1,X_4) = (3b+b' + s) e_{01},\\
 [X_2,X_5]_\ff &= [X_2,X_5] - \kappa(X_2,X_5) = (b+3b' + s) f_{01},
 \end{align}
 which lie in $\ff$, so $b' = b = -\tfrac{s}{4}$.  Also, $\kappa(X_2,X_4) = -\kappa(X_1,X_5) = -s T + t e_{21}$, so 
 \begin{align}
 [X_2,X_4]_\ff &= -[X_1,X_5]_\ff = - 6 b T + a X_3 + (\tfrac{4ab}{3} - t) e_{21}
 \end{align}
 lies in $\ff$, hence $t = \tfrac{4ab}{3}$.  Similarly:
 \begin{align}
 \ff \ni [X_3,X_4]_\ff &= -(a^2+3b) X_1 - (u + \tilde{u} + 2a^2b ) e_{32} \qRa u = -\tilde{u} - 2a^2b,\\
 \ff \ni [X_3,X_5]_\ff &= (a^2+3b) X_2 - 2 \tilde{u} e_{31} \qRa \tilde{u} = 0, \\
   [X_4,X_5]_\ff &= a(a^2-b) T - 2bX_3.
 \end{align}
 We can now verify that the Jacobi identity for $[\cdot,\cdot]_\ff$ is satisfied.
 Observe that $\kappa=0$ if and only if $b=0$.  The $\sfD.6$ structure in Table \ref{F:alg-models} is obtained by normalizing $b \neq 0$ to $b=1$ using $\exp(\ad_{\langle \sfZ \rangle})$.

In Table \ref{F:Aut-algmodel}, $\Aut(\ff;\fg,\fp)$ is described (for $b=1$).  Consider $A_\zeta \in \Aut(\fg)$, extended from $A_\zeta \begin{psm} f_{10} \\ f_{11} \end{psm} = \begin{psm} \zeta f_{10} \\ f_{11} \end{psm}$, which induces $(\kappa_H,a,b) \mapsto (\frac{\kappa_H}{\zeta^2}, \frac{a}{\zeta}, \frac{b}{\zeta^2})$.  Setting $\zeta^4=1$ preserves $\kappa_H$, and $a^2$ is the classifying invariant.
 \end{proof}
 
 From the proof above, let us record the limiting $b=0$ case:
 
 \begin{cor} \label{C:b=0} For any $a \in \bbC$, the algebraic model $(\ff;\fg,\fp)$ given below has $\kappa=0$.  (It is non-maximal, i.e. $(\ff;\fg,\fp) < (\fg;\fg,\fp)$, and has type $\sfO$.)
 \begin{align}
 \begin{array}{c@{\quad}c} 
 \ff: \,\begin{array}{l}
 T = h_{01} = -\sfZ_1 + 2\sfZ_2,\\
 X_1 = f_{10} + a e_{11},\\
 X_2 = f_{11} + a e_{10}, \\
 X_3 = f_{21} + a^2 e_{21},\\
 X_4 = f_{31} + a^3 e_{32},\\
 X_5 = f_{32} + a^3 e_{31}.
 \end{array} &
 %%%%%
 \begin{array}{c|cccccc}
 [\cdot,\cdot]_\ff & T & X_1 & X_2 & X_3 & X_4 & X_5\\ \hline
 T & \cdot & X_1 & -X_2 & \cdot & X_4 & -X_5\\
 X_1 &&  \cdot & 3aT - 2X_3 & \,\,\,\,2aX_1 + 3X_4 & \cdot & -aX_3\\
 X_2 && & \cdot & - 2aX_2 - 3X_5 & aX_3  &\cdot\\
 X_3 &&&&\cdot & -a^2 X_1 & a^2 X_2\\
 X_4 &&&&&\cdot & a^3 T\\
 X_5 &&&&&&\cdot
 \end{array}
 %%%%%
 \end{array} 
 \end{align}
 \end{cor}
 
 \begin{cor} \label{C:D6ss} For $\sfD.6_a$, $\ff$ is semisimple if and only if $a^2 \not\in \{ -\frac{9}{4}, 4 \}$, in which case $\ff \cong \fsl(2,\bbC) \times \fsl(2,\bbC)$.
 \end{cor}
 
 \begin{proof}
 The Killing form for $\ff$ has determinant $4096(4a^2+9)^3(a^2-4)^2$.  When nonzero, $\ff$ is 6-dimensional semisimple, so it must be isomorphic to $\fsl(2,\bbC) \times \fsl(2,\bbC)$.
 \end{proof}

%%%%%%%%%%%%%%%%%%%%%%%%%%%%%%%%%%%%%%%%%%%%%%%%%%%%
 \section{The rolling distribution}
 \label{S:rolling}
%%%%%%%%%%%%%%%%%%%%%%%%%%%%%%%%%%%%%%%%%%%%%%%%%%%%
 Consider two 2-spheres in $\bbR^3$ with radii $R \geq r > 0$, rolling on each other without twisting or slipping -- see Figure \ref{F:spheres}.  Let $\rho := \frac{R}{r} \geq 1$.  We picture the larger sphere as fixed and the smaller sphere rolling on the first one.  Any point of the configuration space $M$ is determined by:
 \begin{enumerate}
 \item the unique point of contact of the spheres, i.e.\ a point on $S^2$ (the larger sphere), and
 \item how the smaller sphere is rotated relative to a fixed reference frame attached to the larger sphere.  This corresponds to a point in $\SO(3)$.
 \end{enumerate}
 So we can identify $M$ with the 5-manifold $S^2 \times \SO(3)$.  The space of velocities is constrained by the no twist, no slip conditions.  On the tangent bundle $TM$, these constraints distinguish a rank two distribution $\cD_\rho$ referred to as the {\sl rolling distribution}.  When $\rho = 1$, the $\cD_\rho$ is {\em holonomic}, i.e.\ Frobenius-integrable, while when $\rho > 1$, it is {\em non-holonomic} and has growth $(2,3,5)$.  We refer to Bor--Montgomery \cite{BM2009} for an excellent account of the rolling distribution $(M,\cD_\rho)$.  In this section, we will describe it Cartan-theoretically.
  
 \begin{figure}[h]
 \begin{tikzpicture}
  \shade[ball color = blue!40, opacity = 0.7] (0,0) circle (2cm);
  \draw (0,0) circle (2cm);
  \draw (-2,0) arc (180:360:2 and 0.6);
  \draw[dashed] (2,0) arc (0:180:2 and 0.6);
  \fill[fill=black] (0,0) circle (1pt);
  \draw[dashed] (0,0 ) -- node[above]{$R$} (2,0);
  \shade[ball color = green!40, opacity = 0.4] (3,0) circle (1cm);
  \draw (3,0) circle (1cm);
  \draw (2,0) arc (180:360:1 and 0.4);
  \draw[dashed] (4,0) arc (0:180:1 and 0.4);
  \fill[fill=black] (3,0) circle (1pt);
  \draw[dashed] (3,0 ) -- node[above]{$r$} (4,0);
  \fill[fill=black] (2,0) circle (1pt);
  \draw [purple,thick,domain=25:45,->] plot ({3*cos(\x)}, {3*sin(\x)});
 \end{tikzpicture}
 \caption{Two 2-spheres rolling on each other without twisting or slipping}
 \label{F:spheres}
 \end{figure}
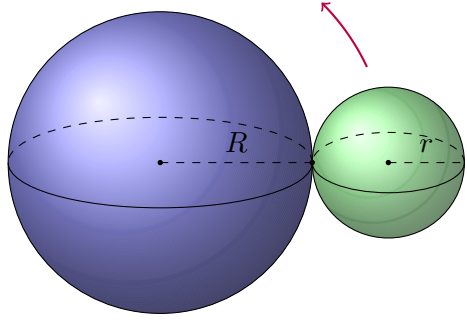

%%%%%%%%%%%%%%%%%%%%%%%%%%%%%%%%%%%%%%%%%%%%%%%%%%%%
 \subsection{The exceptional ratio}
%%%%%%%%%%%%%%%%%%%%%%%%%%%%%%%%%%%%%%%%%%%%%%%%%%%%
 The following is a fascinating symmetry result:

 \begin{theorem} \label{T:rolling}
 Let $\rho > 1$.  The symmetry algebra of the rolling distribution $(M,\cD_\rho)$ is $\fso(3) \times \fso(3)$ when $\rho \neq 3$, and $\Lie(\GS)$ when $\rho = 3$.  (Here, $\GS$ refers to the split real form of $G_2$.)
 \end{theorem}

 This was first observed by Bryant -- see \cite[pg.2]{BM2009} for historical remarks and a brief sketch of Bryant's (unpublished) argument deducing it from Cartan's study in \cite[\S 53]{Car1910}.  Via abnormal extremals, Zelenko \cite[Example 4]{Zel2006} calculated that the Cartan quartic vanishes when $\rho = 3$ (see also Agrachev's survey \cite{Agr2007}).  Bor--Montgomery \cite{BM2009} found an embedding $\fso(3) \times \fso(3) \hookrightarrow \Lie(\GS)$ when $\rho = 3$.  (However, in avoiding Cartan--Tanaka machinery, they could not assert that the symmetry is no larger than $\Lie(\GS)$.) Baez--Huerta \cite{BH2014} explained exceptionality by adopting a global picture of $\GS$ as the symmetries of a ``spinorial ball rolling on a projective plane'' when the ratio of radii is $3:1$.
 
 We show below how our $\sfD.6$ results from \S\ref{S:MT} yield an efficient Cartan-theoretic proof of Theorem \ref{T:rolling}.
%%%%%%%%%%%%%%%%%%%%%%%%%%%%%%%%%%%%%%%%%%%%%%%%%%%%
 \subsection{Lie-theoretic description}
%%%%%%%%%%%%%%%%%%%%%%%%%%%%%%%%%%%%%%%%%%%%%%%%%%%%
 Our starting point is the Lie-theoretic description of $(M,\cD_\rho)$ given in \cite[Sec.3.2]{BM2009}.  Consider $\fso(3)$ with standard brackets
 \begin{align}
 [\bi,\bj] = \bk, \quad  [\bj,\bk] = \bi, \quad  [\bk,\bi] = \bj.
 \end{align}
Define $(\ffr,\ffr^0) = (\fso(3) \times \fso(3),\langle (\bk,\bk) \rangle)$.  Fix $\rho \in \bbR^\times$ and define the $\ffr^0$-invariant subspace \cite[Prop.2]{BM2009}:
 \begin{align}
 \ffr^{-1} = \langle (\bi,-\rho \bi), (\bj,-\rho \bj), (\bk,\bk) \rangle.
 \end{align}
 The following involutions of $\ffr$ preserve $\ffr^0$ and induce transformations of $\rho$:
 \begin{itemize}
 \item Swapping $\fso(3)$ factors induces $\rho \mapsto \frac{1}{\rho}$.  Geometrically, we can swap the spheres.
 \item $(\bi,\bj,\bk) \mapsto (\bi,-\bj,-\bk)$ and $(\bi,\bj,\bk) \mapsto (- \bi, \bj,-\bk)$
 are two involutions of $\fso(3)$.  Using one on each $\fso(3)$ factor induces $\rho \mapsto -\rho$.  This is geometrically induced by a simultaneous orientation-flip on each sphere.  (The tangent spaces for the spheres are modelled on $\ffr / (\ffr^0 + \fso(3))$ and $\ffr / (\fso(3) + \ffr^0)$.)
 \end{itemize}
 Hence, $\rho \geq 1$ suffices.  We also exclude $\rho=1$, for which $\ffr^{-1}$ is a Lie subalgebra.  Thus, \framebox{$\rho > 1$} will suffice.
 
The weak derived flag of $\ffr^{-1}$ is defined via $\ffr^{i-1} = \ffr^i + [\ffr^{-1},\ffr^i]$ for $i \leq -1$, and this yields an $\ffr^0$-invariant filtration $\ffr = \ffr^{-3} \supset \ffr^{-2} \supset \ffr^{-1} \supset \ffr^0$.  Here is an adapted basis:
 \begin{align}
 v_0 = (\bk,\bk), \quad 
 v_1 = (\bi,-\rho \bi), \quad 
 v_2 = (\bj,-\rho \bj), \quad  
 v_3 = (\bk,\rho^2 \bk), \quad
 v_4 = (\bj,-\rho^3 \bj), \quad 
 v_5 = (\bi,-\rho^3 \bi).
 \end{align}
 In particular, $\ffr^{-2} / \ffr^{-1} \equiv \langle v_3 \rangle$ and $\ffr^{-3} / \ffr^{-2} \equiv \langle v_4, v_5 \rangle$.  This completes the Lie-theoretic description of the real homogeneous $(2,3,5)$-structure $(M,\cD_\rho)$.
 
%%%%%%%%%%%%%%%%%%%%%%%%%%%%%%%%%%%%%%%%%%%%%%%%%%%%
 \subsection{A Cartan-theoretic proof of exceptionality}
 
 A {\tt Maple} worksheet supporting this section accompanies our {\tt arXiv} submission.
%%%%%%%%%%%%%%%%%%%%%%%%%%%%%%%%%%%%%%%%%%%%%%%%%%%%
 \subsubsection{Generic case}
 \label{S:gencase}
%%%%%%%%%%%%%%%%%%%%%%%%%%%%%%%%%%%%%%%%%%%%%%%%%%%%
 Consider $\ff = \ffr \otimes_\bbR \bbC$, with corresponding $\ff^0$-invariant filtration $\ff = \ff^{-3} \supset \ff^{-2} \supset \ff^{-1} \supset \ff^0$.  Our goal will be to realize $\ff$ as an algebraic model $(\ff;\fg,\fp)$ of exactly the $\sfD.6_a$ form given in Table \ref{F:alg-models}, where $a \in \bbC$ is to be determined.  Since $\ad_T|_\ff$ from Table \ref{F:alg-models} has (multiplicity two) eigenvalues $0,\pm 1$, while $\ad_{v_0} |_{\ff}$ has (multiplicity two) eigenvalues $0, \pm i$, let us define $T := i v_0$.  Now consider the most general basis of $\ff$ that: (i) extends $T$, (ii) is adapted to the filtration on $\ff$, and (iii) is an eigenbasis for $\ad_T$.  This yields:
 \begin{align} \label{E:gen-filtration}
 \begin{split}
 T := i v_0, \quad
 \begin{array}{l}
 X_1 := s_1(v_1 + i v_2), \\
 X_2 := s_2(v_1 - i v_2),
 \end{array} \quad
 X_3 := s_3 v_3 + t_1 i v_0, \quad
 \begin{array}{l}
 X_4 := s_4(v_4 - i v_5) + t_2 (v_1 + i v_2), \\
 X_5 := s_5(v_4 + i v_5) + t_3 (v_1 - i v_2),
 \end{array}
 \end{split}
 \end{align}
 where $s_i \in \bbC^\times$ and $t_i \in \bbC$ are parameters to be determined.  We now evaluate the $\sfD.6_a$ bracket relations $[\cdot,\cdot]_\ff$ in Table \ref{F:Fstr} using the adapted basis above:
 \begin{align}
 0 &= [X_1,X_2]_\ff - 3aT + 2X_3 = 2(s_3 - i s_1 s_2) X_3 + (2 t_1 - 3 a) T, \label{E:D6fit1}\\
 0 &= [X_1,X_3]_\ff - 2aX_1 - 3 X_4 = -(3 t_2 + (2a+t_1) s_1) X_1 - (3s_4 + s_1 s_3) X_4, \label{E:D6fit2}\\
 0 &= [X_1,X_5]_\ff - 6 T + a X_3 = (2 s_1 s_5 (\rho^2 + 1) - 2 i s_1 t_3  + a s_3) X_3 + (2 i s_1 s_5 \rho^2 + a t_1 - 6) T, \label{E:D6fit3}\\
 0 &= [X_3,X_4]_\ff + (a^2+3) X_1 
 = (s_3 s_4 \rho^2 + a^2 s_1 + t_1 t_2 + 3 s_1) X_1 - (is_3 s_4 (\rho^2+1) - s_3 t_2 - t_1 s_4) X_4. \label{E:D6fit4}
 \end{align}
 Solving these yields the following:
 \begin{align} \label{E:embed-rel}
 \begin{split}
 s_2 &= -\tfrac{5a}{s_1(\rho^2+1)}, \quad
 s_3 = -\tfrac{5 i a }{\rho^2+1}, \quad
 s_4 = \tfrac{5i s_1 a}{3(\rho^2+1)}, \quad
 s_5 = \tfrac{25 i a^2}{3 s_1(\rho^2+1)^2}, \\
 t_1 &= \tfrac{3}{2} a, \quad 
 t_2 = -\tfrac{7}{6} a s_1, \quad
 t_3 = \tfrac{35 a^2}{6 s_1 (\rho^2+1)}, \quad
 \framebox{$a^2 = \tfrac{4(\rho^2+1)^2}{(\rho+3)(\rho-3)(\rho+\tfrac{1}{3})(\rho-\tfrac{1}{3})} =: \cI(\rho)$}.
 \end{split}
 \end{align}
 All remaining $\sfD.6_a$ $[\cdot,\cdot]_\ff$ bracket relations from Table \ref{F:Fstr} are satisfied.  
 The classifying invariant $\cI = a^2$ satisfies $\cI(\rho) = \cI(\rho^{-1}) = \cI(-\rho)$, and is strictly decreasing on $(1,3) \cup (3,\infty)$ with range $(-\infty, -\frac{9}{4}) \cup (4,\infty)$.

This realizes $(\ff;\fg,\fp)$ as a $\sfD.6$ algebraic model, which is {\em maximal} in $(\cM,\leq)$, so $(M,\cD_\rho)$ has precisely 6-dimensional symmetry, i.e.\ $\fso(3) \times \fso(3)$ is the full symmetry algebra when $\rho \in (1,3) \cup (3,\infty)$.
 
%%%%%%%%%%%%%%%%%%%%%%%%%%%%%%%%%%%%%%%%%%%%%%%%%%%%
 \subsubsection{Exceptionality of the $3:1$ ratio}
%%%%%%%%%%%%%%%%%%%%%%%%%%%%%%%%%%%%%%%%%%%%%%%%%%%%
 Let \framebox{$\rho=3$}.  We adopt the above strategy, but use structure equations from Corollary \ref{C:b=0}.  While \eqref{E:D6fit1} and \eqref{E:D6fit2} are the same, \eqref{E:D6fit3} and \eqref{E:D6fit4} become:
 \begin{align}
  0 &= [X_1,X_5]_\ff + a X_3 = (2 s_1 s_5 (\rho^2 + 1) - 2 i s_1 t_3  + a s_3) X_3 + (2 i s_1 s_5 \rho^2 + a t_1) T, \label{E:D6fit3'}\\
 0 &= [X_3,X_4]_\ff + a^2 X_1 = (s_3 s_4 \rho^2 + a^2 s_1 + t_1 t_2) X_1 - (is_3 s_4 (\rho^2+1) - s_3 t_2 - t_1 s_4) X_4. \label{E:D6fit4'}
 \end{align} 
 The solution to \eqref{E:D6fit1}, \eqref{E:D6fit2}, \eqref{E:D6fit3'}, \eqref{E:D6fit4'} is given by the same $s_i, t_j$ (but ignore $a^2$) from \eqref{E:embed-rel}, with $\rho=3$.  (All other brackets are consistent.)  Since $\kappa=0$, we have realized a Lie algebra embedding $\ff \inj \fg$.  This yields the analogous result for the real form $\ffr$, which completes the proof of Theorem \ref{T:rolling}.

%%%%%%%%%%%%%%%%%%%%%%%%%%%%%%%%%%%%%%%%%%%%%%%%%%%%
 \section{Real forms}
 \label{S:real}
%%%%%%%%%%%%%%%%%%%%%%%%%%%%%%%%%%%%%%%%%%%%%%%%%%%%
 We now complement \S\ref{S:rolling} with a broader, systematic study of real forms.  Real $(2,3,5)$-distributions were recently studied by Willse \cite{Willse2019} (by classifying anti-involutions on complex Lie-theoretic data arising from \cite{Car1910}) and Zhitomirskii \cite{Zhi2021} (via normal form methods).  Our goal is to give a Cartan-theoretic presentation of the real classification.
 
  We can study algebraic models in the real setting, but with the complex classification in-hand, we can equivalently classify anti-involutions (Definition \ref{D:AI}) on each (complex) $(\ff;\fg,\fp)$ up to $\Aut(\ff;\fg,\fp)$.

%%%%%%%%%%%%%%%%%%%%%%%%%%%%%%%%%%%%%%%%%%%%%%%%%%%%
 \subsection{Classification summary}
%%%%%%%%%%%%%%%%%%%%%%%%%%%%%%%%%%%%%%%%%%%%%%%%%%%% 
 Start by defining certain automorphisms $A_\lambda, \widetilde{A} \in \Aut(\fg)$ as in Table \ref{F:Autg}.
 
 \begin{footnotesize}
 \begin{table}[h]
 \[
 \begin{array}{|c||@{\,}c@{\,}|@{\,}c@{\,}|@{\,}c@{\,}|@{\,}c@{\,}|@{\,}c@{\,}|@{\,}c@{\,}|@{\,}c@{\,}|@{\,}c@{\,}|@{\,}c@{\,}|@{\,}c@{\,}|@{\,}c@{\,}|@{\,}c@{\,}|@{\,}c@{\,}|@{\,}c|} \hline
 A & A(f_{32}) & A(f_{31}) & A(f_{21}) & A(f_{11}) & A(f_{10}) & A(f_{01}) & A(\sfZ_1) & A(\sfZ_2) & A(e_{01}) & A(e_{10}) & A(e_{11}) & A(e_{21}) & A(e_{31}) & A(e_{32}) \\ \hline\hline
 A_\lambda & \lambda f_{32} & \lambda^2 f_{31} & \lambda f_{21} & f_{11} & \lambda f_{10} & \frac{1}{\lambda} f_{01} & \sfZ_1 & \sfZ_2 & \lambda e_{01} & \frac{1}{\lambda} e_{10} & e_{11} & \frac{1}{\lambda} e_{21} & \frac{1}{\lambda^2} e_{31} & \frac{1}{\lambda} e_{32}\\ \hline
 \widetilde{A} & f_{31} & f_{32} & -f_{21} & f_{10} & f_{11} & e_{01} & \sfZ_1 & \sfZ_1 - \sfZ_2 & f_{01} & e_{11} & e_{10} & -e_{21} & e_{32} & e_{31} \\ \hline
 \end{array}
 \]
 \caption{Some automorphisms of $\fg = \Lie(G_2)$}
 \label{F:Autg}
 \end{table}
 \end{footnotesize}
 
 We use these to describe $\Aut(\ff;\fg,\fp)$ in Table \ref{F:Aut-algmodel}.  (Let $\fder(\ff;\fg,\fp)$ denote its Lie algebra and recall that all derivations of $\fg$ are inner.)
 
 \begin{footnotesize}
 \begin{table}[h]
 \[
 \begin{array}{|c|c|c|c|}\hline
 \mbox{Label} & \fder(\ff;\fg,\fp) & \begin{array}{c} \mbox{Generators for} \\ \Aut(\ff;\fg,\fp) / \exp(\fder(\ff;\fg,\fp)) \end{array} & \mbox{Remarks}\\ \hline\hline
 \sfN.7_c & \ad_{\sfZ_2}, \, \ad_{f_{01}} & \begin{cases} A_{-1}, & c \neq 0;\\ A_{-1},\, A_{i}, & c = 0 \end{cases} & 
 c \stackrel{A_{i}}{\mapsto} -c \\ \hline
 \sfN.6 & \ad_{f_{01}} & A_{-1} & \cdot \\ \hline
 \sfD.6_a & \ad_{-\sfZ_1+2\sfZ_2} &
 \begin{cases}
 \widetilde{A}, & a \neq 0;\\
 \widetilde{A}, \, A_{-1}, & a = 0
 \end{cases} & a \stackrel{A_{-1}}{\mapsto} -a \\ \hline
 \end{array}
 \]
 \caption{Classification of $\Aut(\ff;\fg,\fp)$}
 \label{F:Aut-algmodel}
 \end{table}
 \end{footnotesize}

 We will use $\Aut(\ff;\fg,\fp)$ to classify anti-involutions of $(\ff;\fg,\fp)$ up to conjugacy (Definition \ref{D:AI}).  To state the result, we first define some anti-involutions of $\fg$ in Table \ref{F:AI-relevant}.
  
 \begin{footnotesize}
 \begin{table}[h]
 \[
 \begin{array}{|c||@{\,}c@{\,}|@{\,}c@{\,}|@{\,}c@{\,}|@{\,}c@{\,}|@{\,}c@{\,}|@{\,}c@{\,}|@{\,}c@{\,}|@{\,}c@{\,}|@{\,}c@{\,}|@{\,}c@{\,}|@{\,}c@{\,}|@{\,}c@{\,}|@{\,}c@{\,}|@{\,}c|} \hline
 \psi & \psi(f_{32}) & \psi(f_{31}) & \psi(f_{21}) & \psi(f_{11}) & \psi(f_{10}) & \psi(f_{01}) & \psi(\sfZ_1) & \psi(\sfZ_2) & \psi(e_{01}) & \psi(e_{10}) & \psi(e_{11}) & \psi(e_{21}) & \psi(e_{31}) & \psi(e_{32}) \\ \hline\hline
 \psi_\zeta & \zeta^3 f_{32} & \zeta^3 f_{31} & \zeta^2 f_{21} & \zeta f_{11} & \zeta f_{10} & f_{01} & \sfZ_1 & \sfZ_2 & e_{01} & \frac{1}{\zeta} e_{10} & \frac{1}{\zeta} e_{11} & \frac{1}{\zeta^2} e_{21} & \frac{1}{\zeta^3} e_{31} & \frac{1}{\zeta^3} e_{32}\\ \hline
 \widetilde\psi_\zeta & \zeta^3 f_{31} & \zeta^3 f_{32} & -\zeta^2 f_{21} & \zeta f_{10} & \zeta f_{11} & e_{01} & \sfZ_1 & \sfZ_1 - \sfZ_2 & f_{01} & \frac{1}{\zeta} e_{11} & \frac{1}{\zeta} e_{10} & -\frac{1}{\zeta^2} e_{21} & \frac{1}{\zeta^3} e_{32} & \frac{1}{\zeta^3} e_{31} \\ \hline
 \tau_\zeta & \zeta f_{32} & f_{31} & \zeta f_{21} & f_{11} & \zeta f_{10} & \zeta f_{01} & \sfZ_1 & \sfZ_2 & \zeta e_{01} & \zeta e_{10} & e_{11} & \zeta e_{21} & e_{31} & \zeta e_{32}\\ \hline
 \end{array}
 \]
 \caption{Relevant anti-involutions of $\fg = \Lie(G_2)$ ($\zeta^4=1$; $\bbC$-antilinearly extended).}
 \label{F:AI-relevant}
 \end{table}
 \end{footnotesize}
 
 \begin{theorem} The Cartan-theoretic classification of non-flat real multiply-transitive $(2,3,5)$-structures is given in
 Table \ref{F:AI}.
 \end{theorem}
 
 \begin{footnotesize}
 \begin{table}[h]
 \[
 \begin{array}{|c|c|c|l|c|} \hline
 \mbox{Label} & \begin{tabular}{c} Existence of\\ anti-involutions \end{tabular} & \begin{tabular}{c} Parameter\\ range \end{tabular} & \mbox{Inequivalent anti-involutions $\psi$} & \mbox{Redundancy}\\ \hline\hline
 \sfN.7_c & c^2 \in \bbR & \begin{array}{c} c \in (0,\infty) \\ c \in i(0,\infty) \\ c = 0 \end{array} & 
 \begin{array}{l}
 \psi_1, \quad \psi_{-1}\\
 \psi_i, \quad \psi_{-i}\\
 \psi_1, \quad \psi_i
 \end{array} & \begin{array}{c} (c,\psi_\zeta) \stackrel{A_i}{\mapsto} (-c,\psi_{-\zeta})\end{array}\\ \hline
 \sfN.6 & \checkmark & \cdot & \begin{array}{l} \tau_1,\quad \tau_{-1} \end{array} & \cdot \\ \hline
 \sfD.6_a & a^2 \in \bbR & \begin{array}{c} a\in (0,\infty) \\ a\in i(0,\infty) \\ a=0 \end{array} & 
 \begin{array}{l}
 \psi_1, \quad \widetilde\psi_{1}, \quad \widetilde\psi_{-1}\\ 
 \psi_i, \quad \widetilde\psi_{i}, \quad \widetilde\psi_{-i}\\ 
 \psi_1, \quad \psi_i, \quad \widetilde\psi_{1}, \quad \widetilde\psi_{i}
 \end{array} &
 \begin{array}{l}
 (a,\psi_\zeta,\widetilde\psi_\zeta) 
 \stackrel{A_{-1}}{\mapsto} (-a,\psi_\zeta,\widetilde\psi_{-\zeta})
 \end{array}\\ \hline
 \end{array}
 \]
 \caption{Classification of non-flat real multiply-transitive $(2,3,5)$-distributions}
 \label{F:AI}
 \end{table}
 \end{footnotesize}
 
 In the next sections, we only give details for the type $\sfD$ cases, leaving calculations for the type $\sfN$ cases as an exercise.  For each anti-involution $\psi$ of $(\ff;\fg,\fp)$, the underlying Lie-theoretic structure $(\ff^\psi,[\cdot,\cdot]_{\ff^\psi})$ for the (real) fixed point Lie algebra $\ff^\psi$ can be directly computed (as in Example \ref{X:real}).  The most interesting cases concern real forms for $\sfD.6$ structures, whose isomorphisms types are  organized in Table \ref{F:D6AI}.

 \begin{footnotesize}
 \begin{table}[h]
  \[
 \begin{array}{|c|c|c|c|c|} \hline
 && \multicolumn{3}{c|}{\mbox{Isomorphism type of $\ff^\psi$ when $a \in \bbR_{\geq 0}$}}\\ \hline
 \psi & \mbox{$\bbR$-basis of $\ff^\psi$} & a \in [0,2) & a=2 & a > 2\\ \hline\hline
 \psi_1 & T, X_1, X_2, X_3, X_4, X_5 & 
 \fsl(2,\bbR) \times \fsl(2,\bbR) & \fsl(2,\bbR) \times \fe(1,1) & \fsl(2,\bbR) \times \fsl(2,\bbR)\\ \hline
 %%%%%%%%%%%%%%%%%%%%%%%%
 \widetilde\psi_{1} & \begin{array}{l} iT,\, X_1 + X_2, \, i(X_1 - X_2), \\ i X_3,\, X_4 + X_5,\, i(X_4 - X_5) \end{array} & 
 \fsl(2,\bbR) \times \fso(3) & 
 \fsl(2,\bbR) \times \fe(2) & 
 \fsl(2,\bbR) \times \fsl(2,\bbR)\\ \hline
 %%%%%%%%%%%%%%%%%%%%%%%%
 \widetilde\psi_{-1} & 
 \begin{array}{l}
 iT,\, X_1 - X_2,\, i(X_1+X_2), \\
 iX_3,\, X_4 - X_5,\, i(X_4 + X_5)
 \end{array} & 
 \underset{\mbox{\tiny (omit when $a = 0$)}}{\fsl(2,\bbR) \times \fso(3)} & 
 \fso(3) \times \fe(2) &
 \fso(3) \times \fso(3) \\ \hline\hline
 %%%%%%%%%%%%%%%%%%%%%%%%
 && \multicolumn{3}{c|}{\mbox{Isomorphism type of $\ff^\psi$ when $a \in i\bbR_{\geq 0}$}}\\ \hline
 \psi & \mbox{$\bbR$-basis of fixed point set $\ff^\psi$} & 
 a \in i[0,\frac{3}{2}) & a = \frac{3}{2} i & a \in i (\frac{3}{2},\infty) \\ \hline
 \psi_i & 
  \begin{array}{l} T, \,(1+i) X_1,\, (1+i) X_2,\\
  iX_3,\, (1-i) X_4,\, (1-i) X_5 \end{array}
  & 
 \fso(1,3) & 
 \fe(1,2) & 
 \fsl(2,\bbR) \times \fsl(2,\bbR) \\ \hline
 %%%%%%%%%%%%%%%%%%%%%%%%
  \widetilde\psi_{i} & 
  \begin{array}{l}
  iT,\, X_1 + iX_2,\, X_2 + iX_1,\\
  X_3,\, X_4 - iX_5,\, X_5 - iX_4
  \end{array} & 
  \fso(1,3) & 
  \fe(3) &
  \fso(3) \times \fso(3) \\ \hline
 %%%%%%%%%%%%%%%%%%%%%%%%
  \widetilde\psi_{-i} & \begin{array}{l}
  iT,\, X_1 - iX_2,\, X_2 - iX_1,\\
  X_3,\, X_4 + iX_5,\, X_5 + iX_4
  \end{array} & 
  \underset{\mbox{\tiny (omit when $a = 0$)}}{\fso(1,3)} &
%  \begin{array}{c} \fso(1,3)\\[-0.1in]{\tiny \mbox{(omit when $a = 0$)}} \end{array} &
  \fe(1,2) & \fsl(2,\bbR) \times \fsl(2,\bbR) \\ \hline
  \end{array}
  \]
  \caption{Alternative classification of $\sfD.6_a$ anti-involutions.}  %Redundancy: $(a,\psi_\zeta^-) \stackrel{A}{\mapsto} (-a,\psi_{-\zeta}^-)$.}
  \label{F:D6AI}
  \end{table}
  \end{footnotesize}
  
%%%%%%%%%%%%%%%%%%%%%%%%%%%%%%%%%%%%%%%%%%%%%%%%%%%%
 \subsection{Anti-involutions for $\sfD.6$ structures} 
 \label{S:D6AI}
%%%%%%%%%%%%%%%%%%%%%%%%%%%%%%%%%%%%%%%%%%%%%%%%%%%%
 \subsubsection{Abstract classification result}
%%%%%%%%%%%%%%%%%%%%%%%%%%%%%%%%%%%%%%%%%%%%%%%%%%%%
 \begin{theorem} \label{T:D6AI} $\sfD.6_a$ admits an anti-involution if and only if $a^2 \in \bbR$. 
 Since $\sfD.6_a \cong \sfD.6_{-a}$, we may restrict to $a \in \bbR_{\geq 0} \cup i \bbR_{\geq 0}$.  The complete list of inequivalent anti-involutions is given in Table \ref{F:AI}.
 \end{theorem}
 
 \begin{proof}
 Let $\psi$ be an anti-involution of $\sfD.6_a$.  We start by examining $\ff^0 = \langle T \rangle$, where $T = h_{01} = -\sfZ_1 + 2\sfZ_2$.  The eigenspaces of $\ad_T \in \tEnd(\fg)$ are:
 \begin{align}
 \begin{array}{|c|ccccc|} \hline
 \mbox{$\ad_T$-eigenvalue } r & -2 & -1 & 0 & 1 & 2\\ \hline
 \mbox{$\ad_T$-eigenspace } \cE_r & \langle f_{01} \rangle & \langle f_{32}, f_{11}, e_{10} ,e_{31} \rangle &  \langle f_{21}, T, h_{21}, e_{21} \rangle & \langle f_{31}, f_{10}, e_{11}, e_{32} \rangle & \langle e_{01} \rangle \\ \hline
 \end{array}
 \end{align}
  Since $\psi(\ff^0) = \ff^0$, then $\psi(T) = \varepsilon T$.  The $\ad_T$-eigenvalues are all real, so $\ad_{\psi(T)}$ must have the same set of eigenvalues, so $\varepsilon = \pm 1$ and $\psi(\cE_r) = \cE_{\varepsilon r}$ (in a filtration-preserving manner).

 Since $\psi(T) = \varepsilon T$, then $\cE_0 = \langle f_{21}, T, h_{21} := \sfZ_1, e_{21} \rangle$ is $\psi$-invariant, and so is the induced filtration.  In particular, $\psi(e_{21}) = s_2 e_{21}$, so evaluating $\psi$ on the standard $\fsl_2$-relations among $\{ h_{21}, e_{21}, f_{21} \}$, we obtain $\psi(h_{21}) = h_{21}+ u e_{21}$ and $\psi(f_{21}) = \tfrac{1}{s_2} (f_{21} - \tfrac{u}{2} h_{21} - \tfrac{u^2}{4} e_{21})$.   
Now $\psi$-invariance of $\ff^{-2} = \langle T, X_1, X_2, X_3 \rangle$ implies that $X_3 = f_{21} + (a^2+1) e_{21}$ has $\psi(X_3) = \psi(f_{21}) + (\bar{a}^2+1) \psi(e_{21}) \in \ff^{-2}$, so $\psi(X_3)$ must be a multiple of $X_3$, which forces $u=0$.  Thus, $\psi$ fixes the grading element $\sfZ_1$, and each $\fg_k$ is $\psi$-invariant.  It suffices to study $\psi|_{\fg_{-1}}$, since $\psi$ is uniquely determined from it.  Noting that $\ad_T\begin{psmallmatrix} f_{10} \\ f_{11} \end{psmallmatrix} = \begin{psmallmatrix} +f_{10} \\ -f_{11} \end{psmallmatrix}$, we have:

 \begin{itemize}
 \item \underline{$\varepsilon=1$}: $\psi$ preserves both $\langle f_{10} \rangle$ and $\langle f_{11} \rangle$.  Write $\psi\begin{psmallmatrix} f_{10} \\ f_{11} \end{psmallmatrix} = \begin{psmallmatrix} c_1 f_{10} \\ c_2 f_{11} \end{psmallmatrix}$, where $|c_1|=|c_2|=1$ since $\psi^2 = \id$.  
% Also, $\psi$ induces $\kappa_H \mapsto \frac{1}{(c_1c_2)^2}\kappa_H$, and since $\psi$ preserves $\kappa_H$, then $1 = (c_1 c_2)^2$.
Conjugating $\psi$ by $\exp(t\, \ad_T)$ induces $(c_1,c_2) \mapsto (\frac{\sigma}{\bar\sigma} c_1, \frac{\bar\sigma}{\sigma} c_2)$, so $\frac{c_1}{c_2} \mapsto \frac{\sigma^2}{\bar\sigma^2} \frac{c_1}{c_2}$, where $\sigma = \exp(t)$, so we may normalize $c_1 = c_2 =: \zeta$, which is preserved for $\sigma^2 \in \bbR^\times$.
%Set $\zeta := c_1 = c_2$, so \framebox{$\zeta^4 = 1$} from above.  Using $\sigma = i$, we may normalize to either \framebox{$\zeta = 1$ or $i$}.
Now $X_1 = f_{10} + a e_{11} + e_{32} \in \ff^{-1}$, so $\psi(X_1) = \zeta (f_{10} + \frac{\bar{a}}{\zeta^2} e_{11} + \frac{1}{\zeta^4} e_{32}) \in \ff^{-1}$.  Thus, \framebox{$\frac{\bar{a}}{\zeta^2} = a$, \, $\zeta^4=1$}, hence \framebox{$\bar{a}^2 = a^2 \in \bbR$}.  Finally, using $\sigma = i$, we may normalize to $\zeta = 1$ or $i$.

 \item \underline{$\varepsilon=-1$}: $\psi$ swaps $\langle f_{10} \rangle$ and $\langle f_{11} \rangle$.  Write $\psi(f_{10}) = \zeta f_{11}$, and so $f_{10} = \psi^2 = \bar\zeta\psi(f_{11})$.  Conjugating by $\exp(t\,\ad_T)$ induces $\zeta \mapsto |\sigma|^2 \zeta$, where $\sigma = \exp(-t)$, so we may assume that $|\zeta|=1$, hence $\psi(f_{11}) = \zeta f_{10}$.   Now $X_1 = f_{10} + a e_{11} + e_{32} \in \ff^{-1}$, so $\psi(X_1) = \zeta (f_{11} + \frac{\bar{a}}{\zeta^2} e_{10} + \frac{1}{\zeta^4} e_{31}) \in \ff^{-1}$ must be a multiple of $X_2 = f_{11} + a e_{10} + e_{31} \in \ff^{-1}$.  Thus, \framebox{$\frac{\bar{a}}{\zeta^2} = a$,\, $\zeta^4 = 1$}, hence \framebox{$\bar{a}^2 = a^2 \in \bbR$}.  Conjugating $\psi$ by $A_{-1} \in \Aut(\fg)$ (induced from
 $A_{-1} \begin{psm}
 f_{10} \\ f_{11}
 \end{psm} = \begin{psm}
 -f_{10} \\ f_{11}
 \end{psm}$) induces $(a,\zeta) \mapsto (-a,-\zeta)$.
 \end{itemize}
 \end{proof}
 
%%%%%%%%%%%%%%%%%%%%%%%%%%%%%%%%%%%%%%%%%%%%%%%%%%%%
 \subsubsection{Identification of $\sfD.6$ real forms}
%%%%%%%%%%%%%%%%%%%%%%%%%%%%%%%%%%%%%%%%%%%%%%%%%%%%
 When $a^2 \in \bbR$, the anti-involutions $\psi$ from Theorem \ref{T:D6AI} act on our chosen basis $\{ T, X_1, X_2, X_3, X_4, X_5 \}$ of $\ff$ in a simple manner:
 \begin{align}
 \begin{array}{|c|c|c|c|c|c|c|} \hline
 \psi & \psi(T) & \psi(X_1) & \psi(X_2) & \psi(X_3) & \psi(X_4) & \psi(X_5) \\ \hline\hline
 \psi_\zeta & T & \zeta X_1 & \zeta X_2 & \zeta^2 X_3 & \zeta^3 X_4 & \zeta^3 X_5 \\ \hline
 \widetilde\psi_\zeta  & -T & \zeta X_2 & \zeta X_1 & -\zeta^2 X_3 & \zeta^3 X_5 & \zeta^3 X_4\\ \hline
 \end{array}
 \end{align}
 For each anti-involution $\psi$, we can examine the (real) fixed point Lie algebra $\ff^\psi := \{ x \in \ff : \psi(x) = x\}$.  From Corollary \ref{C:D6ss}, we know that $\ff \cong \fsl(2,\bbC) \times \fsl(2,\bbC)$ if $a^2 \in \bbR \backslash \{ 4, -\frac{9}{4} \}$, and $\ff^\psi$ can be identified from its Killing form signature $[p,q,r]$ (denoting the dimensions of the maximal positive, negative, null subspaces). Letting $\fe(p,q) := \fso(p,q) \ltimes \bbR^{p,q}$ be the Euclidean Lie algebra in signature $(p,q)$, we have:
 \begin{align}
 &\begin{array}{|c|c|c|c|c} \hline
 [0,6,0] & [2,4,0] & [3,3,0] & [4,2,0]\\ \hline
 \fso(3) \times \fso(3) & \fsl(2,\bbR) \times \fso(3) & \fso(1,3) & \fsl(2,\bbR) \times \fsl(2,\bbR) \\ \hline
 \end{array}\\
 &\begin{array}{|c|c|c|c|c|} \hline
 [3,1,2] & [2,2,2] & [0,4,2] & [2,1,3] & [0,3,3]\\ \hline
 \fsl(2,\bbR) \times \fe(1,1) & \fsl(2,\bbR) \times \fe(2) & \fso(3) \times \fe(2) & \fe(1,2) & \fe(3) \\ \hline
 \end{array}
 \end{align}

 \begin{prop} \label{P:altAI}
 For $\sfD.6$ structures, Table \ref{F:D6AI} gives an alternative complete classification of all anti-involutions.
 \end{prop}

 The proof of Proposition \ref{P:altAI} is a straightforward check using {\tt Maple}, where we have used that the Killing form of $\ff^\psi$ has determinant that is a continuous real-valued function of the parameter $a$, is non-zero for $a^2 \not\in \{ 4, -\frac{9}{4} \}$, so its signature is constant for $a$ lying in each of the intervals $[0,2)$, $(2,\infty)$, $i[0,\frac{3}{2})$, $i(\frac{3}{2},\infty)$.  Thus, it suffices to test a representative value from each interval.

 \begin{example} \label{X:real} 
 Consider the $\widetilde\psi_i$ case on the interval $a \in i[0,\frac{3}{2})$.  For $a=0$, the following {\tt Maple} code returns the signature $[3,3,0]$, so we conclude $\ff^\psi \cong \fso(1,3)$ for any $a \in i[0,\frac{3}{2})$.

 \begin{verbatim} 
restart: with(DifferentialGeometry): with(LieAlgebras): with(Tensor):
br:=[[t,x1]=x1, [t,x2]=-x2, [t,x4]=x4, [t,x5]=-x5,
[x1,x2]= -2*x3 + 3*a*t, [x1,x3]= +3*x4 + 2*a*x1, [x2,x3]= -3*x5 - 2*a*x2,
[x1,x5]= -a*x3 + 6*t, [x2,x4]= +a*x3 - 6*t, [x3,x4]= -(a^2+3)*x1,  
[x3,x5]= +(a^2+3)*x2, [x4,x5]= -2*x3 + (a^3-a)*t]:
LD:=LieAlgebraData(br,[t,x1,x2,x3,x4,x5],alg):
DGsetup(LD,[t,x1,x2,x3,x4,x5],[seq(omega||i,i=0..5)]):
a:=0: fp:=[I*t,x1+I*x2,x2+I*x1,x3,x4-I*x5,x5-I*x4]: # Enter (a,psi) data here.
LD2:=LieAlgebraData(fp): DGsetup(LD2):
map(nops,QuadraticFormSignature(KillingForm()));
 \end{verbatim}
 \end{example}

 From Proposition \ref{P:altAI}, we deduce the {\em real} Cartan-theoretic classification of the rolling sphere distribution:

 \begin{theorem} \label{T:rolling-classify}
 Let $\rho > 1$.  Let $a \in \bbR_{\geq 0} \cup i\bbR_{\geq 0}$ with $\cI = a^2 = \tfrac{4(\rho^2+1)^2}{(\rho+3)(\rho-3)(\rho+\tfrac{1}{3})(\rho-\tfrac{1}{3})} \in \bbR$.  The (real) rolling ball distribution $(M,\cD_\rho)$ corresponds to the  (real) algebraic model determined by:
 \begin{itemize}
 \item the (complex) $\sfD.6_a$ model, and
 \item the anti-involution $\widetilde\psi_{-1}$ if $\rho > 3$ (so $a > 2$), or $\widetilde\psi_i$ if $\rho \in (1,3)$ (so $a \in i(\tfrac{3}{2},\infty)$).
 \end{itemize}
 \end{theorem}

 \begin{proof} The formula for $a^2$ was found in \eqref{E:embed-rel} when identifying the (complex) $\sfD.6_a$ model $(\ff;\fg,\fp)$.  From Table \ref{F:D6AI}, there are only two possibilities for $\fso(3) \times \fso(3)$ arising as a real form, and the indicated ranges $a \in i(\frac{3}{2},\infty) \cup (2,\infty)$ are covered for $\rho \in (1,3) \cup (3,\infty)$.  The result immediately follows.
 \end{proof}

 \begin{footnotesize}
 \begin{table}[h]
 \[
 \begin{array}{|l|c|c|c|c|} \hline
 \multicolumn{3}{|c}{\mbox{Willse's classification}} & \multicolumn{2}{|c|}{\mbox{Our classification}}\\ \hline
 \mbox{Label} & \mbox{Lie algebra} & \mbox{$\lambda$ domain} & \mbox{Anti-involution} & \mbox{$a^2 = \frac{4(\lambda+1)^2}{(\lambda-9)(\lambda-\frac{1}{9})}$ range}\\ \hline\hline
 {\rm {\bf  D.6}}{}^{2-}_\lambda & \fsl(2,\bbR) \times \fsl(2,\bbR) & [-1,0) & \psi_1 & [0,4) \\
 && (0,\frac{1}{9}) & \psi_1 & (4,\infty)\\
 && (\frac{1}{9},1) & \psi_i & (-\infty,-\frac{9}{4}) \\
 %%%%%%
 {\rm {\bf  D.6}}{}^{2+}_\lambda & \fsl(2,\bbR) \times \fsl(2,\bbR) & (0,\frac{1}{9}) & \widetilde\psi_1 &(4,\infty)\\
 && (\frac{1}{9},1) & \widetilde\psi_{-i} & (-\infty,-\frac{9}{4})\\
 %%%%%%
 {\rm {\bf  D.6}}{}^{4}_\lambda & \fsl(2,\bbR) \times \fso(3) & (-\infty,-1) & \widetilde\psi_{-1} & (0,4)\\
 & & [-1,0) & \widetilde\psi_1 & [0,4)\\
 {\rm {\bf D.6}}{}^6_\lambda & \fso(3) \times \fso(3) & (0,\frac{1}{9}) & \widetilde\psi_{-1} & (4,\infty)\\
 & \fso(3) \times \fso(3) & (\frac{1}{9},1) & \widetilde\psi_i & (-\infty,-\frac{9}{4})\\
 {\rm {\bf D.6}}{}^{3-}_{-1} & \fso(1,3) & - & \psi_{i} & 0\\
 {\rm {\bf D.6}}{}^{3+}_{-1} & \fso(1,3) & - & \widetilde\psi_{i} & 0\\ \hline
 %%%%%%
 {\rm {\bf  D.6}}{}^{1}_\infty & \fsl(2,\bbR) \times \fe(1,1) & - & \psi_1 & 4\\
 {\rm {\bf  D.6}}{}^{2}_\infty & \fsl(2,\bbR) \times \fe(2) & - & \widetilde\psi_1 & 4\\
 {\rm {\bf  D.6}}{}^{4}_\infty & \fso(3) \times \fe(2) & - & \widetilde\psi_{-1} & 4\\ \hline
 {\rm {\bf  D.6}}{}^{1-}_* & \fe(1,2) & - & \psi_i & -\frac{9}{4}\\
 {\rm {\bf  D.6}}{}^{1+}_* & \fe(1,2) & - & \widetilde\psi_{-i} & -\frac{9}{4}\\
 {\rm {\bf  D.6}}{}^{3}_* & \fe(3) & - & \widetilde\psi_i & -\frac{9}{4}\\ \hline
 \end{array}
 \]
 \caption{Comparison with Willse \cite[Table 8]{Willse2019}.  (Take $a$ in the ranges of Table \ref{F:D6AI}.)}
 \label{F:Willse}
 \end{table}
 \end{footnotesize}

 In Table \ref{F:Willse}, we confirm almost all of Willse's classification \cite[Tables 8]{Willse2019}.  Minor corrections:
 \begin{itemize}
 \item ${\rm {\bf  D.6}}{}^{6}_\lambda$: $\fd/\fk$ should be $\bA + \sqrt{\lambda} \bA',\,\,\bB + \sqrt{\lambda} \bB'$.
 \item ${\rm {\bf  D.6}}{}^{3\pm}_{-1}$: we should have the brackets ${\bf [D_A,D_B] = -C, \, [D_B,D_C] = -A,\, [D_C,D_A] = -B}$.
 \end{itemize}
 For $\fso(1,3)$-models, we agree for $a = 0$ (two inequivalent models), but we additionally found three inequivalent $\fso(1,3)$-models for each $a \in i(0,\tfrac{3}{2})$.  (Compare also with \cite[p.45-46]{Zhi2021}.)  These missing models are described in \S\ref{S:so13}.  Generating Table \ref{F:Willse} follows similar techniques as in \S\ref{S:so13}, so we omit these details here.

%%%%%%%%%%%%%%%%%%%%%%%%%%%%%%%%%%%%%%%%%%%%%%%%%%%%
 \subsection{Real $(2,3,5)$-distributions with $\fso(1,3)$ symmetry} 
 \label{S:so13}
%%%%%%%%%%%%%%%%%%%%%%%%%%%%%%%%%%%%%%%%%%%%%%%%%%%%
 While Lie-theoretic descriptions underlying $\fso(1,3)$-invariant $(2,3,5)$-structures can be deduced from Table \ref{F:D6AI}, let us use more familiar bases.  Let
 \begin{itemize}
 \item $\{ \sfH,\sfX,\sfY \}$ be a standard $\fsl_2$-triple with brackets $[\sfH,\sfX] = 2\sfX, \, [\sfH,\sfY] = -2\sfY,\, [\sfX,\sfY] = \sfH$;
 \item $\{ \sfA, \sfB, \sfC \}$ be a standard $\fso(3)$-basis with bracket $[\sfA,\sfB] = \sfC, \, [\sfB,\sfC] = \sfA, \, [\sfC,\sfA] = \sfB$.
 \end{itemize}
 For $\ffr = \fso(1,3) \cong \fsl(2,\bbC)_\bbR$, take real bases $\{ \sfH,\sfX,\sfY, \sfI\sfH,\sfI\sfX,\sfI\sfY\}$ or $\{ \sfA,\sfB,\sfC, \sfI\sfA,\sfI\sfB,\sfI\sfC\}$, where $\sfI\sfH := i\sfH$, etc.

 \begin{theorem} \label{T:so13}
 The classification of real $(2,3,5)$-distributions with symmetry $\ffr = \fso(1,3)$ is given by:
 \begin{align}
 \begin{array}{|c|c|c|c|c|} \hline
 \ffr^0 & \ffr^{-1}/\,\ffr^0 & \mbox{Restrictions / redundancy} & a^2 \mbox{ (for $\sfD.6_a$)} & \psi\\ \hline\hline
 \sfH & \begin{array}{l} \sfX + \alpha\, \sfI\sfX,\\ \sfY + \alpha\, \sfI\sfY \end{array} & \begin{array}{c} \alpha \in \bbR^\times\\ (\alpha \sim -\alpha \sim \frac{1}{\alpha})\end{array} & -\frac{9(\alpha^2 - 1)^2}{(\alpha^2+4)(4\alpha^2+1)} \in (-\frac{9}{4},0] & \psi_i\\ \hline
 \sfC & \begin{array}{l} \sfA+ \alpha\,\sfI\sfA, \\ 
 \sfB + \alpha\,\sfI\sfB \end{array} & \begin{array}{c} \alpha \in \bbR^\times\\ (\alpha \sim -\alpha)\end{array} & -\frac{9(\alpha^2 - 1)^2}{(\alpha^2+4)(4\alpha^2+1)} \in (-\frac{9}{4},0] & \begin{array}{c} \begin{cases} \widetilde\psi_i, & 0 < \alpha^2 \leq 1,\\
 \widetilde\psi_{-i}, & \alpha^2 > 1,
 \end{cases} \\ \mbox{assuming } a \in i[0,\frac{3}{2}).\end{array}\\ \hline
 \end{array}
 \end{align}
 \end{theorem}

 \begin{proof} The indicated redundancies are induced by the following automorphisms $A \in \Aut(\ffr)$ preserving $\ffr^0$:
 \begin{itemize}
 \item $A(\sfH,\sfX,\sfY,\sfI\sfH,\sfI\sfX,\sfI\sfY) = (\sfH,\sfX,\sfY,-\sfI\sfH,-\sfI\sfX,-\sfI\sfY)$ and $A(\sfA,\sfB,\sfC,\sfI\sfA,\sfI\sfB,\sfI\sfC) = (\sfA,\sfB,\sfC,-\sfI\sfA,-\sfI\sfB,-\sfI\sfC)$ induce $\alpha \sim -\alpha$.
 \item $A(\sfH,\sfX,\sfY,\sfI\sfH,\sfI\sfX,\sfI\sfY) = (-\sfH,\sfI\sfY,-\sfI\sfX,\sfI\sfH,\sfY,-\sfX)$ induces $\alpha \sim \frac{1}{\alpha}$.
 \end{itemize}
 
 Now proceed as in \S\ref{S:gencase}.  To identify $a^2$, let $\ff = \ffr \otimes_\bbR \bbC$ and specify $\ff = \ff^{-3} \supset \ff^{-2} \supset \ff^{-1} \supset \ff^0$ via an adapted basis $(v_0,v_1,v_2,v_3,v_4,v_5)$ with $\ff^0 = \langle v_0 \rangle$, $\ff^{-1} / \ff^0 = \langle v_1, v_2 \rangle$, $\ff^{-2} / \ff^{-1} = \langle v_3 \rangle$, and $\ff^{-3} / \ff^{-2} = \langle v_4, v_5 \rangle$.
 \begin{align}
 \begin{array}{|c|c|c|} \hline
 \mbox{Case} & (v_0, v_1,v_2,v_3,v_4,v_5) & (T,X_1,X_2,X_3,X_4,X_5) \mbox{ ansatz}\\ \hline\hline
 \mbox{(i)} & \begin{array}{c}
 (\frac{\sfH}{2},\, \sfX + \alpha\,\sfI\sfX,\, \sfY + \alpha\,\sfI\sfY,\, \\
 \sfI\sfH,\, \alpha\sfX - \sfI\sfX,\, \alpha\sfY - \sfI\sfY)
 \end{array} & \begin{array}{c} (v_0, s_1 v_1, s_2 v_2, s_3 v_3 + t_1 v_0, s_4 v_4 + t_2 v_1, s_5 v_5 + t_3 v_2) \end{array}\\ \hline
  \mbox{(ii)} & \begin{array}{c}
 (-i\sfC,\, \sfA + \alpha\,\sfI\sfA,\, \sfB + \alpha\,\sfI\sfB,\, \\
 \sfI\sfC,\, -\alpha\sfB + \sfI\sfB,\, \alpha\sfA - \sfI\sfA)
 \end{array} & \begin{array}{c} (v_0, s_1 (v_1 - i v_2), s_2 (v_1 + i v_2), s_3 v_3 + t_1 v_0, \\ s_4 (v_4 - i v_5) + t_2 (v_1 - i v_2), s_5 (v_4 + iv_5) + t_3 (v_1 + i v_2)) \end{array}\\ \hline
 \end{array}
 \end{align}
 The ansatz for $(T,X_1,...,X_5)$ is chosen by requiring that $T$ act with eigenvalues $0, \pm 1$ (each of multiplicity 2), i.e.\ matching those in Table \ref{F:Fstr} ($\sfD.6_a$).  (Note that $s_i \in \bbC^\times, t_i \in \bbC$.). Imposing brackets, we obtain:
 \begin{align}
 \begin{array}{|c|c|c|}\hline
 \mbox{Case} & \alpha^2 & (s_2,s_3,s_4,s_5,t_1,t_2,t_3) \\ \hline\hline
 \mbox{(i)} & \neq 1 & (\tfrac{-5a}{2s_1(\alpha^2-1)},
 \tfrac{5a\alpha}{2(\alpha^2-1)},
 \tfrac{5as_1\alpha}{3(\alpha^2-1)}, 
 \tfrac{75\alpha}{2s_1(\alpha^2+4)(4\alpha^2+1)},
 -a, -\tfrac{as_1}{3}, \tfrac{-15(\alpha^2-1)}{2s_1(\alpha^2+4)(4\alpha^2+1)}) \\ \hline
 \mbox{(i)} & 1 & (\frac{3i}{2s_1},-\frac{3i\alpha}{2},-i \alpha s_1, \frac{3\alpha}{2s_1},0,0,0) \\ \hline
 \mbox{(ii)} & \neq 1 & (\tfrac{5a}{2s_1(\alpha^2-1)},
 \tfrac{-5ia\alpha}{\alpha^2-1},
 \tfrac{5ias_1\alpha}{3(\alpha^2-1)}, 
 \tfrac{75i\alpha}{2s_1(\alpha^2+4)(4\alpha^2+1)},
 -a, -\tfrac{as_1}{3}, \tfrac{15(\alpha^2-1)}{2s_1(\alpha^2+4)(4\alpha^2+1)}) \\ \hline
 \mbox{(ii)} & 1 & (\frac{3i}{2s_1},3\alpha,-\alpha s_1, \frac{3i\alpha}{2s_1},0,0,0) \\ \hline
 \end{array}
 \end{align}
 along with \framebox{$a^2 = -\frac{9(\alpha^2 - 1)^2}{(\alpha^2+4)(4\alpha^2+1)}$}, which has image in $(-\frac{9}{4},0]$.
 
 We now pin down the anti-involution $\psi$.  From Table \ref{F:D6AI}, the only $\fso(1,3)$ case where the (real) isotropy acts with real eigenvalues is the $\psi_i$ case, so we easily match this with our case (i).  Focus now on (ii).   
 When $\alpha^2 = 1$, we have $a=0$, so Table \ref{F:D6AI} indicates that $\psi = \widetilde\psi_i$ is the only possibility.

 Consider (ii) with $\alpha^2 \neq 1$.  Pick $s_1$ below to align with the real bases in Table \ref{F:D6AI} provided by $\widetilde\psi_i$ or $\widetilde\psi_{-i}$.  Since $a^2 = -\frac{9(\alpha^2 - 1)^2}{(\alpha^2+4)(4\alpha^2+1)} \in (-\frac{9}{4}, 0)$, then  WLOG $a \in i(0,\frac{3}{2})$.  Further, require $s_1 = i s_2$, i.e.\ $(s_1)^2 = \frac{5ai}{2(\alpha^2 - 1)}$.
 \begin{enumerate}
 \item \framebox{$0 < \alpha^2 < 1$}: We have $s_1 \in \bbR$.  Then $X_1+iX_2 = 2s_1 v_1$ and $X_2 + i X_1 = 2s_1 v_2$ are real, so $\psi = \widetilde\psi_i$.
 \item \framebox{$\alpha^2 > 1$}: We have $s_1 \in i\bbR$.  Then $X_1-iX_2 = -2is_1 v_2$ and $X_2 - i X_1 = -2is_1 v_1$ are real, so $\psi = \widetilde\psi_{-i}$.
 \end{enumerate}
 \end{proof}
 
%%%%%%%%%%%%%%%%%%%%%%%%%%%%%%%%%%%%%%%%%%%%%%%%%%%%
 \section{Holonomy and almost-Einstein $(2,3,5)$-distributions}
 \label{S:holAE}
%%%%%%%%%%%%%%%%%%%%%%%%%%%%%%%%%%%%%%%%%%%%%%%%%%%%
 Using \eqref{E:hol}, we easily compute the holonomy algebra $\hol \subset \fg$ for each of the algebraic models that we have found.  Some cases were already known, in particular \cite{Willse2014} ($\sfN.7$), \cite{Willse2018} ($\sfN.6$), and \cite[Example 60]{SW2017a} (two real $\sfD.6_0$ cases), but to our knowledge the general $\sfD.6$ case had not previously been treated.
 
 \begin{theorem} For non-flat, complex, multiply-transitive $(2,3,5)$-distributions, $\hol$ is given by:
 \begin{align}
 \begin{array}{|c|c|c|c|} \hline
  \sfN.7_c & \sfN.6 & \sfD.6_{a \neq 0} & \sfD.6_0\\ \hline\hline
 \langle f_{01}, f_{11}, f_{21}, f_{31}, f_{32} \rangle & \fg & \fg & \begin{array}{c} \langle h_{01}, e_{01}, f_{01},    f_{21} + e_{21},\\
 f_{10} + e_{32}, f_{11} + e_{31},\\
   f_{31} + e_{11}, f_{32} + e_{10} \rangle \end{array} \\ \hline
 \end{array}
 \end{align}
 For $\sfN.7_c$, we have $\hol \cong \mathfrak{heis}_5$ (5-dimensional Heisenberg algebra).  For $\sfD.6_0$, we have $\hol \cong \fsl(3,\bbC)$.
 \end{theorem}
 
 \begin{proof} Use the algebraic models $(\ff;\fg,\fp)$ in Table \ref{F:alg-models}.  The starting point is $\hol^0 = \langle \kappa(x,y) : x,y \in \ff \rangle$.
 \begin{align}
 \begin{array}{|c||c|c|c|c|} \hline
 \mbox{Model} & \sfN.7 & \sfN.6 & \sfD.6_{a \neq 0} & \sfD.6_0\\ \hline
 \hol^0 & f_{01} & \begin{array}{c} e_{10} - 3 f_{01},\\ e_{21} + \frac{8}{3} f_{01},\\ e_{31}, \,\, e_{32} + \frac{2}{3} f_{01} \end{array} & \begin{array}{c}
 e_{01}, h_{01}, f_{01}, e_{21},\\
 e_{10} - \frac{3a}{4} e_{31},\\
 e_{11} - \frac{3a}{4} e_{32}
 \end{array} & e_{01}, f_{01}, h_{01}\\ \hline
 \end{array}
 \end{align}
 Directly applying \eqref{E:hol} yields the claimed results, e.g. for $\sfD.6_0$, we obtain $\hol^0 \subset \hol^1 \subset \hol^2 = \hol$, with
 \begin{align*}
 \hol^1 / \hol^0 \equiv \langle f_{32} + e_{10},\, f_{31} + e_{11}, \, f_{10} + e_{32},\, f_{11} + e_{31} \rangle, \quad \hol^2 / \hol^1 \equiv \langle f_{21} + e_{21} \rangle.
 \end{align*}
 We directly verify that the Killing form on $\hol$ is non-degenerate, so $\hol$ is 8-dimensional complex semisimple.  The only complex simple Lie algebra of smaller dimension is $\fsl(2,\bbC)$, so $\hol$ is simple and $\hol \cong \fsl(3,\bbC)$. 
 \end{proof}
 
 Turn now to real $(2,3,5)$-distributions.  Given an anti-involution $\psi$ of $(\ff;\fg,\fp)$, the real form $\ff^\psi$ has holonomy that is the fixed point set under $\psi$ of the complexified holonomy.  We conclude that {\em real forms of $\sfN.6$ and $\sfD.6_{a \neq 0}$ have full holonomy (the split real form of $\Lie(G_2)$)}.  In particular, using Theorem \ref{T:rolling-classify}:
 
 \begin{cor}
 The real rolling ball distribution $(M,\cD_\rho)$ for $\rho > 1$ has full holonomy.
 \end{cor}
 
 For $\sfN.7_c$ real forms, we have:
 \begin{align}
 \begin{array}{|c|c|c|c|c|c|} \hline
 \mbox{Parameter} & \psi &  \mbox{$\bbR$-basis for } \hol \\ \hline\hline
 c \in [0,\infty) & \psi_1 & \langle f_{01}, f_{11}, f_{21}, f_{31}, f_{32} \rangle \\  \hline
 c \in (0,\infty) & \psi_{-1} & \langle f_{01}, i f_{11}, f_{21}, i f_{31}, i f_{32} \rangle \\ \hline
 c \in i[0,\infty) & \psi_i & \langle f_{01}, (1+i) f_{11}, if_{21}, (1-i) f_{31}, (1-i) f_{32} \rangle \\ \hline
 c \in i(0,\infty) & \psi_{-i} & \langle f_{01}, (1-i) f_{11}, if_{21}, (1+i) f_{31}, (1+i) f_{32} \rangle \\ \hline
 \end{array}
 \end{align}
 In each case, $\hol$ is isomorphic to the (real) 5-dimensional Heisenberg algebra.
 
 The other case of reduced holonomy is the $\sfD.6_0$ case, for which real form data is given in Table \ref{F:hol-R}.     The isomorphism type of $\hol$ is recognized via the signature of the Killing form (computed in {\tt Maple}).
 
 \begin{footnotesize}
 \begin{table}[h]
 \[
 \begin{array}{|c|c|c|c|c|c|} \hline
 \psi & \ff^\psi &  \mbox{$\bbR$-basis for } \hol & \hol \mbox{ isomorphism type} \\ \hline\hline
 \psi_1 & \fsl(2,\bbR) \times \fsl(2,\bbR) & 
 \begin{array}{c} \langle h_{01}, e_{01}, f_{01}, f_{21} + e_{21},\\
 f_{10} + e_{32}, f_{11} + e_{31},\\
   f_{31} + e_{11}, f_{32} + e_{10} \rangle \end{array} & \fsl(3,\bbR)\\ \hline
 \widetilde\psi_1 & \fsl(2,\bbR) \times \fso(3) & \begin{array}{c} \langle ih_{01}, f_{01} + e_{01}, i(f_{01} - e_{01}), i(f_{21} + e_{21}),\\
 f_{10} + e_{32} + f_{11} + e_{31},  i(f_{10} + e_{32} - f_{11} - e_{31}),\\
 f_{31} + e_{11} + f_{32} + e_{10},  i(f_{31} + e_{11} - f_{32} - e_{10}) \rangle \end{array} & \fsu(1,2)\\ \hline
 \psi_i & \fso(1,3) & \begin{array}{c} \langle h_{01}, e_{01}, f_{01}, i(f_{21} + e_{21}),\\
 (1+i)(f_{10} + e_{32}), (1+i)(f_{11} + e_{31}),\\
 (1-i)(f_{31} + e_{11}), (1-i)(f_{32} + e_{10}) \rangle \end{array} & \fsu(1,2)\\ \hline
 \widetilde\psi_{i} & \fso(1,3) & \begin{array}{c} \langle ih_{01}, f_{01} + e_{01}, i(f_{01} - e_{01}), f_{21} + e_{21},\\
 f_{10} + e_{32} + i(f_{11} + e_{31}), f_{11} + e_{31} + i(f_{10} + e_{32}),\\
 f_{31} + e_{11} - i(f_{32} + e_{10}),  f_{32} + e_{10} - i(f_{31} + e_{11}) \rangle \end{array} & \fsl(3,\bbR)\\ \hline
 \end{array}
 \]
 \caption{Infinitesimal holonomy for real forms of $\sfD.6_0$ models}
 \label{F:hol-R}
 \end{table}
 \end{footnotesize}
 
 \begin{theorem} Among $(2,3,5)$-structures of type $\sfD.6$, only the $\sfD.6_0$ structures have reduced holonomy, and real forms are equivalent to one of the following four models:
 \begin{align}
 \begin{array}{|c|c|c|c|cc} \hline
 \ff_\bbR & \ff^0_\bbR & \ff^{-1}_\bbR / \ff^0_\bbR & \hol\\ \hline\hline
 \fsl(2,\bbR) \times \fsl(2,\bbR) & \sfH_1 + \sfH_2 & \sfX_1 + \sfX_2, \,\, \sfY_1 - \sfY_2 & \fsl(3,\bbR)\\ \hline
 \fsl(2,\bbR) \times \fso(3) & \frac{\sfH}{2} + \frac{\sfX - \sfY}{\sqrt{2}} + \sfC & \frac{\sfH}{2} + \frac{\sfX}{\sqrt{2}} + \sfA,\,\, \frac{\sfH}{2} -\frac{\sfY}{\sqrt{2}} + \sfB & \fsu(1,2)\\ \hline
 \fso(1,3) & \sfH & \sfX + \sfI\sfX,\,\, \sfY + \sfI\sfY & \fsu(1,2)\\ \hline
 \fso(1,3) & \sfC & \sfA + \sfI\sfA,\,\, \sfB + \sfI\sfB & \fsl(3,\bbR)\\ \hline
 \end{array}
 \end{align}
 \end{theorem}
 
 \begin{proof} We have $a=0$ for the $\fso(1,3)$ cases from Theorem \ref{T:so13}.  For the first two cases, we match these with the complexified data in Table \ref{F:LieTh}.  It is immediate for the first case, while for the second case we set
 \begin{align}
 \begin{split}
 (\bH, \bX, \bY) &= (2i(\tfrac{\sfH}{2} + \tfrac{\sfX-\sfY}{\sqrt{2}}), \tfrac{\sfH}{2} + \tfrac{\sfX}{\sqrt{2}} + i(\tfrac{\sfH}{2} -\tfrac{\sfY}{\sqrt{2}}), \tfrac{\sfH}{2} + \tfrac{\sfX}{\sqrt{2}} - i(\tfrac{\sfH}{2} -\tfrac{\sfY}{\sqrt{2}}))\\
 (\bH', \bX', \bY') &= (2i\sfC, \sfA + i\sfB, \sfA - i\sfB)
 \end{split}
 \end{align}
 The holonomy is then deduced from Table \ref{F:hol-R}.  By Table \ref{F:D6AI}, the claimed list for $a=0$ is complete.
 \end{proof}
 
 The first two cases with reduced holonomy correspond to special rolling examples (two hyperbolic spaces or hyperbolic space on sphere) identified in \cite[Example 60]{SW2017a} and \cite[Example 17]{SW2017b}.  The reduced holonomy assertion for the $\fso(1,3)$ cases (Theorem \ref{T:so13} when $a=0$) appears to be new.  
 
 Any $(2,3,5)$-distribution $(M,\cD)$ admits a canonical {\sl Nurowski conformal structure} $\sfc_\cD$ \cite{Nur2005} and this is covariant for $\cD$, i.e.\ all symmetries of $\cD$ are conformal symmetries of $\sfc_\cD$.  The problem of determining the existence of an Einstein metric in $\sfc_\cD$ was studied in \cite{SW2017a, SW2017b}.
This is governed by the {\sl almost-Einstein equation}, which is a 2nd order linear PDE on a {\sl scale} $\sigma$.  (Given $g \in \sfc_\cD$, the PDE characterizes if $\sigma^{-2} g$ is Einstein.)  If $\sigma$ is a non-trivial solution (possibly with some vanishing), then it is called an {\sl almost-Einstein scale} and $\cD$ is then called an {\sl almost-Einstein} $(2,3,5)$-distribution.  From \cite[Thm.1 \& 2]{SW2017b}, almost-Einstein scales are in 1-1 correspondence with $\nabla^\cV$-parallel sections of the {\sl standard tractor bundle} $\cV := \cG \times_P \bbV$ (associated to the Cartan bundle $\cG$), where $\bbV$ is the standard 7-dimensional $G_2$-rep and $\nabla^\cV$ is the canonical linear connection induced from the Cartan connection $\omega$.  The curvature of $\nabla^\cV$ is precisely the Cartan curvature $\kappa$ viewed (via the rep $\fg \to \tEnd(\bbV)$) as an element of $\Omega^2(M;\tEnd(\cV)) = \Gamma(\cG \times_P (\bigwedge^2 (\fg/\fp)^* \otimes \tEnd(\bbV)))$.  This must annihilate any parallel section, and hence so too must the holonomy.
  
 Parallel sections are determined by their values in a single fibre, and holonomy cuts down on admissible values in each fibre.  On homogeneous spaces, it suffices by \cite[Prop.1.1]{GZ2021} to determine the $\hol$-annihilated elements $\bbV^{\hol} \subset \bbV$ using \eqref{E:g2rep}, e.g.\ $\bbV^\hol = \bbV$ for the flat model, $\bbV^\hol = 0$ for $\sfN.6$ and $\sfD.6_{a \neq 0}$, while $\bbV^\hol = \langle v_6, v_7 \rangle$ for $\sfN.7_c$, and $\bbV^\hol = \langle v_1 + v_7 \rangle$ for $\sfD.6_0$.  We conclude that:

 \begin{theorem}
 The (real forms of) $\sfN.6$ and $\sfD.6_{a \neq 0}$ models do not admit almost-Einstein scales.  The space of almost-Einstein scales is 7-dimensional for the flat model, 2-dimensional for (real forms of) $\sfN.7_c$ models and 1-dimensional for (real forms of) $\sfD.6_0$ models.
 \end{theorem}
 
 Thus, our algebraic approach allowed us to recover known results \cite{Willse2014, Willse2018, SW2017a} and complete the $\sfD.6$ case.  The almost-Einstein equation belongs to the broad class of BGG equations, and we refer the reader to \cite{GZ2021} for more examples of these on various homogeneous parabolic geometries.

%%%%%%%%%%%%%%%%%%%%%%%%%%%%%%%%%%%%%%%%%%%%%%%%%%%%
 \section*{Acknowledgements}
%%%%%%%%%%%%%%%%%%%%%%%%%%%%%%%%%%%%%%%%%%%%%%%%%%%% 
 We gratefully acknowledge the {\tt DifferentialGeometry} package in {\tt Maple}, and insightful discussions over the years with Robert Bryant, Boris Doubrov, Boris Kruglikov, Pawel Nurowski, Katja Sagerschnig, and Travis Willse. The research leading to these results has received funding from the Norwegian Financial Mechanism 2014-2021 (project registration number 2019/34/H/ST1/00636),  and the Troms\o{} Research Foundation (project ``Pure Mathematics in Norway'').

 \newpage

 \appendix
%%%%%%%%%%%%%%%%%%%%%%%%%%%%%%%%%%%%%%%%%%%%%%%%%%%%
 \section{Cartan connection structure equations}
 \label{S:curvature}
%%%%%%%%%%%%%%%%%%%%%%%%%%%%%%%%%%%%%%%%%%%%%%%%%%%%
 With the curvature module $\bbE \subset \ker(\partial^*)^1$ given explicitly in Table \ref{F:CurvMod}, we can write out the structure equations $d\omega = -\omega \wedge \omega + K$ for the associated regular, normal Cartan geometry $(\cG \to M, \omega)$.  This is not required for our purposes in this article, but we do so to illustrate how representation theory is used to derive, organize, and interpret such equations, and to facilitate comparisons with \cite{Car1910, Nur2005}.  (See an attachment in our {\tt arXiv} submission for {\tt Maple} code supporting this section.)
 
 First, write $\omega$ in the form \eqref{E:g2rep}, but with $z_i,a_{st},b_{st}$ replaced by 1-forms $\zeta_i,\theta_{st}, \pi_{st} \in \Omega^1(\cG)$.  These are then dual to the vector fields $\omega^{-1}(\sfZ_i), \omega^{-1}(f_{st}), \omega^{-1}(e_{st})$.  Next, use \eqref{E:KFconvert} to convert $\kappa$, valued in $\bbE$, from 2-chain form to 2-cochain form, and then identify corresponding contributions to the structure equations, e.g.\ $e_{10} \wedge e_{31} \otimes f_{01}$ becomes a multiple of $f_{10}^* \wedge f_{31}^* \otimes f_{01}$, and so in the $d\theta_{01}$ structure equation this contributes the $\theta_{10} \wedge \theta_{31}$ term (appearing in ${\bf L_{01}}$ below).  In this manner, we obtain the {\sl primary structure equations}:
 \begin{footnotesize}
 \begin{align}
 \begin{split} \label{E:Pstreq}
 d\pi_{32} &= -(3\zeta_1 + 2\zeta_2) \wedge \pi_{32} - \pi_{01} \wedge \pi_{31} - 3 \pi_{11} \wedge \pi_{21} + {\bf K_{32}}\\
 d\pi_{31} &= -(3 \zeta_1 + \zeta_2) \wedge \pi_{31} - \theta_{01} \wedge \pi_{32} + 3\pi_{10} \wedge \pi_{21} + {\bf K_{31}}\\
 d\pi_{21} &= -(2\zeta_1 + \zeta_2) \wedge \pi_{21} + \theta_{10} \wedge \pi_{31} - \theta_{11} \wedge \pi_{32} - 2 \pi_{10} \wedge \pi_{11} + {\bf K_{21}}\\
 d\pi_{11} &= -(\zeta_1 + \zeta_2) \wedge \pi_{11} - 2 \theta_{10} \wedge \pi_{21} + \theta_{21} \wedge \pi_{32} + \pi_{01} \wedge \pi_{10} + {\bf K_{11}}\\
 d\pi_{10} &= -\zeta_1 \wedge \pi_{10} + \theta_{01} \wedge \pi_{11} + 2 \theta_{11} \wedge \pi_{21} - \theta_{21} \wedge \pi_{31} + {\bf K_{10}}\\
 d\pi_{01} &= -\zeta_2 \wedge \pi_{01} - 3 \theta_{10} \wedge \pi_{11} + \theta_{31} \wedge \pi_{32} + {\bf K_{01}}\\
 d\zeta_1 &= -\theta_{01} \wedge \pi_{01} + 2 \theta_{10} \wedge \pi_{10} - \theta_{11} \wedge \pi_{11} + \theta_{21} \wedge \pi_{21} + \theta_{31} \wedge \pi_{31} + {\bf K_1}\\
 d\zeta_2 &= 2\theta_{01} \wedge \pi_{01} - 3 \theta_{10} \wedge \pi_{10} + 3 \theta_{11} \wedge \pi_{11} - \theta_{31} \wedge \pi_{31} + \theta_{32} \wedge \pi_{32} + {\bf K_2}\\
 d\theta_{01} &= \zeta_2 \wedge \theta_{01} - 3 \theta_{11} \wedge \pi_{10} + \theta_{32} \wedge \pi_{31} + {\bf L_{01}}\\
 d\theta_{10} &= \zeta_1 \wedge \theta_{10} + \theta_{11} \wedge \pi_{01} + 2 \theta_{21} \wedge \pi_{11} - \theta_{31} \wedge \pi_{21}\\
 d\theta_{11} &= (\zeta_1 + \zeta_2) \wedge \theta_{11} - \theta_{01} \wedge \theta_{10} - 2\theta_{21} \wedge \pi_{10} + \theta_{32} \wedge \pi_{21}\\
 d\theta_{21} &= (2\zeta_1 + \zeta_2) \wedge \theta_{21} + 2\theta_{10} \wedge \theta_{11} + \theta_{31} \wedge \pi_{10} - \theta_{32} \wedge \pi_{11}\\
 d\theta_{31} &= (3\zeta_1 + \zeta_2) \wedge \theta_{31} - 3\theta_{10} \wedge \theta_{21} - \theta_{32} \wedge \pi_{01}\\
 d\theta_{32} &= (3\zeta_1 + 2\zeta_2) \wedge \theta_{32} + \theta_{01} \wedge \theta_{31} + 3 \theta_{11} \wedge \theta_{21}
 \end{split}
 \end{align}
 \end{footnotesize}
 where the terms that comprise the curvature $K$ are
 \begin{footnotesize}
 \begin{align}
 \begin{split}
 {\bf K_{32}} &= \tfrac{1}{3} D_2 (\theta_{11} \wedge \theta_{31} - \theta_{10} \wedge \theta_{32}) - \tfrac{2}{3}  D_1 \theta_{10} \wedge \theta_{31} + E \theta_{21} \wedge \theta_{31} + \CD_2 \theta_{10} \wedge \theta_{31} + \CD_3 (\theta_{10} \wedge \theta_{32} - \theta_{11} \wedge \theta_{31}) \\
 &\qquad\qquad - \CD_4 \theta_{11} \wedge \theta_{32} - 2 \CE_2 \theta_{21} \wedge \theta_{31} - 2 \CE_3 \theta_{21} \wedge \theta_{32} + \CF_2 \theta_{31} \wedge \theta_{32}\\
 {\bf K_{31}} &= \tfrac{1}{3} D_1 (\theta_{10} \wedge \theta_{32} - \theta_{11} \wedge \theta_{31}) - \tfrac{2}{3} D_2 \theta_{11} \wedge \theta_{32} - E \theta_{21} \wedge \theta_{32} + \CD_1 \theta_{10} \wedge \theta_{31} + \CD_2 (\theta_{10} \wedge \theta_{32} - \theta_{11} \wedge \theta_{31}) \\
 &\qquad\qquad - \CD_3 \theta_{11} \wedge \theta_{32} - 2 \CE_1 \theta_{21} \wedge \theta_{31} - 2 \CE_2 \theta_{21} \wedge \theta_{32} + \CF_1 \theta_{31} \wedge \theta_{32}\\
 {\bf K_{21}} &= -C_1 \theta_{10} \wedge \theta_{31} - C_2 (\theta_{10} \wedge \theta_{32} - \theta_{11} \wedge \theta_{31}) + C_3 \theta_{11} \wedge \theta_{32} + 2 D_1 \theta_{21} \wedge \theta_{31} + 2 D_2 \theta_{21} \wedge \theta_{32} - E \theta_{31} \wedge \theta_{32}\\
 {\bf K_{11}} &= -B_2 \theta_{10} \wedge \theta_{31} + B_3 (\theta_{11} \wedge \theta_{31} - \theta_{10} \wedge \theta_{32}) + B_4 \theta_{11} \wedge \theta_{32} + 2 C_2 \theta_{21} \wedge \theta_{31} + 2 C_3 \theta_{21} \wedge \theta_{32} - D_2 \theta_{31} \wedge \theta_{32}\\
 {\bf K_{10}} &= B_1 \theta_{10} \wedge \theta_{31} + B_2 (\theta_{10} \wedge \theta_{32} - \theta_{11} \wedge \theta_{31}) - B_3 \theta_{11} \wedge \theta_{32} - 2 C_1 \theta_{21} \wedge \theta_{31} - 2 C_2 \theta_{21} \wedge \theta_{32} + D_1 \theta_{31} \wedge \theta_{32}\\
 {\bf K_{01}} &= -A_3\theta_{10} \wedge \theta_{31} + A_4(\theta_{11} \wedge \theta_{31} - \theta_{10} \wedge \theta_{32}) + A_5 \theta_{11} \wedge \theta_{32} + 2 B_3 \theta_{21} \wedge \theta_{31} + 2 B_4 \theta_{21} \wedge \theta_{32} - C_3 \theta_{31} \wedge \theta_{32}\\
 {\bf K_1} &= -A_2\theta_{10} \wedge \theta_{31} - A_3(\theta_{10} \wedge \theta_{32} - \theta_{11} \wedge \theta_{31}) + A_4 \theta_{11} \wedge \theta_{32} + 2 B_2 \theta_{21} \wedge \theta_{31} + 2 B_3 \theta_{21} \wedge \theta_{32} + C_2 \theta_{32} \wedge \theta_{31}\\
 {\bf K_2} &= 2A_2\theta_{10} \wedge \theta_{31} + 2A_3(\theta_{10} \wedge \theta_{32} - \theta_{11} \wedge \theta_{31}) - 2 A_4 \theta_{11} \wedge \theta_{32} - 4 B_2 \theta_{21} \wedge \theta_{31} - 4 B_3 \theta_{21} \wedge \theta_{32} - 2 C_2 \theta_{32} \wedge \theta_{31}\\
 {\bf L_{01}} &= A_1 \theta_{10} \wedge \theta_{31} + A_2 (\theta_{10} \wedge \theta_{32} - \theta_{11} \wedge \theta_{31}) - A_3 \theta_{11} \wedge \theta_{32} - 2 B_1 \theta_{21} \wedge \theta_{31} - 2 B_2 \theta_{21} \wedge \theta_{32} + C_1 \theta_{31} \wedge \theta_{32}
 \end{split}
 \end{align}
 \end{footnotesize}
 
The 24 coefficients $A_1,...,\CF_2$ are functions on $\cG$.  Their exterior derivatives are the {\sl secondary structure equations}, obtained by imposing integrability conditions $d^2=0$ on \eqref{E:Pstreq}, and wedging with {\sl semibasic} forms $\langle \theta_{10}, \theta_{11}, \theta_{21}, \theta_{31}, \theta_{32} \rangle$ for $\cG \to M$, e.g.\ the $d(C_1)$ equation (modulo a semibasic 1-form $\gamma_1$) is calculated from:
 \begin{align} \label{E:C1V}
 \begin{split}
 0 &= d(d\theta_{01}) \wedge \theta_{10} \wedge \theta_{11} \wedge \theta_{21} \\
 &= (d(C_1) + (6 \zeta_1 + 2 \zeta_2) C_1 + 2 \theta_{01} C_2 + 2 \pi_{11} B_1 + 2 \pi_{10} B_2 ) \wedge \theta_{10} \wedge \theta_{11} \wedge \theta_{21} \wedge \theta_{31} \wedge \theta_{32}.
 \end{split}
 \end{align}
 Such equations encode the vertical variation of curvature coefficients.
Since $\omega$ is a Cartan connection, then these can be alternatively obtained from $\fp$-equivariancy of $\kappa$, i.e.\ $\cL_{X^\dagger} \kappa = -\rho(X) \circ \kappa, \,\,\forall X \in \fp$, where $X^\dagger \in \fX(\cG)$ is the associated fundamental vertical vector field, and $\rho : \fp \to \fgl(\bbE)$ is the $\fp$-rep on $\bbE$.  Identifying $\{ A_1,..., \widetilde{F}_2 \}$ with linear coordinates on $\bbE$, we obtain the secondary structure equations below from the induced $\fp$-action on $\bbE^*$.   The action of $X^\dagger$ can be immediately read off, e.g.\ from $\pi_{01}$ coefficients,
 \begin{align} \label{E:vertvar}
 \cL_{e_{01}^\dagger} (A_1,A_2,A_3,A_4,A_5, B_1, B_2, \ldots) = (0,-A_1,-2A_2,-3A_3,-4A_4,0,-B_1, \ldots)
 \end{align}
  
 Our secondary structure equations are:
 \begin{footnotesize}
 \begin{align}
 \begin{split}
d(A_1) &= -( \hspace{0.6in} 4\zeta_1 A_1 \hspace{0.39in} + 4 \theta_{01} A_2 ) + \alpha_1\\
d(A_2) &= -( \,\,\, \pi_{01} A_1 + (4 \zeta_1 +\,\,\, \zeta_2) A_2 + 3 \theta_{01} A_3 ) + \alpha_2\\
d(A_3) &= -( 2 \pi_{01} A_2 + (4 \zeta_1 + 2 \zeta_2) A_3 + 2 \theta_{01} A_4 ) + \alpha_3\\
d(A_4) &= -( 3 \pi_{01} A_3 + (4 \zeta_1 + 3 \zeta_2) A_4 + \,\,\,\theta_{01} A_5 ) + \alpha_4\\
d(A_5) &= -(  4 \pi_{01} A_4 + (4 \zeta_1 + 4 \zeta_2) A_5 \hspace{0.545in} ) + \alpha_5\\
d(B_1) &= -( \hspace{0.56in} (5 \zeta_1 +\,\,\, \zeta_2) B_1 + 3\theta_{01} B_2 +\,\, \pi_{11} A_1 +\,\,\, \pi_{10} A_2 ) + \beta_1\\
d(B_2) &= -( \,\,\,\pi_{01} B_1 + (5 \zeta_1 + 2 \zeta_2) B_2 + 2 \theta_{01} B_3 +\,\, \pi_{11} A_2 +\,\,\, \pi_{10} A_3 ) + \beta_2\\
d(B_3) &= -( 2 \pi_{01} B_2 + (5 \zeta_1 + 3 \zeta_2) B_3 +\,\,\, \theta_{01} B_4 +\,\, \pi_{11} A_3 +\,\,\, \pi_{10} A_4 ) + \beta_3\\
d(B_4) &= -( 3 \pi_{01} B_3 + (5 \zeta_1 + 4 \zeta_2) B_4 \hspace{0.545in} +\,\, \pi_{11} A_4 +\,\,\, \pi_{10} A_5 ) + \beta_4\\
d(C_1) &= -( \hspace{0.555in} (6 \zeta_1 + 2 \zeta_2) C_1 + 2 \theta_{01} C_2 + 2 \pi_{11} B_1 + 2 \pi_{10} B_2 ) + \gamma_1\\
d(C_2) &= -( \,\,\,\pi_{01} C_1 + (6 \zeta_1 + 3 \zeta_2) C_2 +\,\,\, \theta_{01} C_3 + 2 \pi_{11} B_2 + 2 \pi_{10} B_3 ) + \gamma_2\\
d(C_3) &= -( 2 \pi_{01} C_2 + (6 \zeta_1 + 4 \zeta_2) C_3 \hspace{0.54in} + 2 \pi_{11} B_3 + 2 \pi_{10} B_4 ) + \gamma_3\\
d(D_1) &= -( \hspace{0.555in} (7 \zeta_1 + 3 \zeta_2) D_1 + \,\,\theta_{01} D_2 + 3 \pi_{11} C_1 + 3 \pi_{10} C_2 ) + \delta_1\\
d(D_2) &= -( \,\,\pi_{01} D_1 + (7 \zeta_1 + 4 \zeta_2) D_2 \hspace{0.54in} + 3 \pi_{11} C_2 + 3 \pi_{10} C_3 ) + \delta_2\\
d(E) &= -(  \hspace{0.55in} (8 \zeta_1 + 4 \zeta_2) E \hspace{0.59in} + 4 \pi_{11} D_1 + 4 \pi_{10} D_2 ) + \epsilon\\ \\ \hline\\
 d(\CD_1) &= -( \hspace{0.56in} (7 \zeta_1 + 2 \zeta_2) \CD_1 + 3 \theta_{01} \CD_2 \hspace{1.76in} +\,\,\, 3 \pi_{10} C_1 + 3 \pi_{21} B_1 -\,\,\, \pi_{32} A_1 +\,\,\, \pi_{31}A_2 ) + \widetilde\delta_1\\
 d(\CD_2) &= -( \,\,\,\pi_{01} \CD_1 + (7 \zeta_1 + 3 \zeta_2) \CD_2 + 2 \theta_{01} \CD_3 \hspace{1.14in} -\,\,\,\,\,\, \pi_{11} C_1 +\,\,\, 2 \pi_{10} C_2 + 3 \pi_{21} B_2 -\,\,\, \pi_{32} A_2 +\,\,\, \pi_{31} A_3) + \widetilde\delta_2\\
 d(\CD_3) &= -( 2 \pi_{01} \CD_2 + (7 \zeta_1 + 4 \zeta_2) \CD_3 + \,\,\,\theta_{01} \CD_4 \hspace{1.14in} -\,\,\, 2 \pi_{11} C_2 + \,\,\,\,\,\, \pi_{10} C_3 + 3 \pi_{21} B_3 -\,\,\, \pi_{32} A_3 +\,\,\, \pi_{31} A_4 ) + \widetilde\delta_3\\
 d(\CD_4) &= -(  3 \pi_{01} \CD_3 + (7 \zeta_1 + 5 \zeta_2) \CD_4 \hspace{1.69in} -\,\,\, 3 \pi_{11} C_3 \hspace{0.63in} + 3 \pi_{21} B_4 -\,\,\, \pi_{32} A_4 +\,\,\, \pi_{31} A_5 ) + \widetilde\delta_4\\
 d(\CE_1) &= -( \hspace{0.56in} (8 \zeta_1 + 3 \zeta_2) \CE_1 +\, 2 \theta_{01} \CE_2 +\,\,\, \pi_{11} \CD_1 +\,\,\, \pi_{10} \CD_2 \hspace{0.63in} + \tfrac{10}{3} \pi_{10} D_1 + 3 \pi_{21} C_1 -\,\,\, \pi_{32} B_1 +\,\,\, \pi_{31} B_2 ) + \widetilde\epsilon_1\\
d(\CE_2) &= -( \,\,\,\pi_{01} \CE_1 + (8 \zeta_1 + 4 \zeta_2) \CE_2 + \,\,\,\, \theta_{01} \CE_3 +\,\,\, \pi_{11} \CD_2 +\,\,\, \pi_{10} \CD_3 - \,\, \tfrac{5}{3} \pi_{11} D_1 +\,\,\, \tfrac{5}{3} \pi_{10} D_2 + 3 \pi_{21} C_2 -\,\,\, \pi_{32} B_2 +\,\,\, \pi_{31} B_3 ) + \widetilde\epsilon_2\\
d(\CE_3) &= -( 2 \pi_{01} \CE_2 + (8 \zeta_1 + 5 \zeta_2) \CE_3 \hspace{0.56in} +\,\,\, \pi_{11} \CD_3 +\,\,\, \pi_{10} \CD_4 - \tfrac{10}{3}  \pi_{11} D_2 \hspace{0.64in} + 3 \pi_{21} C_3 -\,\,\, \pi_{32} B_3 +\,\,\, \pi_{31} B_4 ) + \widetilde\epsilon_3\\
 d(\CF_1) &= -( \hspace{0.56in} (9 \zeta_1 + 4 \zeta_2) \CF_1 \,+\,\,\, \theta_{01} \CF_2 \,+ 2 \pi_{11} \CE_1 + 2 \pi_{10} \CE_2 \,\hspace{0.64in} + \,\,\,4 \pi_{10} E \,\,\,+ 3 \pi_{21} D_1 -\,\, \pi_{32} C_1 \,+\,\,\, \pi_{31}C_2 ) + \widetilde\phi_1\\
 d(\CF_2) &= -( \,\,\,\,\pi_{01} \CF_1 + (9 \zeta_1 + 5 \zeta_2) \CF_2 \hspace{0.57in} + 2 \pi_{11} \CE_2 + 2 \pi_{10} \CE_3 \,-\,\,\, 4 \pi_{11} E \hspace{0.7in} + 3 \pi_{21} D_2 -\,\, \pi_{32} C_2 \,+\,\,\, \pi_{31} C_3 ) + \widetilde\phi_2
 \end{split}
 \end{align}
 \end{footnotesize}
 Here, $\alpha_1, ..., \widetilde\phi_2$ are 24 semibasic 1-forms, i.e.\ each is a linear combination of $\theta_{10}, \theta_{11}, \theta_{21}, \theta_{31}, \theta_{32}$.

Let us make the $\fp$-action on $\bbE / \bbE^5$ and $\bbE / \widetilde\bbE$ more concrete.  First, take the basis $(\sfy,\sfx,\sfz,\sfw,\sfv) = (e_{10}, -e_{11}, -2e_{21}, 3e_{31}, 3e_{32})$ on $\fp_+$.  The $\fp$-action on $\fp_+$ is via:
 \begin{align} \label{E:pp+}
 \begin{array}{l@{\,}l} 
 e_{01} &\,\,\leftrightarrow\,\, \sfx\partial_\sfy + \sfv\partial_\sfw,\\
 f_{01} &\,\,\leftrightarrow\,\, \sfy\partial_\sfx + \sfw\partial_\sfv,\\
 h_{01} &\,\,\leftrightarrow\,\, \sfx\partial_\sfx - \sfy\partial_\sfy + \sfv\partial_\sfv - \sfw\partial_\sfw,\\
 \sfZ_1 &\,\,\leftrightarrow\,\, \sfx\partial_{\sfx} + \sfy\partial_{\sfy} + 2\sfz\partial_{\sfz} + 3\sfv\partial_\sfv + 3\sfw\partial_\sfw,\\
 \end{array} \qquad
 \begin{array}{l@{\,}l}
 e_{10} &\,\,\leftrightarrow\,\, \sfz \partial_\sfx + 2\sfw \partial_\sfz,\\
 e_{11} &\,\,\leftrightarrow\,\, \sfz\partial_\sfy - 2\sfv\partial_\sfz,\\
 e_{21} &\,\,\leftrightarrow\,\, \sfv \partial_\sfx + \sfw \partial_\sfy,\\
 e_{31} &\,\,\leftrightarrow\,\, 0, \quad
 e_{32} \,\,\leftrightarrow\,\, 0.
 \end{array}
 \end{align}
 Then as $\fp$-modules, we have:
  \begin{enumerate}
 \item $\bbE / \bbE^5 \cong H_2(\fp_+,\fg)$ is identified with binary quartics
 \begin{align} \label{E:F}
 \sfF(\sfx,\sfy) &= A_1 \sfy^4 + 4 A_2 \sfx\sfy^3 + 6 A_3 \sfx^2 \sfy^2 + 4 A_4 \sfx^3\sfy + A_5 \sfx^4 \qquad \mod \sfz, \sfv, \sfw.
 \end{align}
 \item $\bbE / \widetilde\bbE$ is identified with {\em ternary} quartics
 \begin{align} \label{E:G}
 \begin{split}
 \sfG(\sfx,\sfy,\sfz) &= A_1 \sfy^4 + 4 A_2 \sfx\sfy^3 + 6 A_3 \sfx^2 \sfy^2 + 4 A_4 \sfx^3\sfy + A_5 \sfx^4\\
 & \qquad + 4(B_1 \sfy^3 + 3 B_2 \sfx\sfy^2 + 3 B_3 \sfx^2\sfy + B_4 \sfx^3) \sfz\\
 & \qquad + 6(C_1 \sfy^2 + 2 C_2 \sfx\sfy + C_3 \sfx^2) \sfz^2\\
 & \qquad + 4(D_1 \sfy + D_2 \sfx) \sfz^3\\
 & \qquad + E \sfz^4 \hspace{2.3in} \mod \sfv, \sfw
 \end{split}
 \end{align}
 \end{enumerate}
 
 The numerical coefficients have been carefully introduced so that $A_1,A_2,...$ transform (under the action induced from \eqref{E:pp+}) in the same way as the corresponding functions on $\cG$, c.f. \eqref{E:vertvar}.  Interpreted with the latter coefficient functions, \eqref{E:F} and \eqref{E:G} are the covariant tensors discovered by Cartan in \cite[para.33]{Car1910}.
 
%%%%%%%%%%%%%%%%%%%%%%%%%%%%%%%%%%%%%%%%%%%%%%%%%%%%
 \section{From Lie-theoretic to Cartan-theoretic descriptions}
 \label{S:CLieTh}
%%%%%%%%%%%%%%%%%%%%%%%%%%%%%%%%%%%%%%%%%%%%%%%%%%%%
 We summarize conversions from Lie-theoretic to Cartan-theoretic descriptions, c.f. Tables \ref{F:alg-models} and \ref{F:Fstr}, for {\em non-flat} (complex) multiply-transitive $(2,3,5)$-geometries.  Recall that the definition of the former (around \eqref{E:LieTh}) and non-flatness implies $\ff^1 = \ff^2 = \ff^3 = 0$ by Corollary \ref{C:PR}.  We will specify an {\sl adapted filtered basis}, i.e.\ a basis of $\ff^0$ (using $v_0$ and/or $w_0$) and 
  $\ff^{-1} / \ff^0 = \langle v_1, v_2 \rangle$, $\ff^{-2} / \ff^{-1} = \langle v_3 \rangle$, $\ff^{-3} / \ff^{-2} = \langle v_4, v_5 \rangle$.

Coordinate models (Table \ref{F:MongeModels}) can be matched with Lie-theoretic descriptions in a similar manner as in Example \ref{X:N6}.  See also the symmetries provided in \cite[Tables 3 \& 4]{Willse2019}.

 \begin{tiny}
 \begin{table}[h]
 \[
 \begin{array}{|l|l|c|l|} \hline
 \multicolumn{2}{|c}{\mbox{Lie-theoretic description}} &  \multicolumn{2}{|c|}{\mbox{Conversion to Cartan-theoretic description}}\\ \hline
 \mbox{Lie algebra $\ff$} & \mbox{Adapted filtered basis} & \mbox{Classification} & \mbox{Dictionary to our basis in Tables \ref{F:alg-models} and \ref{F:Fstr}} \\ \hline\hline
 %%%%%%%%
 \begin{array}{@{}l@{\,}l@{}}
 \multicolumn{2}{@{}c@{}}{\ff = \bbC^2 \ltimes \mathfrak{heis}_5}\\
 \begin{array}{@{}l@{}}
 [\bN_1,\bN_3] = \bN_5, \\ \\{}
 [\bT_1,\bN_1] = -\bN_1, \\ {}
 [\bT_1,\bN_2] = -\bN_2, \\ {}
 [\bT_1,\bN_3] = -\bN_3, \\ {}
 [\bT_1,\bN_4] = -\bN_4, \\ {}
 [\bT_1,\bN_5] = -2\bN_5, \\ {}
 \end{array} & 
 \begin{array}{@{}l@{}}
 [\bN_2,\bN_4] = \bN_5, \\ \\{}
 [\bT_2,\bN_2] = -m\bN_2, \\ {}
 [\bT_2,\bN_3] = -\bN_3, \\ {}
 [\bT_2,\bN_4] = (m-1)\bN_4, \\ {}
 [\bT_2,\bN_5] = -\bN_5,\\ \\\\
 \end{array}
 \end{array} & 
 \begin{array}{@{}l@{}}
 v_0 = \bT_1,\\
 \!w_0 = \frac{1}{m-1} \bN_1 + \frac{1}{2 m-1} \bN_2 \\
 \hspace{0.3in} + \frac{1}{m} \bN_3 + \frac{1}{m(m-1)} \bN_4, \\ 
 v_1 = \bT_2 - \frac{1}{2} \bT_1,\\
 v_2 = \bN_1 + \bN_2 + \bN_3,\\
 v_3 = - m \bN_1 + (1-m) \bN_3,\\
 v_4 = m^2 \bN_2 + \bN_3,\\
 v_5 = \bN_5
 \end{array} & 
 \begin{array}{@{}c@{}}
 \sfN.7_c\\
 c^2 = -\frac{2(k^2+1)^2}{3(k^2-9)(k^2-\frac{1}{9})}\\ \\
 (k \in \bbC \backslash \{ \pm \tfrac{1}{3}, \pm 1,\pm 3 \},\\
 m := \tfrac{k+1}{2})\\
  % m \in \bbC \backslash \{ -1,\tfrac{1}{3},\tfrac{2}{3},2, 0, 1\}
 \end{array}
 &
 \begin{array}{@{}l@{}}
 \,\,\,T = v_0, \quad 
 N = w_0, \quad
 X_1 = \lambda v_1\\
 X_2 = \lambda \left(v_2 - (m-\frac{1}{2}) w_0\right),\\
 X_3 = \frac{1}{2}\lambda^2\left(v_3+ \frac{1}{10}(7m^2-7m+1) w_0\right),
\\
 X_4 = \frac{1}{6} \lambda^3 \left\{ -v_4 + (m+\frac{1}{2}) v_3 \right.\\
 \hspace{0.5in} \left. + \frac{1}{10}(7m^2 + 3m + 1) v_2 \right.\\
 \hspace{0.5in} + \left. \frac{1}{10}(m-\frac{1}{2})(3m^2 - 3m - 1)w_0\right\},
\\
 X_5 = -\frac{1}{6}\lambda^3 v_5 \\
 \quad \lambda^4 = -\frac{600}{(m+1)(m-2)(3m-1)(3m-2)}, \\
 \quad c = \frac{\lambda^2}{20} (2 m^2 - 2 m+1)
 \end{array} \\ \hline 
 %%%%%%%%
  \begin{array}{@{}ll@{}}
 \multicolumn{2}{@{}c@{}}{\ff = \bbC^2 \ltimes \mathfrak{heis}_5}\\
 \begin{array}{@{}l@{}}
 [\bN_1,\bN_3] = \bN_5, \\ \\{}
 [\bS,\bN_1] = -\bN_1, \\ {}
 [\bS,\bN_2] = -\bN_2, \\ {}
 [\bS,\bN_3] = -\bN_3, \\ {}
 [\bS,\bN_4] = -\bN_4, \\ {}
 [\bS,\bN_5] = -2\bN_5, \\ {}
 \end{array} & 
 \begin{array}{@{}l@{}}
 [\bN_2,\bN_4] = \bN_5, \\ \\{}
 [\bT,\bN_1] = -\bN_2, \\ {}
 [\bT,\bN_3] = -\bN_3, \\ {}
 [\bT,\bN_4] = \bN_3 - \bN_4, \\ {}
 [\bT,\bN_5] = -\bN_5\\ \\\\
 \end{array}
 \end{array} & 
 \begin{array}{@{}l@{}}
 v_0 = \bS,\\
 \!w_0 = \bN_1 + \bN_2 + \bN_3 + \bN_4,\\
 v_1 = \bT - \frac{1}{2} \bS,\\
 v_2 = \bN_1 + \bN_3,\\
 v_3 = \bN_2 + \bN_3,\\
 v_4 = \bN_5,\\
 v_5 = \bN_2
 \end{array} & 
 \begin{array}{c} 
 \sfN.7_c\\
 c^2 = \frac{3}{8}
 \end{array} & 
 \begin{array}{@{}l@{}}
 \,\,\,T =  v_0, \quad N = w_0, \quad
 X_1 = \lambda v_1, \\
 X_2 = -\lambda v_2 + \frac{\lambda}{2}  w_0, \\
 X_3 = -\frac{\lambda^2}{2} v_3 + \frac{\lambda^2}{20} w_0,\\
 X_4 = -\frac{\lambda^3}{60} v_2+\frac{\lambda^3}{12} v_3 - \frac{\lambda^3}{6}  v_5 + \frac{\lambda^3}{120} w_0,\\
 X_5 = -\frac{\lambda^3}{6} v_4, \\
 \quad \lambda^4 = 150, \,\, c = \frac{\lambda^2}{20}
 \end{array} \\ \hline
 %%%%%%%%
 \begin{array}{@{}ll@{}}
 \multicolumn{2}{@{}c@{}}{\ff = \fsl(2,\bbC) \ltimes \mathfrak{heis}_3}\\
 \begin{array}{@{}l@{}}
 [\bH,\bX] = 2\bX, \\ {}
 [\bH,\bY] = -2\bY, \\ {}
 [\bX,\bY] = \bH,
 \end{array} & 
 \begin{array}{@{}l@{}}
 [\bH,\bS] = \bS, \\ {}
 [\bH,\bT] = -\bT, \\ {}
 [\bX,\bT] = \bS, \\ {}
 [\bY,\bS] = \bT, \\ {}
 [\bS,\bT] = \bU
 \end{array}
 \end{array} & 
 \begin{array}{@{}l@{}}
 v_0 = \bX,\\
 v_1 = \bY - \bS - \bU,\\
 v_2 = \bH,\\
 v_3 = 3\bS + 2\bU,\\
 v_4 = \bT,\\
 v_5 = \bU
 \end{array} & \sfN.6 & 
 \begin{array}{@{}l@{}}
 \,\,N = \frac{i \sqrt{2}}{4} v_0\\
 X_1 = -2i\sqrt{2} v_1 + 3i\sqrt{2}
 v_0\\
 X_2 = v_2\\
 X_3 = i \sqrt{2} v_3 + \frac{15 i \sqrt{2}}{4} 
 v_0\\
 X_4 = 4 v_4 - v_2\\
 X_5 = \frac{2i \sqrt{2}}{3} v_5 - \frac{i\sqrt{2}}{3} v_3 - \frac{i\sqrt{2}}{4} v_0
 \end{array}\\ \hline
 %%%%%%%%
 \begin{array}{@{}ll@{}}
 \multicolumn{2}{@{}c@{}}{\ff = \fsl(2,\bbC) \times \fsl(2,\bbC)}\\
 \begin{array}{@{}l@{}}
 [\bH,\bX] = 2\bX, \\ {}
 [\bH,\bY] = -2\bY, \\ {}
 [\bX,\bY] = \bH,
 \end{array} & 
 \begin{array}{@{}l@{}}
 [\bH',\bX'] = 2\bX', \\ {}
 [\bH',\bY'] = -2\bY', \\ {}
 [\bX',\bY'] = \bH'
 \end{array}
 \end{array} & 
 \begin{array}{@{}l@{}}
 v_0 = \bH + \bH',\\
 v_1 = \bX - \lambda \bX',\\
 v_2 = \bY - \bY',\\
 v_3 = \bH + \lambda \bH',\\
 v_4 = \bX - \lambda^2 \bX',\\
 v_5 = \bY - \lambda \bY'
 \end{array} & 
 \begin{array}{@{}c@{}}
 \sfD.6_a\\
 a^2 = \tfrac{4(\lambda+1)^2}{(\lambda-9)\left(\lambda-\tfrac{1}{9}\right)}\\ \\
  (\lambda \in \bbC \backslash \{ 0,\frac{1}{9}, 1, 9 \})
 \end{array} & 
 \begin{array}{@{}l|l}
 \begin{array}{@{}l@{}}
 \underline{\lambda \neq -1, 0, \frac{1}{9}, 1, 9}:\\
 \,\,\,T = \frac{v_0}{2}\\
 X_1 = v_1\\
 X_2 = \tfrac{5a}{(\lambda+1)} v_2\\
 X_3 = -\tfrac{5a}{2(\lambda+1)} v_3 + \frac{3a}{2} v_0\\
 X_4 = \tfrac{5 a}{3(\lambda+1)} v_4 -\frac{7a}{6} v_1\\
 X_5 = \tfrac{25a^2}{3 (\lambda+1)^2} v_5 - \tfrac{35a^2}{6 (\lambda+1)} v_2
 \end{array} &
 \begin{array}{@{}l@{}}
 \underline{\lambda = -1}:\\
 \,\,\,T = \frac{v_0}{2}\\
 X_1 = v_1\\
 X_2 = 3 v_2\\
 X_3 = -\frac{3}{2} v_3\\
 X_4 = v_4\\
 X_5 = 3 v_5
 \end{array}
 \end{array}
 \\ \hline
 \begin{array}{@{}ll@{}}
 \multicolumn{2}{@{}c@{}}{\ff = \fsl(2,\bbC) \times \fe(2)}\\
 \begin{array}{@{}l@{}}
 [\bH,\bX] = 2\bX, \\ {}
 [\bH,\bY] = -2\bY, \\ {}
 [\bX,\bY] = \bH 
 \end{array}
 &
 \begin{array}{@{}l@{}}
 [\bZ,\bV_1] = \bV_1, \\ {}
 [\bZ,\bV_2] = -\bV_2, \\ {}
 \end{array}
 \end{array} & 
 \begin{array}{@{}l@{}}
 v_0 = \frac{1}{2}\bH + \bZ,\\
 v_1 = \bX + \bV_1,\\
 v_2 = \bY + \bV_2,\\
 v_3 = \bH,\\
 v_4 = \bX,\\
 v_5 = \bY
 \end{array} & 
 \begin{array}{c}
 \sfD.6_a\\ a^2=4
 \end{array} & 
 \begin{array}{@{}l@{}}
 \,\,\,T = v_0\\
 X_1 = v_1\\
 X_2 = 5 a v_2\\
 X_3 = -\frac{5}{2} a v_3 + \frac{3}{2} a v_0\\
 X_4 = \frac{5}{3} a v_4 - \frac{7}{6} a v_1\\
 X_5 = \frac{100}{3} v_5 - \frac{70}{3} v_2
 \end{array}\\ \hline
%%%%%%%
 \begin{array}{@{}l@{}l}
 \multicolumn{2}{@{}c@{}}{\ff = \fe(3)}\\
 \begin{array}{@{}ll@{}}
 [\bR_1,\bR_2] = \bR_3,\\{}
 [\bR_2,\bR_3] = \bR_1,\\{}
 [\bR_3,\bR_1] = \bR_2,
 \end{array} & 
 \begin{array}{l@{}}
 [\bR_1,\bV_2] = \bV_3\\ {}
 [\bR_1,\bV_3] = -\bV_3, \\ {}
 [\bR_2,\bV_3] = \bV_1\\ {}
 [\bR_2,\bV_1] = -\bV_3, \\ {}
 [\bR_3,\bV_1] = \bV_2\\ {}
 [\bR_3,\bV_2] = -\bV_1
 \end{array}
 \end{array}
 & 
 \begin{array}{@{}l@{}}
 v_0 = \bR_1,\\
 v_1 = \bR_2 + \bV_2,\\
 v_2 = \bR_3 + \bV_3,\\
 v_3 = \bV_1,\\
 v_4 = \bV_2,\\
 v_5 = \bV_3
 \end{array} & 
 \begin{array}{c}
 \sfD.6_a\\ a^2=-\frac{9}{4}
 \end{array} & 
 \begin{array}{@{}l@{}}
 \,\,\,T = i v_0\\
 X_1 = \frac{3}{2} (v_1 + i v_2)\\
 X_2 = \frac{5i}{2} (v_1 - i v_2)\\
 X_3 = -\frac{15}{2} v_3 + \frac{3i}{2} i v_0\\
 X_4 = -\frac{15 i}{4} (v_4 + i v_5) + \frac{3i}{4} (v_1 + i v_2)\\
 X_5 = \frac{25}{4} (v_4 - i v_5) - \frac{5}{4} (v_1 - i v_2)\\
 \end{array} \\ \hline
 \end{array}
 \]
 \caption{From Lie-theoretic to Cartan-theoretic descriptions}
 \label{F:LieTh}
 \end{table}
 \end{tiny}

%%%%%%%%%%%%%%%%%%%%%%%%%%%%%%%%%%%%%%%%%%%%%%%%%%%%


\begin{thebibliography}{99}

 \bibitem{Agr2007} A. Agrachev, {\it Rolling balls and octonions}, Proc. Steklov Inst. Math. {\bf 258} (2007), 13--22.
 
 \bibitem{BH2014} J.C. Baez, J. Huerta, {\it $G_2$ and the rolling ball}, Trans. Amer. Math. Soc. {\bf 366} (2014), 5257--5293.

 \bibitem{BM2009} G. Bor, R. Montgomery, {\it $G_2$ and the ``rolling distribution"}, L'Enseignement Mathematique {\bf 55} (2009), 157--196.

 \bibitem{Bry2013} Private email communication with Robert Bryant, 23 October 2013.
 
 \bibitem{Cap2006} A. \v{C}ap, {\it Correspondence spaces and twistor spaces for parabolic geometries}, J. Reine Angew. Math. {\bf 582}, 143--172 (2005).
 
 \bibitem{Cap2017} A. \v{C}ap, {\it On canonical Cartan connections associated to filtered $G$-structures}, arXiv:1707.05627 (2017).
 
 \bibitem{CS2009} A. \v{C}ap, J. Slov\'ak, {\it Parabolic Geometries I: Background and General Theory}, Mathematical Surveys and Monographs, vol. 154, American Mathematical Society, 2009.
 
 \bibitem{Car1910} \'E.\ Cartan, {\it Les syst\`emes de Pfaff \`a cinq variables et les \'equations aux d\'eriv\'ees partielles du second ordre}, Ann. Sci. \'Ec. Norm. Sup\'er. (3) {\bf 27} (1910), 109--192.
 
 \bibitem{DG2013} B. Doubrov, A. Govorov, {\it A new example of a generic 2-distribution on a 5-manifold with large symmetry algebra}, arXiv:1305.7297 (2013).
 
 \bibitem{DK2014} B. Doubrov, B. Kruglikov, {\it On the models of submaximal symmetric rank 2 distributions in 5D}, Diff. Geom. Appl. {\bf 35} (2014), 314--322.

 \bibitem{GZ2021} J. Gregorovic, L. Zalabova, {\it First BGG operators via homogeneous examples}, arXiv:2107.10668 (2021).
  
 \bibitem{Hammerl2007} M. Hammerl, {\it Homogeneous Cartan geometries}, Arch. Math. (Brno) {\bf 43}, no. 5 (2007) 431--442.
 
 \bibitem{HS2009} M. Hammerl, K. Sagerschnig, {\it Conformal structures associated to generic rank 2 distributions on 5-manifolds -- characterization and Killing-field decomposition}, SIGMA {\bf 5}, 081 (2009), 29 pages.
 
 \bibitem{KN1964} S. Kobayashi, T. Nagano, {\it On projective connections}, J. Math. Mech. {\bf 13}, no. 2 (1964), pp. 215--235.
 
 \bibitem{Kos1961} B.\ Kostant, {\it Lie algebra cohomology and the generalized Borel--Weil theorem}, Ann. of Math. (2) {\bf 74} (1961), 329--387.

 \bibitem{KT2017} B. Kruglikov, D. The, {\it The gap phenomenon in parabolic geometries}, J. Reine Angew. Math. {\bf 723} (2017), 153--215.

 \bibitem{McK2022} B. McKay, {\it Characteristic forms of complex Cartan geometries II}, arXiv:2201.05038 (2022).
 
 \bibitem{Nur2005} P. Nurowski, {\it Differential equations and conformal structures}, J. Geom. Phys. {\bf 55} (2005), 19--49.
 
 \bibitem{Sag2008} K. Sagerschnig, {\it Weyl structures for generic rank two distributions in
dimension five}, Ph.D. dissertation, University of Vienna (2008).

 \bibitem{SW2017a} K. Sagerschnig, T. Willse, {\it The geometry of almost Einstein $(2,3,5)$ distributions}, SIGMA (2017), 004, 56 pp.

 \bibitem{SW2017b} K. Sagerschnig, T. Willse, {\it The almost Einstein operator for $(2,3,5)$ distributions}, Arch. Math. (Brno) {\bf 53} (2017), 347--370.
 
 \bibitem{Str2009} F. Strazzullo, {\it Symmetry Analysis of General Rank-3 Pfaffian Systems in Five Variables}, Ph.D. dissertation, Utah State University (2009), \href{https://digitalcommons.usu.edu/etd/449}{https://digitalcommons.usu.edu/etd/449}
 
 \bibitem{The2018} D. The, {\it Exceptionally simple PDE}, Diff. Geom. Appl. {\bf 6} (2018), 13--41.
 
 \bibitem{The2021} D. The, {\it On uniqueness of submaximally symmetric parabolic geometries}, arXiv:2107.10500 (2021).

 \bibitem{Willse2014} T. Willse, {\it Highly symmetric 2-plane fields on 5-manifolds and 5-dimensional Heisenberg group holonomy}, Diff. Geom. Appl. {\bf 33}, supp. (2014), 81--111.

 \bibitem{Willse2018} T. Willse, {\it Cartan’s incomplete classification and an explicit ambient metric of holonomy $G_2^*$}, Eur. J. Math. {\bf 4} (2018), 622--638.
 
 \bibitem{Willse2019} T. Willse, {\it Homogeneous real $(2,3,5)$ distributions with isotropy}, SIGMA {\bf 15} (2019), 008, 28 pages.
  
 \bibitem{Zel2006} I. Zelenko, {\it On variational approach to differential invariants of rank two distributions}, Diff. Geom. Appl. {\bf 24} (2006), 235--259.
 
 \bibitem{Zel2009} I. Zelenko, {\it On Tanaka's Prolongation Procedure for Filtered Structures of Constant Type}, SIGMA {\bf 5} (2009), 094, 21 pages.
 
 \bibitem{Yam1999} K. Yamaguchi, {\it $G_2$-geometry of overdetermined systems of second order}, in: Analysis and Geometry in Several Complex Variables, Katata, 1997, in: Trends Math., Birkh\"auser-Verlag, Boston, 1999, pp.289--314.
 
 \bibitem{Zhi2021} M. Zhitomirskii, Normal forms and symmetries for $(2,3,5)$ and $(3,5)$ distributions: 111 years after E. Cartan's 5 variables paper, \href{http://old.cft.edu.pl/grieg/Slides/Zhitomirskii-Lecture123.pdf}{GRIEG seminar slides}, March-April 2021.

 \end{thebibliography}
 \end{document}